%% file: ms.tex
\documentclass{amsart}

\usepackage[T1]{fontenc}
\usepackage[utf8]{inputenc}

\usepackage{fullpage}
\usepackage{indentfirst}

\usepackage[foot]{amsaddr}
\usepackage{amsfonts}
\usepackage{amsmath}
\usepackage{amssymb}
\usepackage{amsthm}
\usepackage{braket}
\usepackage[style=alphabetic, backend=biber, maxbibnames=99]{biblatex}
\usepackage{enumitem}
\usepackage[mathscr]{eucal}
\usepackage[pdfusetitle]{hyperref} 
\usepackage{mathtools}
\usepackage{mathdots}
\usepackage{rotating}
\usepackage{standalone}
\usepackage{subcaption}
\usepackage{tabu}
\usepackage{tikz}
\usepackage{tikz-cd}
\usepackage{xcolor}

\usetikzlibrary{arrows}
\usetikzlibrary{calc}
\usetikzlibrary{decorations}
\usetikzlibrary{decorations.markings}
\usetikzlibrary{decorations.pathreplacing}
\usetikzlibrary{decorations.text}
\usetikzlibrary{math}
\usetikzlibrary{positioning}
\usetikzlibrary{scopes}
\usepackage[outline]{contour}

\usepackage[capitalise,noabbrev]{cleveref}

\theoremstyle{plain}
\newtheorem{theorem}{Theorem}[section]

\newtheorem{fact}[theorem]{Fact}
\newtheorem{lemma}[theorem]{Lemma}
\newtheorem{proposition}[theorem]{Proposition}

\theoremstyle{definition}
\newtheorem{definition}[theorem]{Definition}
\newtheorem{example}[theorem]{Example}

\theoremstyle{remark}
\newtheorem{notation}[theorem]{Notation}

\newtheorem{remark}[theorem]{Remark}

\newcommand{\Qdv}{Q^{0\text{--}2}}
\newcommand{\Qdh}{Q^{1\text{--}3}}

\newcommand{\conftoquiv}{\mathcal{M}}

\newcommand{\defined}[1]{\textbf{#1}}

\newcommand{\fundamentalweight}{\omega}

\newcommand{\genminor}[2]{\Delta^{{#1} \fundamentalweight_{#2}}}
\newcommand{\genminorgk}[1]{\Delta^{\gamma_{#1}}}
\newcommand{\genminornumerator}[2]{\Delta_{\oplus}^{{#1} {#2}}}
\newcommand{\genminoryz}[3]{\Delta_{{#1} \fundamentalweight_{#3}, {#2}%
    \fundamentalweight_{#3}}}
\newcommand{\lifttoG}[1]{{\mathop{\overline{#1}}}}
\newcommand{\dlifttoG}[1]{{\mathop{\overline{\overline{#1}}}}}
\newcommand{\minorcoordmap}{Y}
\newcommand{\monomialmap}{m}
\newcommand{\muflipind}{{\mu_{\text{flip}}^{\ast}}}
\newcommand{\murotind}{{\mu_{\text{rot}}^{\ast}}}
\newcommand{\murottwistind}{{\mu_{\text{rotTw}}^{\ast}}}
\newcommand{\partialtorus}{\widetilde{T}^{\mathcal{A}}}
\newcommand{\posroots}{\Delta_{+}}
\newcommand{\pregenminor}[1]{\widetilde{\Delta}^{#1}}
\newcommand{\Q}[1]{\ifstrempty{#1}{Q}{Q_{\text{#1}}}}
\newcommand{\roots}{\Delta}
\newcommand{\seedAtorus}{T^{\mathcal{A}}}
\newcommand{\seedXtorus}{T^{\mathcal{X}}}

\newcommand{\triang}{\mathcal{T}}

\newcommand{\torush}{\chi^{\ast}}
\newcommand{\TtoZ}{p}

\newcommand{\word}[1]{\mathbf{#1}}

\newcommand{\muflip}{{\mu_{\text{flip}}}}
\newcommand{\mufliptwist}{{\mu_{\text{flipTw}}}}
\newcommand{\murot}{{\mu_{\text{rot}}}}
\newcommand{\murottwist}{{\mu_{\text{rotTw}}}}
\newcommand{\murottwistL}{{\mu^{L}_{\text{rotTw}}}}
\newcommand{\murottwistR}{{\mu^{R}_{\text{rotTw}}}}

\newcommand{\muint}[1]{\widetilde{\mu}_{\text{#1}}}

\newcommand{\muflipcore}{\muint{Flipcore}}
\newcommand{\murowperm}{\muint{RS}}
\newcommand{\mucolperm}{\muint{CS}}
\newcommand{\mucol}{\muint{Col}}
\newcommand{\muT}{\muint{T}}
\newcommand{\muP}{\muint{P}}

\newlength\inlmutseqhang
\newcommand{\inlmutseq}[2]{%
  {\small%
    \settowidth\inlmutseqhang{$#1 = \big\{~$~}%
    \hangindent\inlmutseqhang%
    \noindent $#1 = \big\{$ #2 $\big\}$\par}%
}

\DeclareMathOperator{\ad}{ad}

\DeclareMathOperator{\Conf}{Conf}

\DeclareMathOperator{\facemap}{d}
\DeclareMathOperator{\GL}{GL}
\DeclareMathOperator{\Hom}{Hom}
\DeclareMathOperator{\id}{id}
\DeclareMathOperator{\Id}{Id}
\DeclareMathOperator{\Isom}{Isom}
\DeclareMathOperator{\length}{length}
\DeclareMathOperator{\PGL}{PGL}
\DeclareMathOperator{\PSL}{PSL}
\DeclareMathOperator{\rank}{rank}
\DeclareMathOperator{\sgn}{sgn}
\DeclareMathOperator{\SL}{SL}

\DeclareMathOperator{\teich}{\mathcal{T}}

\DeclarePairedDelimiter\abs||
\DeclarePairedDelimiter\gen\langle\rangle
\DeclarePairedDelimiter\ip\langle\rangle

\setlist[itemize,1]{label=\raisebox{0.25ex}{\tiny$\bullet$}}

\author{S. Gilles}
\email{sgilles@sgilles.net}
\title{Fock--Goncharov Coordinates for Semisimple Lie Groups}
\date{2021-03-24}

\begin{document}

  \begin{abstract}
    Fock and Goncharov \cite{fockgoncharov2006} introduced cluster
    ensembles, providing a framework for coordinates on varieties of
    surface representations into Lie groups, as well as a complete
    construction for groups of type $A_n$. Later, Zickert
    \cite{zickert2016}, Le \cite{le2016}, \cite{le2016_2}, and Ip
    \cite{ip2016} described, using differing methods, how to apply this
    framework for other Lie group types. Zickert also showed that this
    framework applies to triangulated $3$-manifolds. We present a
    complete, general construction, based on work of Fomin and
    Zelevinsky. In particular, we complete the picture for the remaining
    cases: Lie groups of types $F_4$, $E_6$, $E_7$, and $E_8$.
  \end{abstract}

  \maketitle

  \tableofcontents

  \section*{Acknowledgements}

  I owe thanks to far too many people to list them all by name. My
  professors have shared their time and expertise with me, my
  office-mates have put up with me for multiple semesters, and my
  employers have given me flexibility to complete my research. I am
  grateful to them all.

  My advisor, Professor Zickert, has provided me with expert guidance
  and mentoring over the years. He has helped me to achieve the goal for
  which I began my study, and words cannot express my appreciation.

  Finally, I would like to thank S. Kane, who taught me the meaning of
  work, and whom I am unable to repay.

  \emph{Soli Deo Gloria.}

  \section{Introduction}

  The program of Fock--Goncharov, starting with
  \cite{fockgoncharov2006}, aims to describe representation spaces of
  hyperbolic surfaces into Lie groups by moduli spaces defined by
  polynomial equations. These moduli spaces carry a positive structure
  (see \cref{def:positive-structure}), and in the case of the Lie group
  $\PSL_2(\mathbb{R})$, the associated positive spaces can be identified
  with Teichmüller space and decorated Teichmüller space. For more
  complicated Lie groups, these positive spaces give (decorated) higher
  Teichmüller spaces.

  Our result regards an additional structure, a cluster ensemble
  structure (see \cref{sec:cluster-ensembles}), which allows
  manipulating these moduli spaces efficiently via quivers.

  \begin{figure}[h]
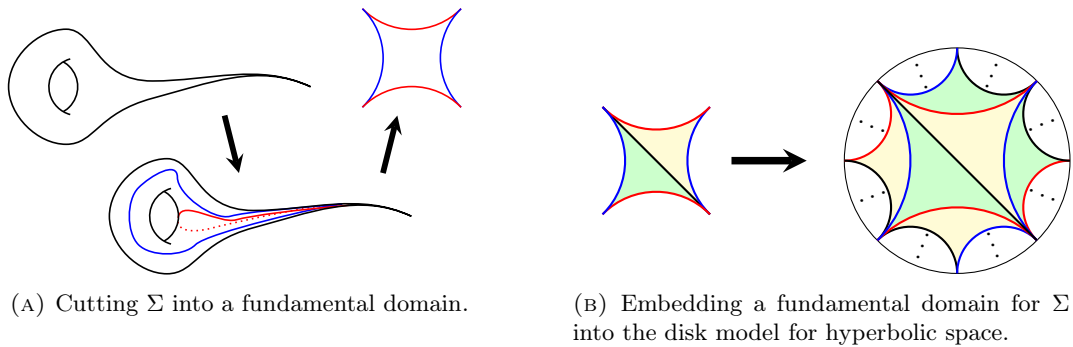

    \centering
    \begin{subfigure}[t]{0.40\textwidth}
      \centering
      \includestandalone[mode=image|tex,height=1.5in]{fig/sigma-1}
      \caption{Cutting $\Sigma$ into a fundamental domain.
        \label{fig:sigma-1}
      }
    \end{subfigure}
    ~ \hspace{0.05\textwidth} ~
    \begin{subfigure}[t]{0.40\textwidth}
      \centering \includestandalone[mode=image|tex]{fig/sigma-2}

      \caption{Embedding a fundamental domain for $\Sigma$ into the disk
        model for hyperbolic space.
        \label{fig:sigma-2}
      }
    \end{subfigure}
    \caption{Putting a hyperbolic structure on $\Sigma$.
      \label{fig:rho-by-fundamental-domain}
    }
  \end{figure}

  \subsection{Classical Teichmüller space}

  We start by describing the Fock--Goncharov program for the case of Lie
  groups of type $A_1$. Let $\Sigma$ be a surface that admits ideal
  triangulation, for example the once-punctured torus $\Sigma =
  \Sigma_{1,1}$. The Fock--Goncharov moduli spaces will describe
  (decorated) Teichmüller space, so fix a hyperbolic structure on
  $\Sigma$. This can be described by an embedding of $\Sigma$'s
  fundamental domain in $\mathbb{H}^2$, as in
  \cref{fig:rho-by-fundamental-domain}.

  We now describe the two moduli spaces with seven key points.

  \begin{quote}
    \emph{1. The structure is determined by ideal points.}
  \end{quote}

  This follows from our ideal triangulation: all vertices are on
  $\partial \mathbb{H}^2$, and all edges between them are unique
  geodesics. See \cref{fig:sigma-3}.

  \begin{figure}[h]
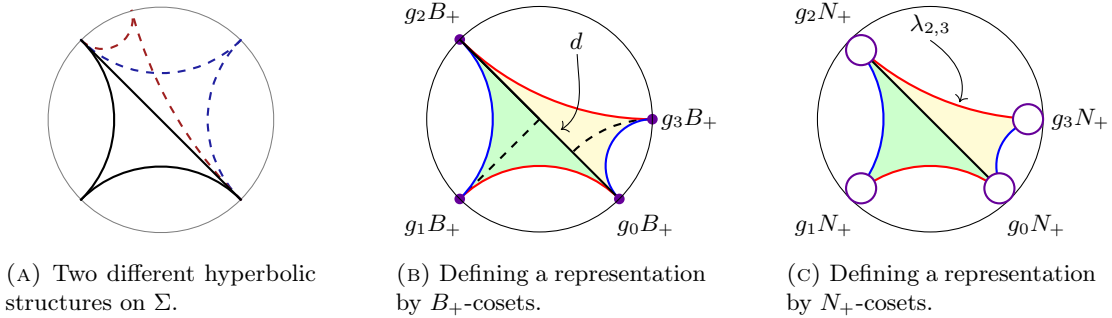

    \centering
    \begin{subfigure}[t]{0.25\textwidth}
      \centering \includestandalone[mode=image|tex]{fig/sigma-3}

      \caption{Two different hyperbolic structures on $\Sigma$.
        \label{fig:sigma-3}
      }
    \end{subfigure}
    ~ \hspace{0.05\textwidth} ~
    \begin{subfigure}[t]{0.25\textwidth}
      \centering \includestandalone[mode=image|tex]{fig/sigma-6}

      \caption{Defining a representation by $B_{+}$-cosets.
        \label{fig:sigma-6}
      }
    \end{subfigure}
    ~ \hspace{0.05\textwidth} ~
    \begin{subfigure}[t]{0.25\textwidth}
      \centering \includestandalone[mode=image|tex]{fig/sigma-5}

      \caption{Defining a representation by $N_{+}$-cosets.
        \label{fig:sigma-5}
      }
    \end{subfigure}
    \caption{Representations of $\Sigma$ by cosets.
      \label{fig:rho-by-flags}
    }
  \end{figure}

  \begin{quote}
    \emph{2. We can identify ideal points with cosets of isometries in
      $\Isom^{+}(\mathbb{H}^2)$ that fix them.}
  \end{quote}

  The specific type of isometry determines which moduli space we create.
  \begin{itemize}
  \item
    We can consider $B_{+}$, a maximal borel subgroup, which fixes a
    point on $\partial \mathbb{H}^2$. Thus different cosets $g B_{+}$
    fix, and can be identified with, different points on $\partial
    \mathbb{H}^2$. See \cref{fig:sigma-6}.
  \item
    Replacing points with horocycles, we can consider $N_{+}$, a maximal
    unipotent subgroup, which fixes a horocycle in $\mathbb{H}^2$. See
    \cref{fig:sigma-5}.
  \end{itemize}

  \begin{quote}
    \emph{3. Choosing an ordered triangulation attaches a non-degenerate
      ordered triple of cosets to each triangle.}
  \end{quote}

  The non-degeneracy condition requires that the cosets be distinct, so
  that the triangle has three well-defined edges. Describing how the
  coordinates of Fock--Goncharov's moduli spaces change under alternate
  choices for the triangulation and the ordering is one of our primary
  concerns.

  \begin{quote}
    \emph{4. Coordinates can be assigned to each triple (or pair of
      triples) of cosets.}
  \end{quote}

  \begin{itemize}
  \item
    If we consider $B_{+}$-cosets, then the coordinates we attach to
    each edge resemble Thurston's shear coordinates (as described in
    e.g.\ \cite{bonahon1996}) along that identified edge, as the
    distance $d$ in \cref{fig:sigma-6}. Some of these coordinates need
    two triples to define. We will not focus on these in this
    introduction.
  \item
    If we consider $N_{+}$-cosets, then the portion of each geodesic
    between horocycles has finite length. These are equivalent to
    Penner's $\lambda$-lengths in his parametrization of decorated
    Teichmüller space of \cite{penner1987}. We associate to each
    triangle's edge the $\lambda$-length of its truncated geodesic as a
    coordinate. See \cref{fig:sigma-5}.
  \end{itemize}

  \begin{figure}[h]
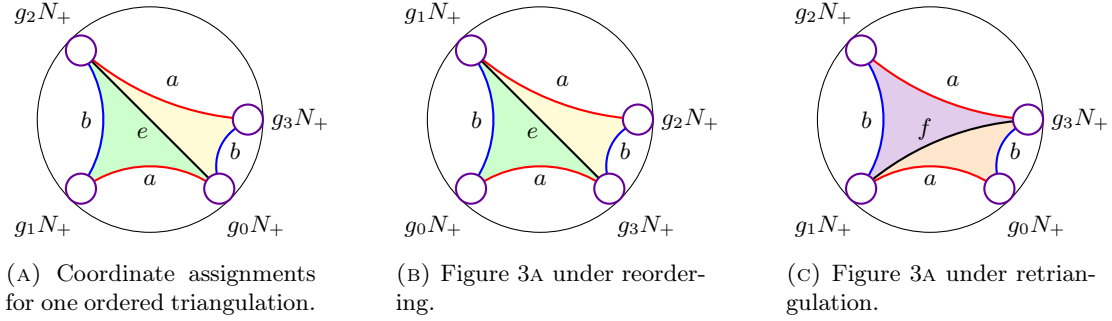

    \centering
    \begin{subfigure}[t]{0.25\textwidth}
      \centering \includestandalone[mode=image|tex]{fig/sigma-7}

      \caption{Coordinate assignments for one ordered triangulation.
        \label{fig:sigma-7}
      }
    \end{subfigure}
    ~ \hspace{0.05\textwidth} ~
    \begin{subfigure}[t]{0.25\textwidth}
      \centering \includestandalone[mode=image|tex]{fig/sigma-8}

      \caption{\cref{fig:sigma-7} under reordering.
        \label{fig:sigma-8}
      }
    \end{subfigure}
    ~ \hspace{0.05\textwidth} ~
    \begin{subfigure}[t]{0.25\textwidth}
      \centering \includestandalone[mode=image|tex]{fig/sigma-9}

      \caption{\cref{fig:sigma-7} under retriangulation.
        \label{fig:sigma-9}
      }
    \end{subfigure}
    \caption{Coordinates by $N_{+}$-cosets under reordering and
      retriangulation.
      \label{fig:rotation-and-flip-on-A-coords}
    }
  \end{figure}

  In our example, we label the edges $a$, $b$, and $e$ as in
  \cref{fig:sigma-7} (identified edges necessarily have the same
  $\lambda$-length). Using the $(0 \to 1,1 \to 2,2 \to 0)$-ordering on a
  triangle's edges, the coordinates we obtain are $(a,b,e)$ on the left
  and $(e,a,b)$ on the right.

  \begin{quote}
    \emph{5. We must describe how the coordinates change under
      (oriented) reordering and retriangulation.}
  \end{quote}

  In this $A_1$ case, reordering has a trivial effect on the
  coordinates. If we rename $g_0$, $g_1$, $g_2$, $g_3$, but do not
  change their values, we merely rearrange the coordinates following the
  new cyclic ordering on the triangle's edges. The coordinates defined
  by \cref{fig:sigma-8} are $(b,e,a)$ on the left and $(a,b,e)$ on the
  right.

  The retriangulation has a more interesting effect, as shown in
  \cref{fig:sigma-9}. The coordinates we obtain are $(b,a,f)$ for the
  top and $(a,f,b)$ for the bottom, but we need to describe the
  relationship of the new $\lambda$-length $f$ to the $\lambda$-lengths
  $a$, $b$, $e$. This is given by the classic Ptolemy's Theorem
  regarding lengths and diagonals of circumscribed quadrilaterals. In
  Euclidean space this is given by \cref{fig:ptolemy-euclidean}, and in
  hyperbolic space by \cref{fig:ptolemy-hyperbolic}. For our example,
  taking identifications into account, \[ ef = a^2 + b^2. \]

  \begin{figure}[h]
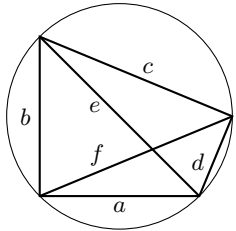
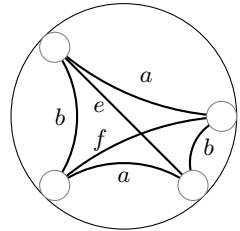

    \centering
    \begin{subfigure}[t]{0.40\textwidth}
      \centering
      \includestandalone[mode=image|tex]{fig/ptolemy-euclidean}
      \caption{Ptolemy's Theorem for lengths in Euclidean space.
        \label{fig:ptolemy-euclidean}
      }
    \end{subfigure}
    ~ \hspace{0.05\textwidth} ~
    \begin{subfigure}[t]{0.40\textwidth}
      \centering
      \includestandalone[mode=image|tex]{fig/ptolemy-hyperbolic}

      \caption{Ptolemy's Theorem for $\lambda$-lengths in hyperbolic
        space.
        \label{fig:ptolemy-hyperbolic}
      }
    \end{subfigure}
    \caption{Ptolemy's theorem: $ac + bd = ef$}
  \end{figure}

  \begin{quote}
    \emph{6. Hyperbolic structures are those for which all coordinates
      are positive.}
  \end{quote}

  The coordinates $a, b, e, f$ with the relation $ef = a^2 + b^2$
  describe an algebraic variety, but not every point corresponds to a
  hyperbolic structure. Those which do are exactly the points with
  positive coordinates, and thus positive geodesic lengths. This
  positivity need only be checked for one collection of coordinates: if
  $a$, $b$, and $e$ are all positive, then $f$ will be as well.

  \begin{quote}
    \emph{7. The representation $\rho : \pi_1(\Sigma) \to G$ can be
      reconstructed.}
  \end{quote}

  Via the coordinates we have mentioned, the holonomy $\rho(\gamma)$
  corresponding to this hyperbolic structure for $\gamma \in
  \pi_1(\Sigma)$ can be computed. See
  \cref{sec:coords-to-representations} for greater detail.

  The moduli space we have described using $N_{+}$-cosets is
  $\mathcal{A}_{\SL_2(\mathbb{R}),\Sigma}$, and points with positive
  coordinates correspond to points in Penner's decorated Teichmüller
  space for $\Sigma$. The moduli space derived from considering
  $B_{+}$-cosets is $\mathcal{X}_{\PGL_2(\mathbb{R}),\Sigma}$, and
  positive points correspond to points in Teichmüller space for
  $\Sigma$.

  \subsection{Cluster ensemble structures}

  In the above program, the retriangulation identity was given by
  Ptolemy's Theorem. Another structure that encodes the identity is a
  \defined{cluster ensemble}. We defer a detailed description to
  \cref{sec:cluster-ensembles}, but for now we only need that a cluster
  ensemble consists of a \defined{quiver} (a directed graph with a
  skew-symmetrizable adjacency matrix) and a pair of coordinate
  structures ($\mathcal{A}$- and $\mathcal{X}$-coordinates), and that
  the shape of the quiver dictates how the coordinates change under
  \defined{mutation} (a certain local alteration of the quiver).

  The quiver of the cluster ensemble which realizes $ac + bd = ef$ for
  the $\mathcal{A}$-coordinates is given (along with the effect of the
  only relevant mutation) in \cref{fig:QA1-and-flip}. The quivers are
  inscribed in two quadrilaterals with different diagonals.

  \begin{figure}[h]
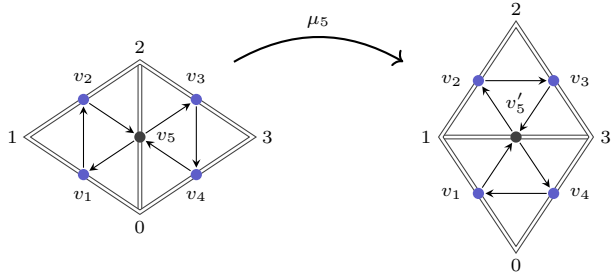

    \centering \includestandalone[mode=image|tex]{fig/QA1-and-flip}

    \caption{Quiver for type $A_1$. The $\mathcal{A}$-coordinate at
      $v_k$ is $a_k$. The relation on $\mathcal{A}$-coordinates for
      mutating at $v_5$ is $a_1 a_3 + a_2 a_4 = a_5 a_5'$.
      \label{fig:QA1-and-flip}
    }
  \end{figure}

  For other surfaces $\Sigma$, larger collections of triangles can be
  glued together. Retriangulations (and reorderings) can be computed as
  sequences of mutations, and the coordinate changes are given by
  collections of low-degree polynomial relations.

  \subsection{Higher Teichmüller spaces}

  We view higher Teichmüller spaces as collections of representations \[
    \Set{ \rho : \pi_1(\Sigma) \to G \text{ discrete and faithful}} / G,
  \] where the standard Teichmüller space case is given by $G =
  \Isom^{+}(\mathbb{H}^2) = \PSL_2(\mathbb{R})$. In order to apply the
  Fock--Goncharov program, we need the following conditions.
  \begin{itemize}
  \item
    The surface $\Sigma$ must admit an ideal triangulation, so we demand
    that it be hyperbolic with $n \ge 1$ punctures.
  \item
    The group $G$ must be a split semisimple algebraic group over
    $\mathbb{Q}$; for technical reasons the moduli space
    $\mathcal{A}_{G,\Sigma}$ is defined when $G$ is simply-connected,
    and $\mathcal{X}_{G,\Sigma}$ is defined when $G$ is centerless.
    However, the field underlying $G$ is not critical.
  \end{itemize}

  As another technical point, when $G$ is simply connected we will
  parametrize boundary-unipotent representations, and when $G$ is
  centerless the representations will be boundary-borel (see
  \cref{def:peripheral-stuff}).

  We now replay the program to create moduli spaces
  $\mathcal{A}_{G,\Sigma}$ and $\mathcal{X}_{G,\Sigma}$ for higher rank
  $G$.

  \begin{quote}
    \emph{1. The structure is determined by ideal points.}
  \end{quote}

  By restrictions on $\Sigma$, the representation is still defined by
  ideal points of a fundamental domain.

  \begin{quote}
    \emph{2. We can identify ideal points with cosets of isometries in
      $\Isom^{+}(\mathbb{H}^2)$ that define them}
  \end{quote}

  We replace $\Isom^{+}(\mathbb{H}^2)$ with general $G$. By our
  restrictions on $G$, subgroups $N_{+}$ and $B_{+}$ are still
  available.

  \begin{quote}
    \emph{3. Choosing an ordered triangulation attaches a non-degenerate
      ordered triple of cosets to each triangle.}
  \end{quote}

  Non-degenerate cosets are replaced by \defined{sufficiently generic
    configuration spaces} (see \cref{def:Conf_k_star}). These ensure
  that a well-defined element of $H$, the maximal torus of $G$,
  corresponding to translations along geodesics in the
  $\PSL_2(\mathbb{R})$ case, can be attached to each edge of each
  triangle.

  \begin{quote}
    \emph{4. Coordinates can be assigned to each triple (or pairs of
      triples) of cosets.}
  \end{quote}

  To assign coordinates to configuration spaces, Fock--Goncharov use
  \defined{generalized minors} (see \cref{def:generalized-minor}). In
  the case of $\GL_n$, these correspond to shuffling rows and columns
  according to two permutations, then taking the upper $i \times i$
  minor. In the more general case, two words in the Weyl group of $G$
  are invoked and a coordinate associated to the $i^{\text{th}}$
  fundamental weight is used.

  By work of Lusztig, configuration spaces carry a positive structure
  via \defined{factorization coordinates} (see
  \cref{def:factorization-coords-xi}). By work of Fomin--Zelevinsky,
  generalized minors inherit this positive structure.

  \begin{quote}
    \emph{5. We must describe how the coordinates change under
      (oriented) reordering and retriangulation}
  \end{quote}

  A major result of the Fock--Goncharov program is that reordering of
  triangles (we restrict to orientation-preserving reorderings, which we
  refer to as \defined{rotation}) and retriangulation of quadrilaterals
  (which we refer to as a \defined{flip} of the diagonal) preserve this
  positive structure. All general retriangulations and reorderings can
  be obtained this way.

  \begin{quote}
    \emph{6. Hyperbolic structures are those for which all coordinates
      are positive.}
  \end{quote}

  Instead of hyperbolic structures, we are interested in points in
  higher Teichmüller spaces. Since the general moduli spaces carry
  positive structures, the sets of positive points
  $\mathcal{A}^{+}_{G,\Sigma}$ and $\mathcal{X}^{+}_{G,\Sigma}$ are
  well-defined when $G$ is over $\mathbb{R}$. These do, in fact,
  correspond to higher Teichmüller spaces or decorations of them (see
  \cref{sec:coords-to-representations}).

  \begin{quote}
    \emph{7. The representation $\rho : \pi_1(\Sigma) \to G$ can be
      reconstructed.}
  \end{quote}

  The representation $\rho(\gamma)$ can still be computed by
  $\mathcal{A}$- or $\mathcal{X}$-coordinates (again, see
  \cref{sec:coords-to-representations}). Here our restrictions on $G$
  being simply connected or centerless are necessary.

  So the Fock--Goncharov program holds beyond $\PSL_2(\mathbb{R})$. We
  now address the following question:

  \begin{quote}
    \emph{The positive structure on coordinates is preserved under
      retriangulation and reorientation, but how are those two
      operations realized on coordinates?}
  \end{quote}

  In the $A_1$ case, the cluster ensemble of \cref{fig:QA1-and-flip}
  provides the answer. In \cite{fockgoncharov2006}, Fock--Goncharov
  produced cluster ensembles for all types $A_n$. The quivers appear as
  triangular lattices, see \cref{sec:An-case} for an overview.
  Fock--Goncharov also provided descriptions for coordinate changes
  under the two key operations as sequences of mutations. The
  \defined{rotation} describes reordering of a triangle, and the
  \defined{flip} describes retriangulation of a quadrilateral.

  Fock--Goncharov predicted that $\mathcal{A}_{G,\Sigma}$ and
  $\mathcal{X}_{G,\Sigma}$ would carry cluster ensemble structures as
  well. To realize those structures, we need, we need
  \begin{itemize}
  \item
    A way to assign coordinates of $\Conf_4^{\ast}(G/K)$ to a triangle.
    These coordinates will be encoded in a cluster ensemble. We will
    call the coordinate assignment map $\conftoquiv$. When the group $G$
    is over $\mathbb{R}$, collections of coordinates will be
    $(\mathbb{R}^{\ast})^n$. Seeds with all coordinates positive will be
    the \defined{positive points}.
  \item
    A way to realize (orientable) symmetries of a triangle on those
    coordinates (the \defined{rotation}, see \cref{fig:short-rot}). The
    bottom map will be realized by a quiver mutation which we call
    $\murot$.
  \item
    A way to realize changes of triangulation on those coordinates (the
    \defined{flip}, see \cref{fig:short-flip}). The bottom map will be
    realized by a quiver mutation which we call $\muflip$.
  \item
    The maps $\murot$ and $\muflip$ should preserve the positive
    structure. That is, if a seed has positive coordinates, applying the
    rotation or the flip should not change that. Using generalized
    minors for coordinates is a way to preserve positivity.
  \end{itemize}

  \begin{figure}[h]
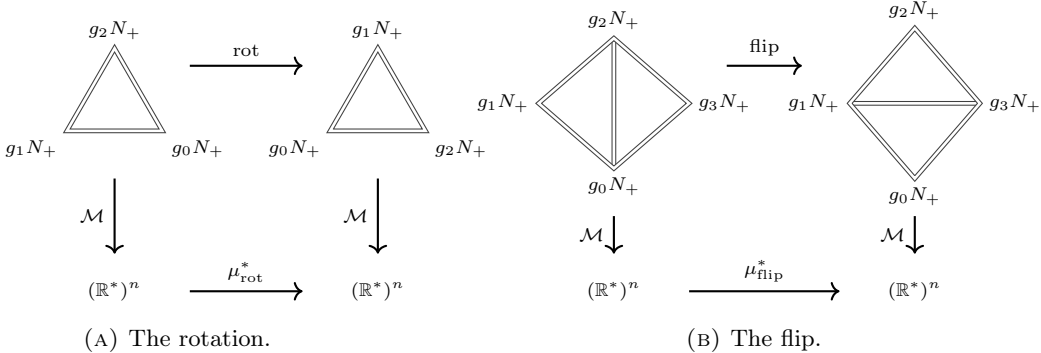

    \centering
    \begin{subfigure}[t]{0.30\textwidth}
      \centering \includestandalone[mode=image|tex]{fig/short-rot}

      \caption{The rotation.
        \label{fig:short-rot}
      }
    \end{subfigure}
    ~ \hspace{0.05\textwidth} ~
    \begin{subfigure}[t]{0.50\textwidth}
      \centering \includestandalone[mode=image|tex]{fig/short-flip}

      \caption{The flip.
        \label{fig:short-flip}
      }
    \end{subfigure}
    \caption{Retriangulation and reordering (for a group over
      $\mathbb{R}$).
      \label{fig:retriangulation-and-reordering}
    }
  \end{figure}

  \subsection{Results}

  Our main result is an explicit construction for cluster ensembles
  described above for all semisimple $G$, including the rotation and the
  flip. In other words, we present a constructive result of the
  following theorem.

  \begin{theorem}
    \label{thm:intro-main-result}
    For a split semisimple simply-connected (or centerless) algebraic
    group $G$ over $\mathbb{Q}$, there exists a quiver $Q$, quiver
    mutations $\murot$ and $\muflip$, and a coordinate map
    $\conftoquiv$.

    The map $\conftoquiv$ associates coordinates of a flag configuration
    to the coordinates on the cluster ensemble for $Q$, the quiver
    mutation $\murot$ describes how the coordinates change under
    rotation of an ordered triangle and $\muflip$ describes how the
    coordinates change under retriangulation.
  \end{theorem}

  We call this collection $(Q, \murot, \muflip, \conftoquiv)$ a
  \defined{Fock--Goncharov coordinate structure}. The construction
  algorithm is presented in \cref{sec:main-result}. First, we use the
  algorithm of \cite[Section~2]{fominzelevinsky1999}, using a specific
  choice of $w_0$ (the longest word in the Weyl group) by
  \cref{fac:w0-construction}. This produces a quiver whose
  $\mathcal{A}$-coordinates generate the coordinate ring for $B_{-} = H
  \times N_{-}$. We call this quiver $Q_0$.

  The quiver $Q_0$ has the same mutable portion that we need for $Q$.
  However, the non-mutable portion isn't complete, so $Q_0$ does not
  have \defined{triangulation-compatible symmetry} (see \cref{def:tcs}).
  We therefore can apply the rotation $\murot$ to $Q_0$, and define $Q$
  to be the smallest quiver with triangulation-compatible symmetry
  containing $Q_0$.

  The quiver $Q$ and the map $\conftoquiv$ allow constructing
  representation varieties easily given a triangulation of a surface: a
  copy of $Q$ is inscribed on each triangle and vertices along edges are
  identified. The maps $\murot$ and $\muflip$ describe how the
  coordinate functions of the variety respond to retriangulation,
  ensuring that the variety itself is independent of triangulation
  choices.

  \begin{remark}
    The existence of cluster ensemble structures was shown
    non-constructively, and by different methods, in
    \cite{goncharovshen2018}.
  \end{remark}

  This theorem can also be used to compute representation varieties for
  $3$-manifolds, following work of Zickert. In this case, quivers are
  drawn on each face of each ideal tetrahedron in the triangulation. The
  map $\muflip$ describes the relation between coordinates on the four
  triangular faces of each tetrahedron. The variety constructed is not,
  however, independent of triangulation. If the triangulation is
  sufficiently fine, however, the variety detects all representations.

  \subsection{Historical context}

  \label{sec:historical-context}

  Fock and Goncharov introduced cluster ensembles in
  \cite{fockgoncharov2006} and \cite{fockgoncharov2003}, employing the
  cluster algebra work of Fomin--Zelevinsky \cite{fominzelevinsky2002},
  \cite{fominzelevinsky2003}, \cite{fominzelevinsky2005} as well as
  Lusztig's study of positivity \cite{lusztig1994}, \cite{lusztig1997}.
  The application was laid out for a split semisimple simply-connected
  Lie group over $\mathbb{R}$, but explicit constructions were only
  known for type $A_n$. Since then, there has been ongoing work.

  \begin{itemize}
  \item
    In \cite{zickert2016}, Zickert produced examples of cluster
    ensembles and mutations for types $A_2$, $B_2$, $C_2$, and $G_2$,
    directly employing Lusztig's positive maps to explicitly construct
    varieties. Our work addresses Conjecture~2.7, which predicts the
    existence of quivers, rotation and flip mutations, and coordinate
    maps for semisimple $G$.

    This work also expanded the use of Fock--Goncharov coordinates to
    representations of ideally-triangulated hyperbolic $3$-manifolds,
    building on \cite{gtz2015}.
  \item
    With \cite{le2016} and \cite{le2016_2}, Le described constructions
    for quivers of types $A_n$, $B_n$, $C_n$, and $D_n$ using tensor
    invariants and webs. Our work addresses Conjecture~3.12 (that the
    cluster algebra for $\Conf_m^{\ast}(G/N_{+})$ is invariant under
    retriangulation and reordering) for the orientation-preserving case,
    for our choice of presentation of $w_0$.
  \item
    Using representations of quivers, Fei constructed quivers and
    mutations for many types in \cite{fei2016}, though associated to a
    slightly different flag variety.
  \item
    In \cite{ip2016}, Ip produced similar ``basic quivers'', though not
    presenting mutation sequences in the general case.
  \item
    Goncharov--Shen continued work on invariants of cluster ensembles
    with \cite{goncharovshen2018}, providing existence proofs.
  \end{itemize}

  There were several obstacles to extending these results to our proof
  of \cref{thm:intro-main-result}. First, the cluster ensembles for type
  $A_n$ have trivial triangular symmetry. However, direct dimension
  counting of the coordinate ring shows that this is not possible for
  general $G$.

  Second, the algorithm of \cite[Section~2]{fominzelevinsky2005} to
  produce a quiver carrying the coordinate structure of $B_{-}$ (which
  is a significant step) is well-known. However, this algorithm depends
  on a particular choice of presentation for $w_0$ in the Weyl group of
  $G$. It is not immediately obvious how to choose $w_0$ for each group,
  or if this choice should matter.

  Third, the same non-triviality of triangular symmetry for $Q$ extends
  to non-triviality of the coordinate assignment map $\conftoquiv$. In
  other words, fully three sides of each square in
  \cref{fig:retriangulation-and-reordering} are trivial in the $A_n$
  case, and therefore give few clues as to the general case.

  Finally, the action of $w_0$ on a simple root $\alpha_i$ is not
  necessarily $w_0(\alpha_i) = -\alpha_i$. This creates various
  complications.

  The following observations are the key to our result.

  \begin{itemize}
  \item
    $A_n$ is the exception, not the base case. Specifically, for all $G$
    except type $A_{2n}$, there exists a very regular presentation of
    $w_0$ in terms of Coxeter elements.
  \item
    This regularity causes the quiver with the coordinate ring of
    $B_{-}$ to be be laid out so that
    \cite[Theorem~1.17]{fominzelevinsky1999} and
    \cite[Theorem~1.5]{yangzelevinsky2008} apply, and these identities
    hint at certain sequences of mutations.
  \end{itemize}

  \section{Ingredients}

  \label{sec:ingredients}

  Here we review some key components of our construction. These include
  generalized minors (see \cref{sec:generalized-minors}), quivers (see
  \cref{sec:quivers}), and cluster ensembles (see
  \cref{sec:cluster-ensembles}).

  \subsection{Root spaces and Weyl groups}

  \label{sec:lie-groups}

  We begin with some fundamentals of Lie groups, referring to e.g.
  \cite[Chapter~II]{knapp1996} or \cite{bourbaki2002_46} for more
  detail. We use the language of Lie groups over $\mathbb{C}$, with
  straightforward generalization to algebraic groups over other fields.

  \begin{definition}
    For $\mathfrak{g}$ a semisimple Lie algebra over $\mathbb{C}$, fix a
    Cartan subalgebra $\mathfrak{h}$. A \defined{root} is some
    simultaneous eigenvalue $\alpha \in \mathfrak{h}^{\ast}$ of all
    $\ad_H : X \mapsto [H, X]$ for $H \in \mathfrak{h}$. That is, there
    is some $X \in \mathfrak{h}$, and for all $H$ we have $[H, X] =
    \alpha(H) X$. $X$ is the simultaneous eigenvector, and $\alpha$ is
    the simultaneous eigenvalue, the root. $\roots$ is the set of all
    roots, the \defined{root system}.

    Every root $\alpha \in \mathfrak{h}^{\ast}$ corresponds to some
    $H_{\alpha} \in \mathfrak{h}$ such that for all $H \in
    \mathfrak{h}$, $B(H, H_{\alpha}) = \alpha(H)$ (with $B$ the Killing
    form). Let $\mathfrak{h}_0$ be the $\mathbb{R}$-linear span of all
    $H_{\alpha}$, and $\mathfrak{h}_0^{\ast} \subset
    \mathfrak{h}^{\ast}$ the dual.

    There is also a corresponding \defined{weight space}
    $\mathfrak{g}_{\alpha}$, defined as \[ \mathfrak{g}_{\alpha} =
      \set{X \in \mathfrak{g} : \text{for all $H \in \mathfrak{h}$,
          \quad $[H, X] = \alpha(H)X$}}. \] With this notation,
    $\mathfrak{h} = \mathfrak{g}_0$, and $\mathfrak{g} = \mathfrak{h}
    \oplus \bigoplus_{\alpha \in \roots} \mathfrak{g}_{\alpha}$.
  \end{definition}

  \begin{definition}
    For a root system $\roots$ relative to a Cartan subalgebra
    $\mathfrak{h}$ of a Lie algebra $\mathfrak{g}$ over $\mathbb{C}$,
    the \defined{Weyl group} is the subgroup of
    $\GL(\mathfrak{h}_0^{\ast})$ generated by (where $\ip{\cdot, \cdot}$
    is the usual inner product, viewing $\mathfrak{g}$ as a complex
    vector space) \[ \gen*{{s_{\alpha} : \varphi \mapsto \varphi -
          2\frac{\ip{\alpha, \varphi}}{\ip{\alpha, \alpha}}\alpha} :
        \text{$\alpha$ a root for $\roots$}.} \] Each $s_{\alpha}$ may
    be thought of as a reflection through the hyperplane perpendicular
    to $\alpha$.

    A sequence $(i_1, i_2, \dotsc, i_m)$ such that $w = s_{\alpha_{i_1}}
    s_{\alpha_{i_2}} \dotsb s_{\alpha_{i_r}}$ is a
    \defined{presentation} for $w \in W$. The elements
    $\set{s_{\alpha_i}}$ follow the braid relations, so presentations
    are often not unique. If a presentation for $w$ has the shortest
    possible length, it is \defined{reduced}.
  \end{definition}

  \begin{definition}
    For a root system $\roots$ and $\mathfrak{h}_0^{\ast}$ as above,
    arbitrarily choose some maximal subset $\posroots$ closed under
    addition and scalar multiplication by positive reals. Such a choice
    gives \defined{positive roots} for $\roots$.

    A positive root $\alpha \in h^{\ast}$ is \defined{simple} if it is
    positive and cannot be expressed as a sum of other positive roots
    with positive coefficients. The span $\gen{s_{\alpha} :
      \text{$\alpha$ a simple root}}$ generates the Weyl group.
  \end{definition}

  \begin{definition}
    Each simple root $\alpha_{i}$ has an associated \defined{fundamental
      weight} $\fundamentalweight_{i}$ in $\mathfrak{h}_0^{\ast}$,
    defined by \[ 2 \frac{\ip{\fundamentalweight_i,
          \alpha_j}}{\ip{\alpha, \alpha}} = \delta_{ij}. \]
  \end{definition}

  \begin{fact}
    The Weyl group is finite. In particular, there is a \defined{longest
      word} $w_0$. That is, $w_0$ is an element of $W$ such that for any
    $s_{\alpha}$, the word $w_0 s_{\alpha}$ has a shorter presentation
    than the shortest presentation of $w_0$.

    The element $w_0$ induces an action $\alpha \mapsto -w_0(\alpha)$ on
    the simple roots.
  \end{fact}

  \begin{definition}
    \label{def:sigma_G}
    Let $\sigma_G$ be the permutation such that $-w_0(\alpha_i) =
    \alpha_{\sigma_G(i)}$. This permutation is of order $1$ or $2$, and
    is often trivial. Where convenient, we will label $\sigma_G(i)$ as
    $i^{\ast}$.

    In an abuse of notation, we will view $\sigma_G$ as acting on $H$ as
    the unique automorphism defined by (looking ahead, $\torush_{i}$ is
    from \cref{def:root-space-decomposition}) \[
      \sigma_G(\torush_{i}(t)) = \torush_{\sigma_G(i)}(t). \]
  \end{definition}

  \begin{remark}
    \label{rem:sigma_G-action-on-Dynkin-diagram}
    The permutation $\sigma_G$ may be realized as a graph automorphism
    of the Dynkin diagram associated to $G$. For example, for $G$ of
    type $D_5$, the involution $\sigma_G$ acts on the simple roots as
    shown:

    {\centering
      \includestandalone[mode=image|tex]{fig/sigmaG-action-on-D5}

    }
  \end{remark}

  \subsection{Unipotent subgroups}

  \begin{definition}
    \label{def:root-space-decomposition}
    For a Lie group $G$ (the Lie algebra of which is $\mathfrak{g}$), a
    Cartan subalgebra $\mathfrak{h}$ and a root system $\roots$ with a
    choice of positive roots $\posroots$, there is a \defined{root space
      decomposition} defining $\mathfrak{n}_{\pm}$ by \[ \mathfrak{g} =
      \overbrace{\bigoplus_{\alpha \in \posroots}
        \mathfrak{g}_{-\alpha}}^{\mathfrak{n}_{-}} {}\oplus{}
      \mathfrak{h} \oplus \overbrace{\bigoplus_{\alpha \in \posroots}
        \mathfrak{g}_{\alpha}}^{\mathfrak{n}_{+}}. \]

    The Lie subgroups of $G$ with these Lie algebras are, respectively,
    $N_{-}$, $H$, and $N_{+}$. The $N_{\pm}$ subgroups are
    \defined{maximal unipotent} subgroups of $G$. $H$ is a
    \defined{maximal torus}. We also have the \defined{borel subgroups}
    $B_{\pm} = H N_{\pm}$.

    We also fix standard generators $e_i \in \mathfrak{g}_{\alpha_i}$,
    $f_i \in \mathfrak{g}_{- \alpha_i}$, and $h_i \in \mathfrak{h}$
    (with $i \in \set{1, 2, \dotsc, \rank G}$, so that we may write \[
      x_i(t) = \exp(t e_i) \in N_{+}, \qquad y_i(t) = \exp(t f_i) \in
      N_{-}, \qquad \torush_{i}(t) = \exp(t h_i) \in H. \]
  \end{definition}

  \subsection{Coxeter elements}

  \begin{definition}
    For a root system generated by simple roots $\roots =
    \gen{\set{\alpha_1, \alpha_2, \dotsc, \alpha_r}}$, any element $c =
    s_{\alpha_1} s_{\alpha_2} \dotsb s_{\alpha_r}$ is a \defined{Coxeter
      element}. The ordering of the roots is irrelevant: any such
    product of all simple roots is a Coxeter element. All Coxeter
    elements have the same order, which is the \defined{Coxeter number},
    denoted $h$.
  \end{definition}

  \begin{fact}
    \label{fac:w0-construction}
    For a Weyl group $W$ with Coxeter element $c$ and even Coxeter
    number $h$, $c^{h/2} = w_0$. That is, if $c = s_{\alpha_1}
    s_{\alpha_2} \dotsb s_{\alpha_r}$, a presentation for $w_0$ is \[
      \word{i} = \set{\overbrace{1, 2, \dotsc, r}^{1}, \quad
        \overbrace{1, 2, \dotsc, r}^{2}, \quad \dotsc, \quad
        \overbrace{1, 2, \dotsc, r}^{h/2}}. \]
  \end{fact}

  This is standard, see
  e.g.~\cite[Chapter~VI,~\S~1.11]{bourbaki2002_46}.

  \subsection{Generalized minors}

  \label{sec:generalized-minors}

  We review generalized minors. These are the extension of flag minors
  on $\GL_n$ of \cite{fominzelevinsky1996},
  \cite{berensteinzelevinsky1997} to arbitrary semisimple algebraic
  groups. This generalization is described fully in
  \cite{fominzelevinsky1999}, to which we refer for more details.

  \begin{definition}
    \label{def:gaussian-decomposition}
    For $G$ a Lie group over $\mathbb{C}$ admiting $H$ and $N_{\pm}$,
    $G_0 = N_{-} H N_{+}$. For $x \in G_0$, the \defined{Gaussian
      decomposition} of $x$ into these components is $x = [x]_{-}
    [x]_{0} [x]_{+}$, with $[x]_{0} \in H$ and $[x]_{\pm} \in N_{\pm}$.
    This decomposition is necessarily unique.
  \end{definition}

  \begin{definition}
    For $G$ a Lie group over $\mathbb{C}$ admiting $H$ as above, we may
    identify the Weyl group $W$ with $N_G(H)/H$, the quotient of the
    normalizer of $H$ in $G$ by $H$. We denote \[ \lifttoG{s_i} =
      x_i(-1) y_i(1) x_i(-1) \qquad \dlifttoG{s_i} = x_i(1) y_i(-1)
      x_i(1). \] These satisfy the braid relations, and whenever
    $\length{u v} = \length{u} + \length{v}$, we have \[ \lifttoG{u v} =
      \lifttoG{u} \cdot \lifttoG{v}, \qquad \dlifttoG{u v} =
      \dlifttoG{u} \cdot \dlifttoG{v}. \]

  \end{definition}

  \begin{remark}
    \label{rem:w0-switches-Ns}
    If $G$ is semisimple, then the standard choice of bases $\set{e_i}$,
    $\set{h_i}$, and $\set{f_i}$ for $\mathfrak{n}_{-}$, $\mathfrak{h}$,
    and $\mathfrak{n}_{+}$ give rise, via $\exp$, to one-parameter
    subgroups in $N_{-}$, $H$, $N_{+}$ respectively, and the
    identification of $W$ with $N_G(H)/H$ identifies $w_0$ with
    $\lifttoG{w_0}$. Acting by conjugation on $G$, this element switches
    each pair of subgroups $\exp(\mathbb{R}e_i)$, $\exp(\mathbb{R}f_i)$.
    Thus $\lifttoG{w_0}$ switches $N_{-}$ and $N_{+}$.
  \end{remark}

  \begin{remark}
    \label{rem:sG}
    The element $\lifttoG{w_0}^2$ arises frequently. It shall be denoted
    $s_G$, and is in the center of $G$, having order either $1$ or $2$.
  \end{remark}

  \begin{definition}
    \label{def:pre-generalized-minor}
    For $G$ a Lie group over $\mathbb{C}$, and any fundamental weight
    $\fundamentalweight_i$, define \[ \pregenminor{\fundamentalweight_i}
      : H \mapsto \mathbb{C} \qquad \text{by} \qquad
      \pregenminor{\fundamentalweight_i} : h \mapsto
      \exp(\fundamentalweight_i\exp^{-1}(h))). \]
  \end{definition}

  \begin{definition}
    \label{def:generalized-minor}
    For $G$ a Lie group over $\mathbb{C}$, admitting decomposition as
    above, $v, w$ elements of the Weyl group $W$, and
    $\fundamentalweight_i$ a fundamental weight, the
    \defined{generalized minor} for the words and fundamental weight $v
    \fundamentalweight_i, w \fundamentalweight_i$ is the regular
    function $\genminoryz{v}{w}{i} : G \mapsto \mathbb{C}$, defined by
    its restriction on $\dlifttoG{v} G_0 \lifttoG{w^{-1}}$ by \[
      \genminoryz{v}{w}{i} : g \mapsto
      \pregenminor{\fundamentalweight_i} \left( \left[\dlifttoG{v^{-1}}g
      \lifttoG{w}\right]_{0} \right). \]

    In the case $w$ is trivial, as it often is in our construction, we
    denote \[ \genminor{v}{i}(g) = \genminoryz{v}{}{i}(g). \]
  \end{definition}

  \begin{example}
    If $G = \SL(n,\mathbb{C})$ (of type $A_{n-1}$), we may take $H$ as
    diagonal matrices, with $N_{+}$ and $N_{-}$ as strictly upper and
    lower triangular matrices.

    The fundamental weights are $\fundamentalweight_i = (e_1 + e_2 +
    \dotsb + e_i)^{\ast}$, and the function
    $\pregenminor{\fundamentalweight_i}$ corresponds to taking the
    product $h_{11} h_{22} \dotsb h_{ii}$: the $i \times i$ minor of the
    first $i$ rows and columns. The reflections $s_{\alpha_i}$ lift to
    \[ \dlifttoG{s_i} = \Id_{i-1} \boxplus \begin{bmatrix} & 1 \\ -1 &
      \end{bmatrix} \boxplus \Id_{n-i-1} \qquad \lifttoG{s_i} =
      \Id_{i-1} \boxplus \begin{bmatrix} & -1 \\ 1 & \end{bmatrix}
      \boxplus \Id_{n-i-1}. \] (Here $\boxplus$ is diagonal matrix
    concatenation.) Thus $\dlifttoG{s_i}$ acts on $\SL(n, \mathbb{C})$
    on the left by permuting rows, $\lifttoG{s_i}$ acts on the right by
    permuting columns, and the generalized minor for $v
    \fundamentalweight_i,w\fundamentalweight_i$ on $\SL(n, \mathbb{C})$
    is exactly the $i \times i$ minor of the top left rows and columns
    after permutation.
  \end{example}

  \begin{definition}
    \label{def:gamma-k}
    For a fixed presentation $\word{i} = (i_1, i_2, \dotsc, i_m)$ of
    $w_0$, define \[ w_k = s_{i_m} \dotsb s_{i_{k+1}} s_{i_k}, \] and
    define the \defined{$\word{i}$-chamber weights} \[ \gamma_k = w_k
      \fundamentalweight_{i_k}. \] We also count
    $\fundamentalweight_{i_k}$ as $\word{i}$-chamber weights.
  \end{definition}

  \begin{remark}
    The $w_k$ are pairwise distinct and fill the involution set of
    $w_0$, see \cite[Chapter VI, \textsection 1.6]{bourbaki2002_46}.
  \end{remark}

  \begin{remark}
    \label{rem:all-wkomegaj-are-gamma-k-or-trivial}
    Any $\genminor{w_k}{j}$ is some $\genminorgk{k}$ or
    $\genminor{}{j}$. This follows from the fact that
    $\pregenminor{i}([\dlifttoG{s_{j \ne i}} g]_{0}) =
    \pregenminor{i}([g]_{0})$, so for use with generalized minors, any
    $w_k \fundamentalweight_{j}$ may be reduced to $w_{k+1}
    \fundamentalweight_{j}$ whenever $w_k$ is not trivial and $i_k \ne
    j$.
  \end{remark}

  \begin{definition}
    \label{def:non-negative-generalized-minor}
    When restricted to $H$, we may express $\genminor{v}{i}$ in another
    fashion. For $G$ of rank $r$, let $P$ denote the \defined{weight
      lattice}, integral combinations of $\set{\fundamentalweight_i}$.
    For $\beta \in P \cong \mathbb{Z}^r$, and for $A = \prod_{k=1}^{r}
    \torush_k(t_k) \in H$, \cref{def:pre-generalized-minor} is
    equivalent to \[ \pregenminor{\beta}(A) = \exp( \ip{\beta, t} ), \]
    where $\ip{\cdot, \cdot}$ is the normal Euclidean inner product.

    Now for $v \in W$, we may consider $v$ acting on $P \subset
    \mathfrak{h}_0$ by $v \beta = (v \beta^{\ast})^{\ast}$, using the
    action of $W$ on $\mathfrak{h}_{0}^{\ast}$. This allows expanding
    the above definition to \[ \Delta_{v \beta, v \beta}(A) = \exp(\ip{v
        \beta, t}). \]

    Now, define the \defined{non-negative generalized minor},
    $\genminornumerator{v}{\beta} : H \to \mathbb{C}$ as \[
      \genminornumerator{v}{\beta}(A) = \exp(\ip{ \max(v \beta, (0, 0,
        \dotsc, 0)), t}), \] where $\max$ operates componentwise.
  \end{definition}

  The $\oplus$ is intended to suggest tropical addition.

  \begin{remark}
    \label{rem:numerator-calculation}
    For $A \cong (t_1, t_2, \dotsc, t_r)$, the map
    $\genminornumerator{v}{\fundamentalweight_i}(A)$ can be computed by
    computing $\genminoryz{v}{v}{i}((s_1, s_2, \dotsc, s_r))$
    abstractly, taking the numerator of the result, then substituting
    $t_i$ for $s_i$.
  \end{remark}

  \subsection{Factorization coordinates}

  \label{sec:factorization-coordinates}

  We also review factorization coordinates, referring again to
  \cite{fominzelevinsky1999} for details. In brief, they will give
  coordinates on $\Conf_3^{\ast}(G/N_{+})$ (of \cref{def:Conf_k_star})
  which are governed by a word $\word{i}$ in the Weyl group of $G$.

  \begin{definition}
    \label{def:factorization-coords-xi}
    For $G$ a Lie group of rank $r$ over $\mathbb{C}$ admitting $H$,
    $N_{\pm}$, $x_i$, and $y_i$, with Weyl group $W$, $\word{i} = (i_1,
    \dotsc, i_k)$ a reduced word in $W$, and $t \in
    (\mathbb{C}^{\ast})^k$, we define \[ x_{\word{i}}(t) = x_{i_1}(t_1)
      \dotsb x_{i_k}(t_k) \in N_{+}, \qquad y_{\word{i}}(t) =
      y_{i_1}(t_1) \dotsb y_{i_k}(t_k) \in N_{-}. \]

    Let $\word{i} = (i_1, i_2, \dotsc, i_m)$ be a fixed representation
    for $w_0$. Then
    \begin{align*}
      x_{\word{i}}(\set{t_k}_{k=1}^{m}) &= x_{i_1}(t_1) x_{i_2}(t_2) \dotsb x_{i_m}(t_m) \in N_{+} \cap G_0 \lifttoG{w_0} \\
      y_{\word{i}}(\set{t_k}_{k=1}^{m}) &= y_{i_m}(t_m) \dotsb y_{i_2}(t_2) y_{i_1}(t_1) \in N_{-} \cap \lifttoG{w_0} G_0.
    \end{align*}
  \end{definition}

  By \cite[Theorem~1.3]{fominzelevinsky1999} these are isomorphisms.

  \begin{definition}
    For $u \in N_{-} \cap \lifttoG{w_0} G_0$, the \defined{factorization
      coordinates} of $u$ are $(t_1, \dotsc, t_m) \in \mathbb{C}^m$ such
    that $u = x_{\word{i}}(t_1, \dotsc, t_m)$.
  \end{definition}

  We also recall two utility functions from \cite{fominzelevinsky1999}
  and \cite{fockgoncharov2006}.

  \begin{definition}
    \label{def:phi-and-psi}
    The biregular anti-automorphism $\Psi : G \to G$ and the biregular
    automorphism $\Phi : N_{+} \cap G_0 \lifttoG{w_0} \to N_{-} \cap
    \lifttoG{w_0} G_0$ are the unique maps such that, for $t \in
    \mathbb{C}$, $h \in H$, $x \in N_{+}$, and $y \in N_{-}$,
    \begin{align*}
      \Psi(x_i(t)) &= y_i(t) & \Psi(y_i(t)) &= x_i(t) & \Psi(h) &= h
    \end{align*}
    and
    \begin{align*}
      \Phi(x) &= [x \lifttoG{w_0}]_{-}, & \Phi^{-1}(y) &= \lifttoG{w_0} [\lifttoG{w_0}^{-1} y]_{-} \lifttoG{w_0}^{-1}.
    \end{align*}
  \end{definition}

  \label{sec:quivers-and-cluster-ensembles}

  \subsection{Quivers}

  \label{sec:quivers}

  We review quivers. These are graphical encodings of cluster algebras
  (specifically, those of geometric type), studied in
  \cite{fominzelevinsky2002}, \cite{fominzelevinsky2003}, and
  \cite{fominzelevinsky2005}. The quiver interpretation, first discussed
  in \cite{marshreinekezelevinsky2003}, is given a complete introduction
  in \cite{marsh2013}, which we follow. The only departure we need from
  Marsh's definition is to allow some edge weights to be in $\frac{1}{2}
  \mathbb{Z}$ instead of $\mathbb{Z}$. We also give a name to the
  symmetrized edge weights $\sigma(v,w)$, and encode the
  symmetrizability of edge weights into vertex weights $d_v$.

  \begin{definition}
    \label{def:quiver}
    A \defined{quiver} is a directed graph with no $1$- or $2$-cycles or
    multiple edges, with weights on edges and vertices. Vertex weights
    are in $\mathbb{N}_{>0}$, and edge weights are in
    $\frac{1}{2}\mathbb{Z}$. The edge weights are denoted $\sigma(v,w)$,
    and the vertex weights $d_v$. Also defined is the auxiliary quantity
    $\epsilon(v,w)$, by \[ \epsilon(v,w) = \frac{d_w}{\gcd(d_v,d_w)}
      \sigma(v,w). \] Vertices are either \defined{frozen} or
    \defined{non-frozen}, and edge weights are integral unless the edge
    is between two frozen vertices, in which case they are
    half-integral.
  \end{definition}

  \begin{figure}[h]
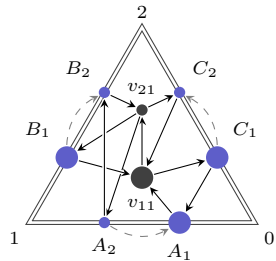

    \centering \includestandalone[mode=image|tex]{fig/QC2}

    \caption{A quiver. Larger vertices are of weight $2$, and dotted
      edges correspond to $\sigma(i,j) = \frac{1}{2}$.
      \label{fig:QC2}
    }
  \end{figure}

  \begin{remark}
    Almost all of the time, $\sigma(i,j)$ will lie in $\set{-1,0,+1}$.
    In this case, the edge weights $\sigma(i,j)$ are just the incidence
    matrix of the graph underlying the quiver.
  \end{remark}

  \begin{remark}
    The term \defined{quiver} occurs in other areas of literature as a
    multi-digraph. We use it here as a edge-weighted digraph with
    skew-symmetrizable adjacency matrix, encoding the exchange matrix of
    Fomin--Zelevinsky's cluster algebras.
  \end{remark}

  \begin{definition}
    For each non-frozen $v$, a quiver $\mu_{v}(Q)$ is defined. This is
    the \defined{mutation} of $Q$ at $v$. Edges in $\mu_{v}(Q)$ are
    defined by $\epsilon'(\cdot, \cdot)$, with \[ \epsilon'(x,y) =
      \begin{cases} \epsilon(x,y) & \text{if } \epsilon(x,v)
        \epsilon(v,y) \le 0, v \notin \set{x,y} \\ -\epsilon(x,y) &
        \text{if } v \in \set{x,y} \\ \epsilon(x,y) + |\epsilon(x,v)|
        \epsilon(v,y) & \text{if } \epsilon(x,v) \epsilon(v,y) > 0, v
        \notin \set{x,y}. \end{cases} \]
  \end{definition}

  Customarily, the vertex $v$ is renamed in $\mu_v(Q)$, e.g. to $v'$.
  This is because $\mu_v$ induces a map on the seed torus of $Q$ which
  changes the coordinate associated to $v$ (see
  \cref{def:cluster-ensemble}). For very long sequences, we will usually
  ignore the renamings for readability and think of $v$ as a vertex in a
  graph, not as a coordinate function.

  \begin{notation}
    \label{not:writing-mutations}
    Sequences of mutations are performed left to right: that is
    $\mu_{v,w} = \mu_{w} \circ \mu_{v}$. When subscripts would be
    awkward, we will also write mutation sequences as $\mu\set{v_1, v_2,
      \dotsc}$. We shall also employ $\prod_{v \in I} \mu_{v}$ to mean
    mutating at all vertices in $I$, in the order given by $I$.
  \end{notation}

  \begin{remark}
    When depicting quivers, we shall assume that all black edges have
    weight $\sigma = 1$, and gray, dashed edges have weight $\sigma =
    \frac{1}{2}$.

    Vertex weights will be depicted by the relative size of circles; we
    will only use weights 1, 2, and (only in the case of $G_2$) 3.
    Frozen vertices will be colored blue, but this will also be
    described in the text if at all relevant.

    \begin{figure}[h]
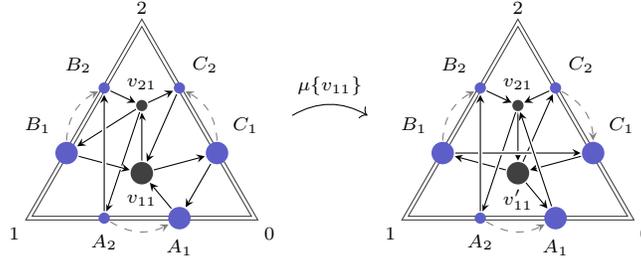

      \centering
      \includestandalone[mode=image|tex]{fig/quiver-mutation-example}

      \caption{Mutating the quiver of \cref{fig:QC2} at $v_{11}$.
        \label{fig:quiver-mutation-example}
      }
    \end{figure}

  \end{remark}

  \subsection{Cluster ensembles}

  \label{sec:cluster-ensembles}

  We now define cluster ensemble structures following
  \cite{fockgoncharov2003}. Also useful is \cite{ishibashi2018}.

  \begin{definition}
    \label{def:cluster-ensemble}
    For a quiver $Q$ with vertices $V$, let $\Lambda_V$ be the free
    abelian group generated by $V$, with dual $\Lambda_V^{\ast} =
    \Hom(\Lambda_V, \mathbb{Z})$. Let $\set{e_v : v \in V}$ be a basis
    for $\Lambda_V$, and $\set{f_v = d_v^{-1} e_v^{\ast}}$ a basis for
    $\Lambda_V^{\ast}$. Let $\Lambda^{\circ}$ be the $\mathbb{Q}$-span
    of $\set{f_v}$ inside $\Lambda^{\ast} \otimes \mathbb{Q}$.

    The \defined{seed $\mathcal{X}$-torus} is $\seedXtorus_Q =
    \Hom(\Lambda, \mathbb{C}^{\ast})$. The \defined{seed
      $\mathcal{A}$-torus} is $\seedAtorus_Q = \Hom(\Lambda^{\circ},
    \mathbb{C}^{\ast})$. The coordinates for these tori are denoted
    $\set{x_v}_{v \in V}$ and $\set{a_v}_{v \in V}$ and called
    \defined{$\mathcal{X}$-coordinates} and
    \defined{$\mathcal{A}$-coordinates} respectively. A quiver together
    with these coordinate tori is a \defined{seed}.

    There is a map $p: \seedAtorus_Q \to \seedXtorus_Q$. This map is
    characterized by the pullback $p^{\ast}(x_v)$ of an
    $\mathcal{X}$-coordinate to $\seedAtorus_Q$, which is defined as \[
      p^{\ast}(x_v) = \prod_{w \in V} a_w^{\epsilon(v,w)}. \]

    As these objects are defined in terms of quivers, we must define how
    quiver mutation affects the coordinates. As with $p$, we shall
    define $\mu_v$ by pullbacks--in this case, the pullback of the
    coordinates $x'_w$ and $a'_w$ for vertices $w$ of $\mu_v(Q)$.
    \begin{align*}
      \mu_v^{\ast} x'_w &= \begin{cases}
        x_v^{-1} & \text{if $v' = w$} \\
        x_w(1 + x_v^{\sgn \epsilon(v,w)})^{\epsilon(v,w)} & \text{if $v' \ne w$} \\
      \end{cases} \\
      \mu_v^{\ast} a'_w &=
      \begin{dcases}
        \frac{1}{a_v} \left(
        \prod_{\epsilon(v,z) > 0} a_z^{\epsilon(v,z)} +
        \prod_{\epsilon(v,z) < 0} a_z^{-\epsilon(v,z)}
        \right) & \text{if $v' = w$} \\
        a_w & \text{if $v' \ne w$}
      \end{dcases}
    \end{align*}

    Finally, define a \defined{cluster ensemble} as the orbit of a
    single seed under all quiver mutations.
  \end{definition}

  \begin{example}
    In the mutation of \cref{fig:quiver-mutation-example}, we have
    \begin{align*}
      a_{v_{11}'} = \frac{1}{a_{v_{11}}} \left( a_{C_1} a_{v_{21}} + a_{A_1} a_{C_2} a_{B_1} \right)
    \end{align*}
    The effect on $\mathcal{X}$-coordinates is
    \begin{align*}
      x_{v_{11}} &\mapsto \frac{1}{x_{v_{11}}} & x_{v_{21}} &\mapsto x_{v_{21}} (1 + x_{v_{11}})^2 = x_{v_{21}} + 2 x_{v_{11}} x_{v_{21}} + x_{v_{11}}^2 x_{v_{21}}\\
      x_{A_{1}} &\mapsto x_{A_{1}} (1 + x_{v_{11}}^{-1})^{-1} = \frac{x_{A_{1}} x_{v_{11}}}{1 + x_{v_{11}}} & x_{A_{2}} &\mapsto x_{A_{2}} \\
      x_{B_{1}} &\mapsto x_{B_{1}} (1 + x_{v_{11}}^{-1})^{-1} = \frac{x_{B_{1}} x_{v_{11}}}{1 + x_{v_{11}}} & x_{B_{2}} &\mapsto x_{B_{2}} \\
      x_{C_{1}} &\mapsto x_{C_{1}} (1 + x_{v_{11}}) = x_{C_{1}} + x_{C_{1}} x_{v_{11}} & x_{C_{2}} &\mapsto x_{C_{2}} (1 + x_{v_{11}}^{-1})^{-2} = \frac{x_{C_{2}} x_{v_{11}}^{2}}{1 + 2 x_{v_{11}} + x_{v_{11}}^2}
    \end{align*}
  \end{example}

  \begin{remark}
    The mutation relation for the $\mathcal{A}$- and
    $\mathcal{X}$-coordinates follows
    \cite[Equation~2.3]{fominzelevinsky2007}. The
    $\mathcal{A}$-coordinates especially form the \defined{cluster
      algebra} of Fomin--Zelevinsky which define the quiver. For
    consistency, we will refer to the cluster variables as
    $\mathcal{A}$-coordinates.
  \end{remark}

  A cluster ensemble, therefore, is a pair of coordinate structures on
  spaces, each defined by a quiver. Quiver mutation acts on each
  coordinate structure by replacing some coordinate functions, and $p$
  is always a map (generally non-surjective, non-injective) between the
  coordinate structures.

  \begin{remark}
    \label{rem:p-commutes-with-mutation}
    The map $p$ commutes with quiver mutation, as by
    \cite[Section~1.2]{fockgoncharov2003} and
    \cite[Proposition~3.9]{fominzelevinsky2007}.
  \end{remark}

  \section{Key identities}

  \label{sec:external-results}

  Here we reproduce two external results of particular significance to
  our constructions and proofs, from \cite{yangzelevinsky2008} and
  \cite{fominzelevinsky1999}. We refer to their respective origins for
  more details. These identities allow us to perform a program we call
  ``adjusting a minor coordinate''. Depending on the surrounding
  conditions, these will allow us to conclude that the action of
  mutation on $\seedAtorus_Q$ is to replace \[ \genminoryz{u}{v}{i}(g)
    \qquad \text{with} \qquad \genminoryz{u'}{v'}{i}(q) \] for some $u,
  v, u', v'$.

  \subsection{Actions of \texorpdfstring{$\sigma_G$}{sigma\_G}}

  \label{sec:yz-dynkin-identity}

  We give an overview of the identity used to justify mutations which
  apply $\sigma_G$ to the Dynkin diagram-like graph underlying a quiver,
  as in \cref{rem:sigma_G-action-on-Dynkin-diagram}. We use this in the
  proofs of \cref{lem:colsigmaG-does-sigmaG,lem:rowsigmaG-does-sigmaG}.
  This identity is collected from several results in
  \cite{yangzelevinsky2008}, so we import some notation.

  \begin{definition}
    Let the \defined{reduced double Bruhat cell} $L^{s,t} = N_{+}
    \lifttoG{s} N_{+} \cap B_{-} \lifttoG{t} B_{-}$. (We shall not need
    many facts about this object.) Recall that $k^{\star} =
    \sigma_G(k)$. The relation $a \prec_{c} b$ means that $a$ precedes
    $b$ in the Coxeter element $c$, and that $a$ and $b$ are connected
    by an edge in the Dynkin diagram associated to $c$.
  \end{definition}

  We shall mostly be interested in minors of the form
  $\genminoryz{c^m}{c^m}{k}$. Yang--Zelevinsky denote these as $x_{c^m
    \fundamentalweight_k;c}$, but we will avoid this for notational
  consistency.

  The main result is that for a cluster algebra $\mathcal{A}(c)$ defined
  by an initial seed with quiver of Dynkin type associated to Coxeter
  element $c$, there exists a $g \in L^{c,c^{-1}}$ such that all cluster
  coordinates and coefficients of $\mathcal{A}(c)$ are given by certain
  generalized minors of $g$. This allows combinatorially defining the
  exchange relations of all \defined{source} or \defined{sink mutations}
  (at vertices where all edges point out or in). These always exchange
  variables of the form $\genminoryz{c^m}{c^m}{i}$ with
  $\genminoryz{c^{m \pm 1}}{c^{m \pm 1}}{i}$. Yang--Zelevinsky refer to
  these as \defined{primitive exchange relations}.

  Finally, a number $h(i;c)$ is defined such that
  $\genminoryz{c^{h(i;c)}}{c^{h(i;c)}}{i} = \genminoryz{}{}{i^{\ast}}$.
  This allows convenient analysis of periodicity of source/sink mutation
  sequences.
  \begin{proposition}[{\cite[Equation~2.13]{yangzelevinsky2008}}]
    \label{prop:def-h(i;c)}
    In some cases, $h(i;c)$ is easily calculated.
    \begin{itemize}
    \item
      When $i^{\ast} = i$ (for example, if $\sigma_G = e$) we trivially
      have $h(i;c) = 0$.
    \item
      Otherwise, when $h$ is even, $c^{h/2} = w_0$, and by action of
      $w_0$ on simple roots, $h(i;c) = \frac{h}{2}$.
    \end{itemize}
    When $G$ is of type $A_{\ell}$, however, the calculation is more
    delicate. Let \[ t_{+} = s_1 s_3 \dotsb s_{2 \lfloor (\ell+1)/2
        \rfloor - 1}, \qquad t_{-} = s_2 s_4 \dotsc s_{2 \lfloor
        (\ell+1)/2 \rfloor}, \qquad w_0 = \overbrace{t_{+} t_{-} \dotsb
        t_{\pm}}^{\text{$h = \ell + 1$ factors}}, \qquad c = t_{+}
      t_{-}. \] Then by \cite[Equation~2.13]{yangzelevinsky2008}, as $h
    = \ell + 1$ be the Coxeter number for $A_{\ell}$, \[ h(i;c) =
      \begin{cases} \lfloor \frac{h}{2} \rfloor = \lfloor \frac{\ell +
          1}{2} \rfloor & \text{$i$ even} \\ \lceil \frac{h}{2} \rceil =
        \lceil \frac{\ell + 1}{2} \rceil & \text{$i$ odd.} \end{cases}
    \]
  \end{proposition}
  Note that \[ \genminoryz{c^{m + h(k;c) + 1}}{c^{m + h(k;c) + 1}}{k} =
    \genminoryz{c^m}{c^m}{k^{\ast}}. \]

  We now combine several results of \cite{yangzelevinsky2008}, mainly
  Theorem~1.5.

  \begin{proposition}
    \label{prop:compressed-yang-zelevinsky}
    The cluster variables in $\mathcal{A}(c)$ satisfy the following
    primitive exchange relation (letting $A$ be the Cartan matrix): \[
      \genminoryz{c^{m-1}}{c^{m-1}}{k} \genminoryz{c^m}{c^m}{k} =
      \prod_{i \prec_{c} k} (\genminoryz{c^m}{c^m}{i})^{-A_{i,k}}
      \prod_{k \prec_{c} i}
      (\genminoryz{c^{m-1}}{c^{m-1}}{i})^{-A_{i,k}} + 1. \]
  \end{proposition}

  \begin{remark}
    This does not give us much information about $g \in L^{c,c^{-1}}$,
    but we will restrict our attention to periodic mutations. The only
    effect will be to change which root the vertices of the quiver are
    associated to.
  \end{remark}

  \subsection{Grid exchange relations}

  \label{sec:thm-1.17}

  Here we give an overview of the identity
  \cite[Theorem~1.17]{fominzelevinsky1999} used to justify the exchange
  relations of $\murottwist$ and $\muflipcore$, and to prove
  \cref{lem:fgcs-exists-rot-works-trivial,lem:fgcs-exists-flip-works}:

  \begin{theorem}[{\cite[Theorem~1.17]{fominzelevinsky1999}}]
    \label{thm:local-1.17}
    For $u, v$ two words in the Weyl group $W$ such that $\length(u s_i)
    = \length(u) + 1$ and $\length(v s_i) = \length(v) + 1$, and $A$ the
    Cartan matrix, \[ \genminoryz{u}{v}{i} \genminoryz{u s_i}{v s_i}{i}
      = \genminoryz{u s_i}{v}{i} \genminoryz{u}{v s_i}{i} + \prod_{j \ne
        i} (\genminoryz{u}{v}{j})^{-A_{j,i}}. \]
  \end{theorem}

  \begin{remark}
    It appears that this theorem only allows us to adjust the minor
    coordinates by one letter at a time: $s_i$. By the same logic as
    \cref{rem:all-wkomegaj-are-gamma-k-or-trivial}, however, we actually
    are able to adjust $u$ and $v$ by longer sequences of letters, as
    long as they do not contain multiple $s_i$.
  \end{remark}

  We use this theorem in situations where, mutating at $v_{a,b}$,
  \begin{align*}
    \genminoryz{u s_i}{v s_i }{i} &= \text{Coordinate of $v_{a,b}$ before mutation} \\
    \genminoryz{u}{v}{i} &= \text{Coordinate of $v_{a,b}$ after mutation} \\
    \genminoryz{u }{v s_i}{i} &= \text{Coordinate to left of $v_{a,b}$} \\
    \genminoryz{u s_i}{v }{i} &= \text{Coordinate to right of $v_{a,b}$} \\
    \genminoryz{u }{v}{j} &= \text{Coordinates in same column as $v_{a,b}$}
  \end{align*}

  \begin{example}
    We look ahead to the results of \cref{sec:main-result} and assume
    all edge coordinates are $1$ for simplicity.
    \Cref{fig:QC4-middle-murot} depicts $Q_{C_4}$ partway through the
    rotation mutation. The minors associated to relevant coordinates
    (see the proof of \cref{lem:fgcs-exists-rot-works-trivial} for
    details) are at this point
    \begin{align*}
      v_{22} &: \genminoryz{w_1}{w_{10}}{2} = \genminoryz{w_{3} s_2}{w_{11} s_2}{2}\\
      v_{21} &: \genminoryz{w_5}{w_{10}}{2} = \genminoryz{w_{3} }{w_{11} s_2}{2}\\
      v_{23} &: \genminoryz{w_1}{w_{14}}{2} = \genminoryz{w_{3} s_2}{w_{11} }{2}\\
      v_{12} &: \genminoryz{w_5}{w_{13}}{1} = \genminoryz{w_{3} }{w_{11} }{1}\\
      v_{32} &: \genminoryz{w_1}{w_{11}}{3} = \genminoryz{w_{3} }{w_{11} }{3}\\
    \end{align*}

    \begin{figure}[h]
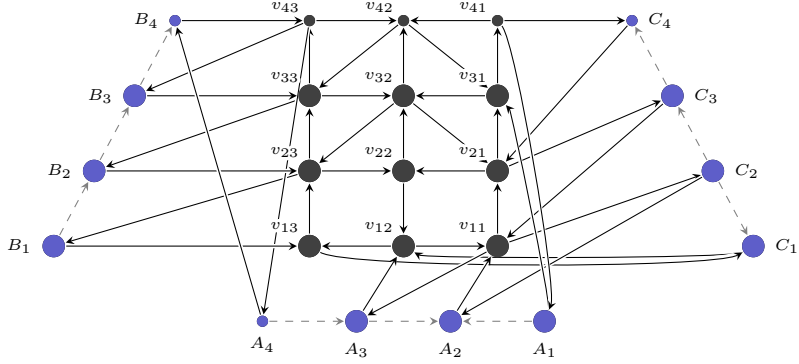

      \centering
      \includestandalone[mode=image|tex]{fig/QC4-middle-murot}

      \caption{After applying $\mu\Set{v_{11}, v_{21}, v_{31}, v_{41},
          v_{12}}$ to $Q_{C_4}$. The next mutation in $\murot$ is at
        $v_{22}$.
        \label{fig:QC4-middle-murot}
      }
    \end{figure}

    The equality follows from
    \cref{rem:all-wkomegaj-are-gamma-k-or-trivial}. Therefore, by
    applying the identity we see that mutating at $v_{22}$ will change
    the coordinate to $\genminoryz{w_{3}}{w_{11}}{2} =
    \genminoryz{w_{6}}{w_{14}}{2}$.
  \end{example}

  Finally, we note that the requirements on $u$ and $v$ may be loosened,
  which we need to prove \cref{lem:fgcs-exists-flip-works}.

  \begin{lemma}
    \label{lem:thm-1.17-for-repeated-c}
    Even if $\length(u s_i) \ne \length(u) + 1$ or $\length(v s_i) \ne
    \length(v) + 1$, the result of \cref{thm:local-1.17} still holds as
    long as $u$ and $v$ are subwords of a repeated Coxeter element $c$.
  \end{lemma}
  \begin{proof}
    From the proof in \cite{fominzelevinsky1999}, the only reason for
    the length condition is to ensure that $\lifttoG{v} \lifttoG{s_i} =
    \lifttoG{vs_i}$ (and similar for $u$), since the lifting is not
    quite a homomorphism: for example $\lifttoG{s_i s_i} \ne
    \lifttoG{e}$.

    However, the only situations in which multiplication is not
    preserved is when the multiplication by $s_i$ induces a
    length-shortening identity in $W$. But if $u$ and $v$ are subwords
    of a repeated Coxeter element, their suffixes will always be of the
    form $w_k$. These admit no length-shortening braid relations.

    So the only possibility for $\lifttoG{v s_i} \ne \lifttoG{v}
    \lifttoG{s_i}$ is if $v s_i$ contains a copy of $w_0 = c^{h/2}$. But
    $\lifttoG{c^{h}} = s_G$ is an element of $H$. By prepending copies
    of $c^h$ to either $u$ or $v$, we may ensure that the difference in
    length between $u$ and $v$ is never more than $\frac{h}{2} - 1$.
    Then, using \[ \genminoryz{w u}{w v}{i}(g) =
      \genminoryz{u}{v}{i}(w^{-1} g w) \] we reduce to the case where
    $u$ and $v$ are each some $w_k$ shorter than $w_0$. By replacing $g$
    with $w g w^{-1} s_G^{\epsilon}$ (where $\epsilon \in {0,1}$
    depending on the exact difference between $u$ and $v$), the desired
    identity follows from the regular theorem.
  \end{proof}

  \section{Coordinates on generically-decorated representations}

  \label{sec:reps-from-conf3star}

  Here we review \defined{decorations} of representations and define
  \defined{Fock--Goncharov coordinate structures}. For a representation
  $\rho$ and a chosen triangulation, a decoration of $\rho$ is a
  collection of flags for each simplex in the triangulation. These flags
  can recreate $\rho$, and have a canonical form that admits generalized
  minor coordinates. Fock--Goncharov coordinate structures will describe
  coordinates on these decorations.

  These decorations apply to surfaces, but also to $3$-manifolds with
  fixed ideal triangulations. We follow \cite{fockgoncharov2006} for the
  surface case and \cite{zickert2016} for the $3$-manifold case.

  \subsection{Configurations and gluings}

  \label{sec:configurations-and-gluings}

  \begin{definition}
    \label{def:Conf_k_star}
    Let $G$ be a Lie group over $\mathbb{C}$, with sufficient choices to
    define $G_0$ as in \cref{def:gaussian-decomposition}. Let $K$ be a
    subgroup of $G$ (we will ultimately use $N_{+}$). A tuple of cosets
    $(g_0K, g_1K, \dotsc, g_{m}K)$ is \defined{sufficiently generic} if
    each $g_i^{-1} g_j \in \lifttoG{w_0}G_0$.

    Such a tuple corresponds to a labeling of edges in an oriented
    $m$-simplex by elements of $\lifttoG{w_0}G_0$, as in
    \cref{fig:labeled-sufficiently-generic-simplex}, with the edge from
    $i$ to $j$ (assuming $j>i$) labeled by $g_i^{-1} g_j$.

    The variety of such sufficiently generic $(m+1)$-tuples is the
    \defined{configuration space} $\Conf_{m+1}(G/K)$. Identifying
    elements which differ by left-multiplication of $G$ yields the
    variety $\Conf_{m+1}^{\ast}(G/K)$.
  \end{definition}

  \begin{figure}[h]
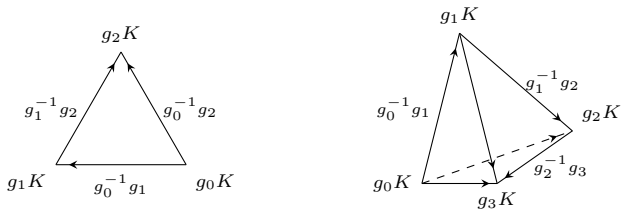

    \centering \includestandalone[mode=image|tex]{fig/labeled-sufficiently-generic-simplex}

    \caption{Oriented $2$- and $3$-simplices labeled by elements of
      $\Conf_3^{\ast}(G/K)$ and $\Conf_4^{\ast}(G/K)$.
      \label{fig:labeled-sufficiently-generic-simplex}
    }
  \end{figure}

  \begin{definition}
    \label{def:peripheral-stuff}
    Let $M$ be a compact manifold (possibly with boundary) and $G$ a Lie
    group over $\mathbb{C}$, with $K$ a subgroup of $G$. A subgroup $L$
    of $\pi_1(M)$ is \defined{peripheral} if there is a boundary
    component $D$ of $M$ such that $L$ is induced by the inclusion of
    $D$; that is $\iota^{\ast}(\pi_1(D)) = L$.

    A representation $\rho : \pi_1(M) \to G$ is a
    \defined{$(G,K)$-representation} if, for every peripheral subgroup
    $L$, the image $\rho(L)$ is a conjugate of $K$. In the case $K =
    N_{+}$, the term \defined{boundary-unipotent} is used. If $K =
    B_{+}$, then $\rho$ is \defined{boundary-borel}.
  \end{definition}

  \begin{lemma}[{\cite[Lemma~5.8, Proposition~5.9]{zickert2016}}]
    \label{lem:canonical-form-for-conf_3_star}
    There is an isomorphism of varieties \[ \Conf_3^{\ast}(G/N_{+})
      \cong H^3 \times_{H} (N_{-} \cap \lifttoG{w_0} G_0). \] The fiber
    product $\times_{H}$ means that for $(h_1, h_2, h_3, u) \in
    \Conf_3^{\ast}(G/N_{+}))$, we have \[ [\lifttoG{w_0}^{-1} u]_{0} =
      (w_0(h_3 h_1) h_2)^{-1}. \]
  \end{lemma}

  Thus we will often write $(h_1, h_2, h_3, u)$ for $\alpha$, which has
  one representative $(N_{+}, \lifttoG{w_0} h_1 N_{+}, u w_0(h_1) h_2
  s_G N_{+})$. We refer to these as \defined{canonical forms} for
  $\Conf_3^{\ast}(G/N_{+})$.

  \begin{definition}
    \label{rem:canonical-form-for-B-cosets}
    By expanding $N_{+}$ to $B_{+}$ above, we obtain that $\alpha \in
    \Conf_3^{\ast}(G/B_{+})$ can be given as $(B_{+}, \lifttoG{w_0}
    B_{+}, u B_{+})$. Choose the unique $u_x$ such that, for all $k$,
    the minor $\genminoryz{w_k}{}{1}(u_x) = 1$. This element $u_x$ is
    the \defined{canonical form} for $\alpha$.

    For an arbitrary $(B_{+}, \lifttoG{w_0} B_{+}, u B_{+})$, by
    expanding \cref{def:generalized-minor} there is a unique (up to the
    center of $G$) element $h$ such that $h u h^{-1}$ is the canonical
    form $u_x$. When $G$ is centerless, we denote this element by
    $n(u)$.
  \end{definition}

  \begin{proposition}
    \label{prop:Phi-Psi-and-rot}
    This summarizes \cite[Proposition~5.10]{zickert2016}. Let $\alpha
    \in \Conf_3^{\ast}(G/N_{+})$ be given by $\alpha = (h_1, h_2, h_3,
    u)$. Recall $\Phi$ and $\Psi$ of \cref{def:phi-and-psi}. The map \[
      \operatorname{rot}: \Conf_3^{\ast}(G/N_{+}) \to
      \Conf_3^{\ast}(G/N_{+}), \qquad (g_0 N_{+}, g_1 N_{+}, g_2 N_{+})
      \mapsto (g_2 N_{+}, g_0 N_{+}, g_1 N_{+}) \] is given, in this
    form, via \[ \operatorname{rot} : (h_1, h_2, h_3, u) \mapsto (h_3,
      h_1, h_2, h_2^{-1} (w_0(h_1))^{-1} (\Phi \Psi \Phi \Psi)(u)
      (w_0(h_1)) h_2). \]
  \end{proposition}

  \begin{definition}
    \label{def:gluing-configurations}
    This summarizes \cite[Section~2.2.1]{zickert2016}. We define and
    $\Conf_3^{\ast}(G/N_{+}) \times_{13}^{s_G} \Conf_3^{\ast}(G/N_{+})$
    by gluing copies of $\Conf_3^{\ast}(G/N_{+})$ along matching copies
    of $\Conf_2^{\ast}(G/N_{+})$. Specifically,
    \begin{align*}
      \Conf_3^{\ast}(G/N_{+}) \times_{02}^{s_G} \Conf_3^{\ast}(G/N_{+}) = \Set{ \substack{(g_0 s_G N_{+}, g_1 N_{+}, g_2 N_{+}),\\ (g_0 N_{+}, g_2 N_{+}, g_3 N_{+})} : (g_0 N_{+}, g_1 N_{+}, g_2 N_{+}, g_3 N_{+}) \in \Conf_4^{\ast}(G/N_{+}) } \\
      \Conf_3^{\ast}(G/N_{+}) \times_{13}^{s_G} \Conf_3^{\ast}(G/N_{+}) = \Set{ \substack{(g_1 N_{+}, g_2 N_{+}, g_3 N_{+}),\\ (g_0 N_{+}, g_1 s_G N_{+}, g_3 N_{+})} : (g_0 N_{+}, g_1 N_{+}, g_2 N_{+}, g_3 N_{+}) \in \Conf_4^{\ast}(G/N_{+}) } \\
    \end{align*}
    We also define the following maps from $\Conf_4^{\ast}(G/N_{+})$ to
    $\Conf_{3}^{\ast}(G/N_{+}) \times_{jk}^{s_G}
    \Conf_3^{\ast}(G/N_{+})$ for $jk = 02, 13$:
    \begin{align*}
      \Psi_{02} : (g_0 N_{+}, g_1 N_{+}, g_2 N_{+}, g_3 N_{+}) &\mapsto \substack{(g_0 s_G N_{+}, g_1 N_{+}, g_2 N_{+}),\\ (g_0 N_{+}, g_2 N_{+}, g_3 N_{+})} \\
      \Psi_{13} : (g_0 N_{+}, g_1 N_{+}, g_2 N_{+}, g_3 N_{+}) &\mapsto \substack{(g_1 N_{+}, g_2 N_{+}, g_3 N_{+}),\\ (g_0 N_{+}, g_1 s_G N_{+}, g_3 N_{+})}
    \end{align*}

    See part of \cref{fig:QC2-and-flip} for a graphical depiction.

    In the surface case, these correspond to the two different ways of
    triangulating a quadrilateral. The act of retriangulating a single
    quadrilateral this way is called a ``flip''. In the $3$-manifold
    case, these correspond to decomposing a tetrahedron's vertices into
    the two triangles' worth.
  \end{definition}

  \begin{remark}
    \label{rem:purpose-of-sG}
    The element $s_G$ is in the center of $G$, so is only relevant for
    $N_{+}$-cosets. Its purpose is to allow the quiver amalgamation (see
    \cref{def:dv-dh-construction} ahead) to agree with the
    identification of elements of $\Conf_3^{\ast}(G/N_{+})$ along a copy
    of $\Conf_2^{\ast}(G/N_{+})$.

    To see this, note that \cref{fig:dv-construction} shows identifying
    a $0$--$2$ edge with a $0$--$1$ edge, and \cref{fig:dh-construction}
    identifes a $0$--$2$ edge with a $1$--$2$ edge. Since the associated
    elements of $\Conf_2^{\ast}(G/N_{+})$ are oriented via the cyclic
    ordering on $(0,1,2)$, this quiver amalgamation corresponds to
    identifying elements of $\Conf_2^{\ast}(G/N_{+})$ ``backwards''.

    This use of $s_G$ is why $\mathcal{A}_{G,\Sigma}$ is the moduli
    space of \emph{twisted}, decorated representations from
    $\pi_1(\Sigma)$ into $G$. See \cite[Section~8.6]{fockgoncharov2006}
    and \cite[Section~2.3]{zickert2016} for more details.
  \end{remark}

  \subsection{Triangular quivers and Fock--Goncharov coordinate
    structures}

  We want a variety that is independent of the choice of triangulation.
  Therefore, we will need coordinate change maps corresponding to
  changes in triangulation (the rotation of individual triangles, or the
  quadrilateral flip).

  Here we define the outputs of the algorithm of \cref{sec:main-result}.
  These are \defined{triangular quivers} (see
  \cref{def:triangular-quiver}), which have vertices on edges and in the
  interior. If, moreover, these quivers admit a rotation and a flip as
  quiver mutations, they have \defined{triangulation-compatible
    symmetry} (see \cref{def:tcs}). And if we can then associate the
  $\mathcal{A}$- and $\mathcal{X}$-coordinates to coordinates on
  $\Conf_3^{\ast}(G/N_{+})$ and $\Conf_3^{\ast}(G/B_{+})$, we have a
  \defined{Fock--Goncharov coordinate structure} (see \cref{def:fgcs}).

  \begin{definition}
    \label{def:dynkin-type}
    Let $D$ be a Dynkin diagram, with associated Cartan matrix
    $\begin{bmatrix}a_{ij}\end{bmatrix}$. A quiver $Q$ is of
    \defined{Dynkin type $D$} if $\abs{\epsilon(i,j)} = -a_{ji}$ (for $i
    \ne j$) and $\sigma(i,j)$ takes values only in $\set{-1, 0, +1}$.
    The undirected graph of $Q$ carries the adjacency information of
    $D$, and the vertex weights carry the edge details.

    $Q$ is of \defined{half-Dynkin type $D$} if $2Q$, the quiver
    obtained by multiplying all edge weights by $2$, is of Dynkin type
    $D$.
  \end{definition}

  \begin{definition}
    \label{def:Coxeter-elt-induced-by-quiver}
    Let $D$ be a quiver of Dynkin type (with no oriented cycles). We can
    obtain a particular Coxeter element $c$ by requiring that
    $s_{\alpha_i}$ appear before $s_{\alpha_j}$ in $c$ if $\sigma(i, j)
    > 0$ in $D$. Such a $c$ is \defined{induced} by the quiver $D$.
  \end{definition}

  \begin{definition}
    \label{def:tree-like-quiver}
    A quiver of Dynkin type is \defined{tree-like} if, as an unoriented
    graph, it is a tree, and the directed edges always point away from
    the root. In this case, there is a partition of vertices $\set{T_0,
      T_1, \dotsc, T_m}$ such that the vertices in $T_i$ are of distance
    $i$ (in the unoriented graph) from the root vertex.

    Note that Coxeter elements induced by tree-like Dynkin quivers
    always take the form \[ c = \left( \prod_{\alpha \in T_0} s_{\alpha}
      \middle) \middle( \prod_{\alpha \in T_1} s_{\alpha} \middle)
      \dotsb \middle( \prod_{\alpha \in T_m} s_{\alpha} \right). \]
  \end{definition}

  \begin{definition}
    \label{def:well-rooted-tree-like-quiver}
    The involution $\sigma_G$ acts on quivers of Dynkin type in the same
    way that it would act on the Dynkin diagram as in
    \cref{rem:sigma_G-action-on-Dynkin-diagram}. A tree-like quiver $Q$
    of Dynkin type $D$ is \defined{well-rooted} if this action of
    $\sigma_G$ preserves all $T_i$ partitions. See
    \cref{fig:well-rooted-quivers}.
  \end{definition}

  \begin{figure}[h]
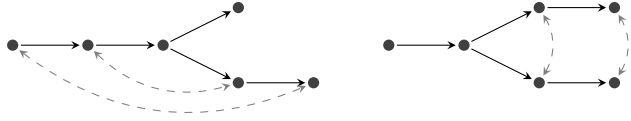

    \centering \includestandalone[mode=image|tex]{fig/well-rooted-E6}

    \caption{A non-well-rooted quiver (left) and a well-rooted quiver
      (right) for $E_6$.
      \label{fig:well-rooted-quivers}
    }
  \end{figure}

  \begin{remark}
    In the most cases, the quiver obtained by applying a na\"ive
    ``left-to-right'' ordering to the common presentation of the Dynkin
    diagram is well-rooted. The pathological case is $E_6$.
  \end{remark}

  \begin{definition}
    \label{def:triangular-quiver}
    A quiver $Q$ is \defined{triangular} if, for a standard $2$-simplex
    $\sigma$ with vertices $0, 1, 2$, each vertex in $Q$ is associated
    to the interior of some sub-simplex $\sigma'$ of $\sigma$.
    $\facemap_{\sigma'}(Q)$ is defined to be the sub-quiver of $Q$
    obtained by deleting all vertices except those that lie on the
    interior of $\sigma$. See \cref{fig:visually-triangular}.

    In other words, a triangular quiver $Q$ is one that can be split
    into sub-quivers:
    \begin{itemize}
    \item
      $\facemap_{(0,1)}(Q)$, $\facemap_{(1,2)}(Q)$, and
      $\facemap_{(0,2)}(Q)$: the vertices on the edges of the triangle.
    \item
      $\facemap_{(0,1,2)}(Q)$: the vertices in the interior of the
      triangle.
    \end{itemize}

    We label the vertices on the $\facemap_{(0,1)}$ edge's vertices by
    $A_i$, those on the $\facemap_{(1,2)}$ edge by $B_i$, and those on
    the $\facemap_{(0,2)}$ edge by $C_i$.

    An isomorphism $\varphi$ of triangular quivers must preserve these
    classifications: $\facemap_{\Delta}(\varphi(Q)) =
    \varphi(\facemap_{\Delta}(Q))$.
  \end{definition}

  \begin{figure}[h]
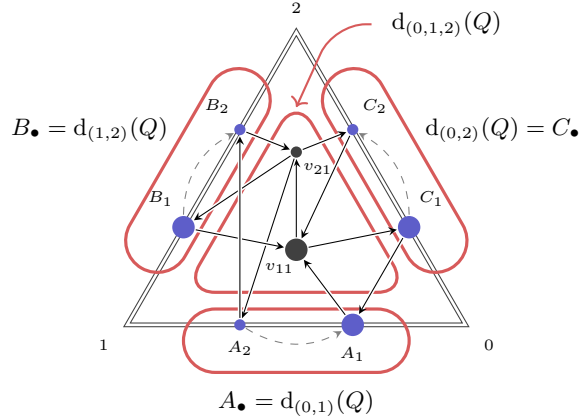

    \centering
    \includestandalone[mode=image|tex]{fig/visually-triangular}

    \caption{A triangular quiver (this one is associated to a Lie group
      of type $C_2$).
      \label{fig:visually-triangular}
    }
  \end{figure}

  \begin{definition}
    \label{def:dv-dh-construction}
    Let $Q$ be a triangular quiver such that the edge sub-quivers are
    isomorphic as sets of vertices. That is, fix isomorphisms $\varphi:
    \facemap_{(0,1)}(Q) \cong \facemap_{(0,2)}(Q)$ and $\psi :
    \facemap_{(0,2)}(Q) \cong \facemap_{(1,2)}(Q)$. Then we define
    \begin{itemize}
    \item
      $\Qdv$ as the quiver obtained by identifying two copies of $Q$
      along $\varphi$, and
    \item
      $\Qdh$ as the quiver obtained by identifying two copies of $Q$
      along $\psi$.
    \end{itemize}
    In performing the identification of vertices, edge weights are
    added. They need not agree, and in particular might cancel; see
    \cref{fig:dv-construction,fig:dh-construction}.
  \end{definition}

  This is a special case of \defined{amalgamation} as defined in
  \cite[\S 2.2]{fockgoncharov2006_2}.

  \begin{figure}[h]
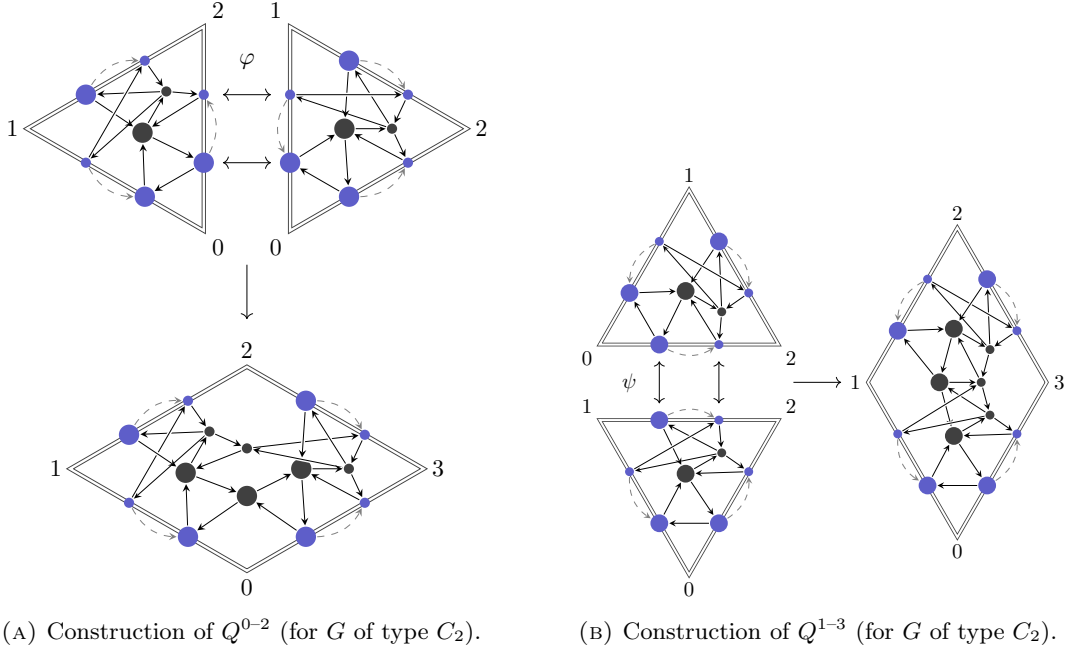

    \centering
    \begin{subfigure}[t]{0.40\textwidth}
      \centering \includestandalone[mode=image|tex,width=\textwidth]{fig/dv-construction}

      \caption{Construction of $\Qdv$ (for $G$ of type $C_2$).
        \label{fig:dv-construction}
      }
    \end{subfigure}
    ~ \hspace{0.05\textwidth} ~
    \begin{subfigure}[t]{0.40\textwidth}
      \centering \includestandalone[mode=image|tex,width=\textwidth]{fig/dh-construction}

      \caption{Construction of $\Qdh$ (for $G$ of type $C_2$).
        \label{fig:dh-construction}
      }
    \end{subfigure}
    \caption{Amalgamating triangular quivers.}
  \end{figure}

  \begin{definition}
    \label{def:tcs}
    A quiver \defined{has triangulation-compatible symmetry} if
    \begin{enumerate}
    \item
      \label{def:tcs-1}
      It is triangular by \cref{def:triangular-quiver},
    \item
      \label{def:tcs-rot}
      There is a quiver mutation $\murot$ which is an isomorphism of
      triangular quivers between $Q$ and $Q'$, where $Q'$ is obtained
      from $Q$ by permuting simplex indices from $(0,1,2)$ to $(2,0,1)$.
    \item
      \label{def:tcs-flip}
      There is a quiver mutation $\muflip$ which transforms $\Qdv$ to
      $\Qdh$.
    \end{enumerate}
  \end{definition}

  If $\triang$ is an oriented triangulation of a surface, and $\triang'$
  is an orientation-preserving retriangulation of $\triang$, then
  inscribing $Q$ in each simplex of $\triang$ and $\triang'$ and
  amalgamating along edges results in two quivers that differ by some
  sequence of $\murot$ and $\muflip$. Therefore, if the cluster
  variables of $Q$ are coordinates for the representation variety, the
  existence of $\murot$ and $\muflip$ will (eventually) show the variety
  to be independent of the specific choice of triangulation.

  \begin{definition}
    \label{def:fgcs}
    Fix a semisimple Lie group $G$ over $\mathbb{C}$ with a fixed
    maximal torus $H$ and a maximal unipotent subgroup $N_{+}$, with
    $B_{+} = HN_{+}$. A quiver $Q$ carries a \defined{Fock--Goncharov
      coordinate structure for $G$ } if
    \begin{enumerate}
    \item
      \label{def:fgcs-1}
      $Q$ has triangulation-compatible symmetry as in \cref{def:tcs}.
    \item
      \label{def:fgcs-2}
      Each of the edge sub-quivers of $Q$ are of half-Dynkin type for
      $G$, as by \cref{def:dynkin-type}, and the isomorphisms $\varphi$
      and $\psi$ of \cref{def:dv-dh-construction} identify vertices that
      come from the same nodes in the Dynkin diagram.
    \item
      \label{def:fgcs-M-is-birational-equivalence}
      There exists a map $\conftoquiv : \Conf_3^{\ast}(G/N_{+}) \to
      \seedAtorus_Q$, which is a birational equivalence.
    \item
      \label{def:fgcs-rot-works}
      Moreover, $\conftoquiv$ respects the rotation. That is, the
      following diagrams commute:

      {\centering
        \begin{tikzcd}
          \Conf_3^{\ast}(G/N_{+}) \ar[r, "\conftoquiv"] \ar[d, "\operatorname{rot}"] & \seedAtorus_Q \ar[d, "\murotind"] &
          \Conf_3^{\ast}(G/B_{+}) \ar[r, "\conftoquiv"] \ar[d, "\operatorname{rot}"] & \seedXtorus_Q \ar[d, "\murotind"] \\
          \Conf_3^{\ast}(G/N_{+}) \ar[r, "\conftoquiv"] & \seedAtorus_{\murot(Q)} &
          \Conf_3^{\ast}(G/B_{+}) \ar[r, "\conftoquiv"] & \seedXtorus_{\murot(Q)}
        \end{tikzcd}

      } with $\operatorname{rot} : (g_0 K, g_1 K, g_2 K) \mapsto (g_2 K,
      g_0 K, g_1 K)$ (see also \cref{fig:QC2-and-rot}).
    \item
      \label{def:fgcs-flip-works}
      Moreover, $\conftoquiv$ respects the flip. That is, the following
      diagrams commute:

      {\centering
        \begin{tikzcd}
          \Conf_3^{\ast}(G/N_{+}) \times_{02}^{s_G} \Conf_3^{\ast}(G/N_{+}) \ar[r, "\conftoquiv"] \ar[d, "\Psi_{13} \circ \Psi_{02}^{-1}"] & \seedAtorus_{\Qdv} \ar[d, "\muflipind"] &
          \Conf_3^{\ast}(G/B_{+}) \times_{02}^{s_G} \Conf_3^{\ast}(G/B_{+}) \ar[r, "\conftoquiv"] \ar[d, "\Psi_{13} \circ \Psi_{02}^{-1}"] & \seedXtorus_{\Qdv} \ar[d, "\muflipind"] \\
          \Conf_3^{\ast}(G/N_{+}) \times_{13}^{s_G} \Conf_3^{\ast}(G/N_{+}) \ar[r, "\conftoquiv"] & \seedAtorus_{\Qdh} &
          \Conf_3^{\ast}(G/B_{+}) \times_{13}^{s_G} \Conf_3^{\ast}(G/B_{+}) \ar[r, "\conftoquiv"] & \seedXtorus_{\Qdh}
        \end{tikzcd}

      } where $\Psi_{jk}$ are defined as in
      \cref{def:gluing-configurations} (see also
      \cref{fig:QC2-and-flip}).
    \end{enumerate}
  \end{definition}

  \begin{remark}
    By the map $p$ of \cref{def:cluster-ensemble}, $\conftoquiv$
    provides a map to $\seedXtorus_Q$. In fact, this map factors through
    $\Conf_3^{\ast}(G/B_{+})$. By abuse of notation, we also refer to
    this map as $\conftoquiv$ in
    \cref{def:fgcs-rot-works,def:fgcs-flip-works}.
  \end{remark}

  \begin{example}
    \Cref{fig:QC2-and-rot,fig:QC2-and-flip} illustrate the last two
    demands: to move back and forth between $\Conf_{3}^{\ast}(G/N_{+})$
    and $\seedAtorus_{Q}$, or $\Conf_{4}^{\ast}(G/N_{+})$ and
    $\seedAtorus_{\Qdv}$ in terms of coordinates and quiver mutations.
  \end{example}

  \begin{figure}[h]
    \centering \includestandalone[mode=image|tex]{fig/QC2-and-rot}

    \caption{Quiver for type $C_2$, illustrating $\murotind \circ
      \conftoquiv = \conftoquiv \circ \operatorname{rot}$, with
      $\seedAtorus_Q$ depicted by a graph for $Q$.
      \label{fig:QC2-and-rot}
    }
  \end{figure}

  \begin{figure}[h]
    \centering \includestandalone[mode=image|tex]{fig/QC2-and-flip}

    \caption{Quiver for type $C_2$, illustrating $\muflipind \circ
      \conftoquiv = \conftoquiv \circ \Psi_{13} \circ \Psi_{02}^{-1}$.
      The tetrahedron shows the use of the flip in the $3$-manifold
      context.
      \label{fig:QC2-and-flip}
    }
  \end{figure}

  We can now restate \cref{thm:intro-main-result} as the following

  \begin{theorem}
    \label{thm:fgcs-exists}
    For a split semisimple simply-connected (or centerless) algebraic
    group $G$ over $\mathbb{Q}$, a quiver with a Fock--Goncharov
    structure exists.
  \end{theorem}

  We will mostly focus on the simple case. The semisimple case follows
  quickly as described in \cref{sec:products}. The construction for
  \cref{thm:fgcs-exists} is given in \cref{sec:main-result}.

  \subsection{From coordinates to representations}

  \label{sec:coords-to-representations}

  The Fock--Goncharov program for defining $\mathcal{A}_{G,\Sigma}$ and
  $\mathcal{X}_{G,\Sigma}$ works for any compact, oriented $\Sigma$ with
  boundary components. In the case that the boundary consists of
  punctures, we can describe reconstructing a representation from a
  point in the moduli space. We start with the $\mathcal{A}$-coordinate
  version due to complications in technicalities, referring to
  \cite[Section~8]{fockgoncharov2006}, \cite[Section~6]{zickert2016} for
  more details.

  We start with an ideal triangulation $\triang$ of $\Sigma$, together
  with a quiver $Q$ and seed torus $\seedAtorus_Q$ for each triangle in
  $\triang$, amalgamated together.

  First, we truncate the triangulation, as in
  \cref{fig:truncated-simplex}. We will associate an element of $G$ to
  each directed edge: elements on the long edges will be called
  $\alpha_{ij}$, and elements on short edges will be called
  $\beta_{ijk}$. The local labeling of vertices by $0$, $1$, and $2$ is
  determined by triangularity (see \cref{def:triangular-quiver}) of the
  quiver $Q$.

  \begin{figure}[h]
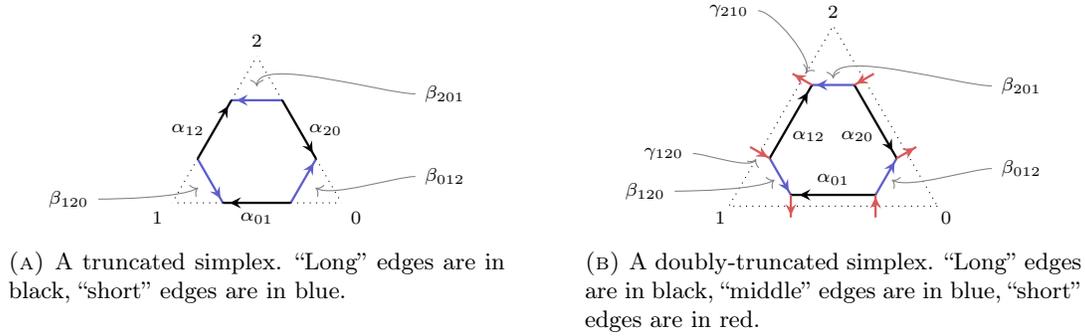

    \centering
    \begin{subfigure}[t]{0.40\textwidth}
      \centering
      \includestandalone[mode=image|tex]{fig/truncated-simplex}

      \caption{A truncated simplex. ``Long'' edges are in black,
        ``short'' edges are in blue.
        \label{fig:truncated-simplex}
      }
    \end{subfigure}
    ~ \hspace{0.05\textwidth} ~
    \begin{subfigure}[t]{0.40\textwidth}
      \centering
      \includestandalone[mode=image|tex]{fig/doubly-truncated-simplex}

      \caption{A doubly-truncated simplex. ``Long'' edges are in black,
        ``middle'' edges are in blue, ``short'' edges are in red.
        \label{fig:doubly-truncated-simplex}
      }
    \end{subfigure}
    \caption{Truncations of $\triang$. The left is used for
      $\mathcal{A}$-coordinates, the right for
      $\mathcal{X}$-coordinates.}
  \end{figure}

  To compute $\rho(\gamma)$, we homotope $\gamma$ to follow these edges,
  then multiply together the elements in the order given by $\gamma$.
  When the orientation of $\gamma$ disagrees with the orientation of the
  edge, we use the inverse of the attached element.

  All that remains is to give a formula for the $\alpha_{ij}$ and
  $\beta_{ijk}$.

  \begin{itemize}
  \item
    We start with an element in $\seedAtorus_Q$; the
    $\mathcal{A}$-coordinates of $Q$.
  \item
    Since we have a Fock--Goncharov coordinate structure, we may apply
    $(\conftoquiv)^{-1}$. This gives an explicit element of
    $\Conf_3^{\ast}(G/N_{+})$. It is most conveniently expressed as
    $(h_1, h_2, h_3, u)$.
  \item
    Set
    \begin{align*}
      \alpha_{01} &= \lifttoG{w_0} h_1 & u_0 &= u & \beta_{201} &= \left( w_0(h_1) h_2 \right)^{-1} (\Psi \Phi \Psi)(u_0) \left( w_0(h_1) h_2 \right) \\
      \alpha_{12} &= \lifttoG{w_0} h_2 & u_1 &= \left( w_0(h_1) h_2 \right)^{-1} (\Phi \Psi)^2(u_0) \left( w_0(h_1) h_2 \right) & \beta_{120} &= \left( w_0(h_2) h_3 \right)^{-1} (\Psi \Phi \Psi)(u_1) \left( w_0(h_2) h_3 \right) \\
      \alpha_{20} &= \lifttoG{w_0} h_3 & u_2 &= \left( w_0(h_1) h_2 \right)^{-1} (\Phi \Psi)^2(u_1) \left( w_0(h_1) h_2 \right) & \beta_{012} &= \left( w_0(h_3) h_1 \right)^{-1} (\Psi \Phi \Psi)(u_2) \left( w_0(h_3) h_1 \right).
    \end{align*}
  \end{itemize}

  \begin{remark}
    Since all $\beta_{ijk}$ are assigned elements of $N_{+}$, and any
    peripheral $\gamma$ can be homotoped to follow only short edges, any
    $\rho$ constructed this way is boundary-unipotent.
  \end{remark}

  Also, given a point in $\mathcal{X}^{+}_{G,\Sigma}$, we can
  reconstruct a boundary-borel representation $\rho$ (up to conjugation)
  as follows. We refer to \cite[Section~6]{fockgoncharov2006},
  \cite[Section~9]{ggz2015} for more details.

  Again, we start with an ideal triangulation $\triang$ of $\Sigma$,
  together with a quiver $Q$ and seed torus $\seedXtorus_Q$ for each
  triangle in $\triang$. Instead of triangulating, however, we doubly
  triangulate, as in \cref{fig:doubly-truncated-simplex}. To compute
  $\rho(\gamma)$, we homotope $\gamma$ to follow the edges, then
  multiply $\alpha_{ij}$, $\beta_{ijk}$, or $\gamma_{ijk}$ (or their
  inverses) as appropriate.

  \begin{itemize}
  \item
    We start with elements of $\seedXtorus_Q$.
  \item
    Taking $(\conftoquiv)^{-1}$ gives elements of
    $\Conf_3^{\ast}(G/B_{+})$ for each triangle. To compute edges at a
    particular triangle, assume its flags are $(g_0 B_{+}, g_1 B_{+},
    g_2 B_{+}) = (B_{+}, \lifttoG{w_0} B_{+}, u_x B_{+})$, and that the
    neighboring triangles have vertices $(0, 1, 3)$, $(1, 2, 4)$, and
    $(2, 0, 5)$ (with cosets $g_3 B_{+}$, $g_4 B_{+}$, and $g_5 B_{+}$).
  \item
    Recall $n(u)$ of \cref{rem:canonical-form-for-B-cosets}. The element
    assignments are
    \begin{align*}
      \alpha_{01} &= \lifttoG{w_0} & \alpha_{12} &= \lifttoG{w_0} & \alpha_{20} &= \lifttoG{w_0} \\
      \beta_{012} &= \Phi^{-1}(u_x) & \beta_{120} &= \lifttoG{w_0}^{-1} u_x^{-1} \lifttoG{w_0} & \beta_{201} &= \left( \Phi^{-1}(u_x^{-1}) \right)^{-1} \\
      \gamma_{012} &= n(g_3) & \gamma_{120} &= n([\lifttoG{w_0} u_x^{-1} g_4]_{-}) & \gamma_{201} &= n([\lifttoG{w_0} [\lifttoG{w_0}^{-1} u_x^{-1}]_{-}^{-1} \lifttoG{w_0}^{-1} u_x^{-1} g_5 ]_{-}]_{-}) \\
      \gamma_{102} &= \left( n([\lifttoG{w_0} g_3]_{-}) \right)^{-1} & \gamma_{210} &= \left( n(u_x^{-1} g_4) \right)^{-1} & \gamma_{021} &= \left( n([\lifttoG{w_0} [\lifttoG{w_0}^{-1} u_x]_{-} \lifttoG{w_0}^{-1} g_5 ]_{-}) \right)^{-1} \\
    \end{align*}
  \end{itemize}

  \begin{remark}
    All middle edges are assigned elements of $N_{+}$, and all short
    edges are assigned elements of $H$. Therefore any peripheral loop
    lies in $B_{+}$ and $\rho$ is boundary-borel.
  \end{remark}

  \begin{remark}
    Since $n(u)$ may require taking square roots, the field over which
    the $\mathcal{X}$-coordinates are defined may need to be extended in
    order to define $\rho$.
  \end{remark}

  Since we can now explicitly interpret points in these moduli spaces as
  representations, we describe the positivity conditions of higher
  Teichmüller spaces.

  \begin{definition}
    \label{def:positive-structure}
    Let $N_{+}$ be a maximal unipotent subgroup of $G$. Then
    \cref{def:factorization-coords-xi} gives coordinate charts on
    $N_{+}$ for each presentation of $w_0$. These charts are the
    \defined{positive structure} on $N_{+}$. The elements of $N_{+}$
    which have coordinates entirely in $\mathbb{R}_{>0}$ for all these
    charts are the \defined{positive part} of $N_{+}$.

    The maps $\chi_{\fundamentalweight_k}$ of
    \cref{def:root-space-decomposition} are also a \defined{positive
      structure} on the maximal torus $H$ of $G$. Again, the
    \defined{positive part} consists of the elements which have
    coordinates in $\mathbb{R}_{>0}$ by the positive structure.

    By \cite[Section~8.1]{fockgoncharov2006}, together with the
    canonical form of \cref{lem:canonical-form-for-conf_3_star}, these
    provide a \defined{positive structure}, and a \defined{positive
      part}, of $\Conf_3^{\ast}(G/N_{+})$. There is also a completely
    analogous positive structure on $\Conf_3^{\ast}(G/B_{+})$ following
    \cite[Section~5.5]{fockgoncharov2006}.
  \end{definition}

  \begin{remark}
    Generalized minors are compatible with these positive structures, as
    from \cite[Theorem~5.1]{fockgoncharov2006}. Therefore, positive
    points in the flag varieties $\Conf_3^{\ast}(G/K)$ are exactly those
    points for which all $\mathcal{A}$- and $\mathcal{X}$-coordinates in
    the cluster ensemble are in $\mathbb{R}_{>0}$.
  \end{remark}

  \begin{definition}
    \label{def:teich-plus}
    The space $\teich^{+}(\Sigma)$ is Teichmüller space, together with
    choices of orientation for non-cuspidal boundary components. For a
    surface $\Sigma$ with $n \ge 0$ boundary circles $b_1, \dotsc, b_n$,
    define \[ \teich^{+}(\Sigma) = \set{(p, \epsilon_1, \dotsc,
        \epsilon_n) : p \in \teich(S), \epsilon_i = \pm 1}, \] where
    $\epsilon_i$ is positive to denote that the chosen orientation of
    $b_i$ agrees with that naturally induced by $\Sigma$.
  \end{definition}

  \begin{remark}
    When $\Sigma$ has only cuspidal boundary components,
    $\teich^{+}(\Sigma) = \teich(\Sigma)$.
  \end{remark}

  \begin{definition}
    When $G$ is centerless, we may repeat the above construction, but
    take the $\mathcal{X}$-coordinates of the cluster ensemble. This
    gives the moduli space $\mathcal{X}_{G,\Sigma}$ of \defined{framed,
      $G$-local systems} on $\Sigma$ by
    \cite[Section~2.1]{fockgoncharov2006}. The $\mathbb{R}_{>0}$ points
    also give $\mathcal{X}^{+}_{G,\Sigma}$. This space is identified
    with $\teich^{+}(\Sigma)$.

    That $\teich^{+}(\Sigma)$ appears instead of $\teich(\Sigma)$ is
    rather a technicality. Restricting our attention to orderings of
    ideal triangulations that agree with the surface's natural
    orientation restrictrs $\teich^{+}(\Sigma)$ to a set we can
    canonically identify with $\teich(\Sigma)$.
  \end{definition}

  \begin{definition}
    When $G$ is simply connected, choose an ordered, oriented
    triangulation of $\Sigma$. Associate a copy of $Q$ to each triangle,
    and amalgamate all the quivers together by identifying shared edge
    vertices between triangles. The $\mathcal{A}$-coordinates of the
    cluster ensemble form the moduli space $\mathcal{A}_{G,\Sigma}$.

    By \cite[Section~8.6]{fockgoncharov2006} the $\mathbb{C}$-points of
    this moduli space parameterize \defined{twisted, decorated
      representations} into $G$. The $\mathbb{R}_{>0}$-points give
    $\mathcal{A}^{+}_{G,\Sigma}$, a \defined{higher Teichmüller space}.
    These correspond to flag varieties such that all generalized minors
    are strictly positive.
  \end{definition}

  These moduli spaces have further interpretations and properties,
  explored in \cite{fockgoncharov2006}, \cite{fockgoncharov2006_2},
  \cite{fockgoncharov2009}, \cite{goncharovshen2018}, etc.

  \subsection{Regarding \texorpdfstring{$3$}{3}-manifolds}

  The process described can be applied to $3$-manifolds as well as
  surfaces. We omit all details. An element of $\Conf_4^{\ast}(G/K)$ is
  attached to each tetrahedron, and the coordinates for these flag
  varieties are encoded on a quiver for each face of the tetrahedron.

  As in \cref{def:gluing-configurations}, only two copies of
  $\Conf_3^{\ast}(G/K)$ are necessary to define an element in
  $\Conf_4^{\ast}(G/K)$, so any two copies of $Q$ should determine all
  the coordinates of the tetrahedron. Therefore, the coordinates on any
  pair of faces determine the coordinates on the other pair. This
  relation is given by the mutation $\muflip$.

  Finally, as coordinates are identified along glued edges in the
  surface case, in the $3$-manifold case they are glued along faces, as
  in \cref{fig:tetrahedra-with-quivers}. When reconstructing $\rho$, the
  path $\gamma$ is homotoped as before and the same elements on each
  segment are used.

  \begin{figure}[h]
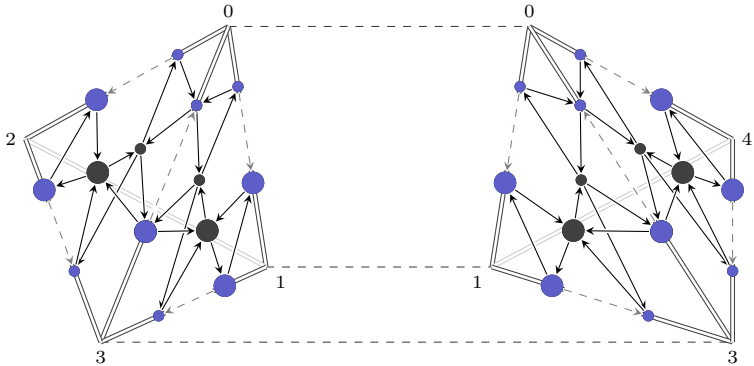

    \centering
    \includestandalone[mode=image|tex]{fig/tetrahedra-with-quivers}

    \caption{Two tetrahedra with a face identification. The quivers
      $Q_{C_2}$ on the front two faces of each tetrahedron are shown.
      \label{fig:tetrahedra-with-quivers}
    }
  \end{figure}

  The $\mathcal{A}$-coordinate variety constructed this way is called
  the \defined{Ptolemy variety}, see \cite{ggz2015}, \cite{gtz2015},
  \cite{zickert2016} for details. The variety constructed by the
  $\mathcal{X}$-coordinates will be the analogue of the \defined{shape
    coordinates} of \cite{ggz2015}, and the defining equations will
  generalize Thurston's gluing equations.

  Defining varieties by quivers in this fashion allows efficient
  computation, and databases have been constructed of Ptolemy varieties
  for large numbers of triangulations, See \cite{unhyperbolic}.

  \section{Main Result: Fock--Goncharov coordinate structures for
    non-\texorpdfstring{$A_n$}{A\_n}}

  \label{sec:main-result}

  We now present an algorithm for constructing Fock--Goncharov
  coordinate structures for a simple Lie group $G$ over $\mathbb{C}$,
  thus satisfying most of \cref{thm:fgcs-exists}. We delay all proofs
  until \cref{sec:proof}.

  We make one slight demand: the ability to present $w_0$ via
  \cref{fac:w0-construction}. Luckily, this demand is satisfied as long
  as the Coxeter number $h$ is even, equivalently $G \ne A_{2n}$. This
  is acceptable, since the $A_n$ case has a particularly nice form which
  has been the subject of considerable study, as in
  \cite{fockgoncharov2006}, \cite{gtz2015}. See \cref{sec:An-case} for a
  review of the results in language consistent with this section.

  \subsection{Overview}

  Here we loosely describe the algorithm for constructing $Q$, $\murot$,
  $\muflip$, and $\conftoquiv$.

  The quiver $Q$ will be divided into an interior and three edges.
  Coordinates at each vertex will be assigned generalized minors of
  elements in the canonical form $(h_1, h_2, h_3, u)$ of an element of
  $\Conf_3^{\ast}$. Each edge should contain information for some $h_i$,
  and the interior should contain information for $u$.

  One minor for each simple root determines an element of $H$; the
  coordinates on an edge will be of the form $\genminor{}{j}(h_i)$.
  Accordingly, the edges of the triangle will be quivers of Dynkin type.
  The interior vertices will be given by generalized minors of the form
  $\genminor{w_k}{\word{i}_k}(u)$. Laying these out to satisfy
  \cite[Theorem~1.17]{fominzelevinsky1999} follows the algorithm of
  \cite[Section~2]{fominzelevinsky2005}. This, together with our choice
  of presentation for $w_0$, means the interior will be a rectangular
  grid of vertices. Recall that these take the names $A_{\bullet}$,
  $B_{\bullet}$, $C_{\bullet}$, and $v_{i,j}$ by
  \cref{def:triangular-quiver}.

  The algorithm of Fomin--Zelevinsky describes a rectangle. Two of the
  rectangle's edges will be edges in the triangle. We call this $Q_0$.
  The chief difficulty of constructing $Q$ is introducing the third edge
  to $Q_0$. To do this, we look ahead to $\murot$. Since that quiver
  mutation must rotate $Q$, and $Q$ and $Q_0$ share a mutable portion,
  we can rotate $Q_0$. And since $Q_0$ already has two of $Q$'s three
  edges, the third edge of $Q$ can be deduced from the action of
  $\murot$ on $Q_0$.

  The flip mutation $\muflip$ is built in the same way that $\murot$ is:
  repeated application of \cite[Theorem~1.17]{fominzelevinsky1999}. We
  need a bit of compensation before and after because the amalgamation
  does not quite line up the quivers as necessary.

  \begin{remark}
    Our construction will produce $\murot$ and $\muflip$ as compositions
    of smaller mutations. They contain smaller mutations ($\murottwist$
    and $\mufliptwist$) which may be more useful for certain
    applications. See \cref{sec:what-does-the-twist-do}.
  \end{remark}

  Finally, the map $\conftoquiv$ is the generalized minors, as described
  above, together with a monomial compensation. At each step of $\murot$
  and $\muflip$, \cite[Theorem~1.17]{fominzelevinsky1999} takes the form
  described in \cref{sec:thm-1.17} \[ (\text{initial}) \cdot
    (\text{final}) = (\text{left}) \cdot (\text{right}) + \prod_{j \ne
      \word{i}_k} \left( \text{in same column} \right). \] However, at
  the edges of $Q$, when $j$ is very low or very high, the left or right
  elements might be trivially $1$ according to the theorem. In our
  quivers, however, these elements are non-trivial, given by frozen
  vertices.

  Therefore, we have to balance the equation of the theorem. To do so,
  we treat the edge coordinates in the mutation relation as ``Extra''
  information, and include this extra information in our definition of
  interior coordinates. Then the quiver mutations take forms similar to
  the following: \[ {\color{blue!60!darkgray!70} \text{Extra} } \cdot
    (\text{initial}) \cdot (\text{final}) =
    \overbrace{({\color{blue!60!darkgray!70}\text{Extra} } \cdot
      1)}^{\text{left}} \cdot (\text{right}) +
    {\color{blue!60!darkgray!70} \text{Extra} } \cdot \prod_{j \ne
      \word{i}_k} \left( \text{in same column} \right). \] Since the
  extraneous factors appear in every term, the desired result still
  holds by the theorem. This is the effect of the monomial map
  $\monomialmap$ in $\conftoquiv$.

  \begin{example}
    \label{ex:dynkin-quiver-f4}
    Based on the Dynkin diagram \raisebox{-0.1ex}{\includestandalone[mode=image|tex,height=1.5ex]{fig/dynkin-F4}}{} for $F_4$, the
    following is a well-rooted tree-like Dynkin quiver (following
    \cref{def:well-rooted-tree-like-quiver}) for $F_4$:

    {\centering \includestandalone[mode=image|tex]{fig/dynkin-type-F4}

    }
  \end{example}

  \begin{remark}
    \label{rem:sizes-get-switched}
    By \cref{def:dynkin-type}, the nodes with higher weight in the
    quiver correspond to the nodes which are ``smaller'' in the Dynkin
    diagram.
  \end{remark}

  With this, we are ready to begin the construction. Let $D$ be a
  well-rooted tree-like quiver of Dykin type for $G$. Let $c = \set{c_1,
    c_2, \dotsc, c_n}$ be the induced Coxeter element as in
  \cref{def:Coxeter-elt-induced-by-quiver}, and let $w_0$ be the longest
  word, with presentation via \cref{fac:w0-construction}.

  \begin{remark}
    We will abuse notation slightly by writing $s_{i_1} s_{i_2} \dotsb
    s_{i_n}$ as $\set{i_1, i_2, \dotsc, i_n}$.
  \end{remark}

  \subsection{Building the rectangle \texorpdfstring{$Q_0$}{Q0}}

  \label{sec:Q0-construction}

  We desire to construct a quiver that holds coordinates for most of
  $H^3 \times_H N_{-}$. Conveniently, $H \times N_{-} = B_{-}$ is the
  \defined{double Bruhat cell} $G^{w_0,e} = B_{+} w_0 B_{+} \cap B_{-} e
  B_{-}$. We will not need any more information about double Bruhat
  cells, except to note that \cite[Section~2]{fominzelevinsky2005}
  describes an algorithm which accepts $u, v$ and creates a quiver whose
  cluster coordinates are coordinates on $G^{u,v}$. Therefore, we will
  follow this algorithm (with slight modifications) in the special case
  $u = w_0$, $v = e$. Since this simplifies the algorithm greatly, we
  can completely reproduce it here.

  The result will be a rectangle of $(h/2) + 1$ copies of $D$, with the
  first and last copies being frozen, and having half-weight edges. The
  first copy of $D$ will have vertices labeled $\set{B_{i}}_{i \in D}$,
  the last will be labeled with $\set{C_{i}}_{i \in D}$, The middle
  vertices will be labeled $v_{jk}$, where $j$ is the corresponding
  entry in $D$, and $k$ is the distance along the path from
  $B_{\bullet}$ to $C_{\bullet}$.

  Begin by setting $Q_0$ to be $D$. Label the vertices $\set{v_{i0}}_{i
    \in D}$. Let $f_i$ be the ``frontier vertices'' of $Q_0$, with
  initially $f_i = v_{i0}$ for each $i \in D$. For convenience, let $j :
  v_{ij} \mapsto j$, so that we may refer to $j(f_i)$. Also for
  convenience, let $\sigma(f_j, f_k)$ be the weight of the edge between
  $f_j$ and $f_k$ (taking values in $\set{-1, 0, 1}$).

  Now, proceed in order through the letters of $w_0$. For each letter
  $k$, let $j' = j(f_k) + 1$.
  \begin{itemize}
  \item
    Add $v_{kj'}$, a vertex of the same weight as $f_k$, to the left of
    $f_k$.
  \item
    For each $\ell \in D$, add an edge of weight $-\sigma(f_k,
    f_{\ell})$ from $v_{kj'}$ to $f_{\ell}$.
  \item
    Add an edge of weight $1$ from $v_{kj'}$ to $f_k$.
  \item
    Set $f_k$ to $v_{kj'}$.
  \end{itemize}

  When finished, all $f_k$ will be $v_{k(h/2)}$. Rename each $v_{k0}$ to
  $C_k$, each $v_{k(h/2)}$ to $B_k$, and halve the weights of any edges
  if they connect two $B_{\bullet}$ vertices or two $C_{\bullet}$
  vertices.

  \begin{remark}
    At each step, the frontier vertices form a sub-quiver of Dynkin
    type. Further, replacing $v_{ij}$ with $v_{i(j + 1)}$ operates,
    graphically, as a quiver mutation on the frontier.
  \end{remark}

  \begin{example}
    \label{ex:f4-part-1}
    \begin{figure}[h]
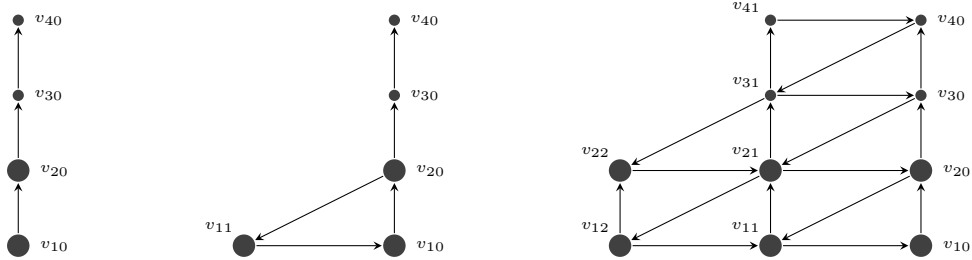

      \centering
      \includestandalone[mode=image|tex]{fig/rect-F4-in-progress}

      \caption{In-progress $Q_0$ for $F_4$ after $0$, $1$, and $6$
        letters of $w_0$.
        \label{fig:rect-F4-in-progress}
      }
    \end{figure}

    \begin{figure}[h]
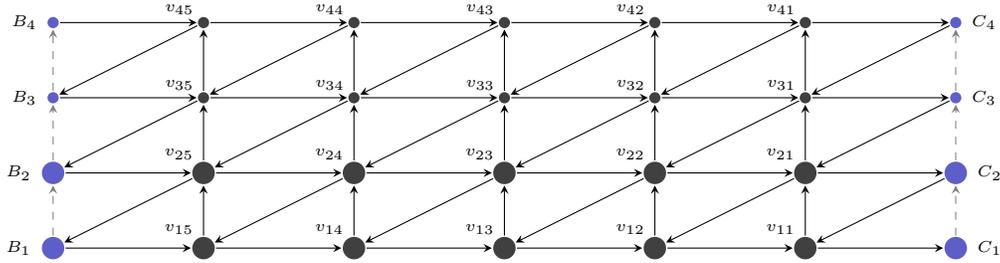

      \centering \includestandalone[mode=image|tex]{fig/rect-F4}

      \caption{Completed $Q_0$ for $F_4$.
        \label{fig:rect-F4}
      }
    \end{figure}

    The Dynkin-type quiver for $G = F_4$ in \cref{ex:dynkin-quiver-f4},
    with $c = \set{1,2,3,4}$ and $\ell = 6$ gives \[ w_0 = \set{1, 2, 3,
        4, 1, 2, 3, 4, 1, 2, 3, 4, 1, 2, 3, 4, 1, 2, 3, 4, 1, 2, 3, 4}.
    \] Carrying out the construction produces the rectangle shown in
    \cref{fig:rect-F4-in-progress,fig:rect-F4}.

    Using \cite{gilles2019}, the initial step of $Q_0$ can be generated
    by

    {\centering \texttt{clav-bfzIII -U -c F4 -v
        "1,2,3,4,1,2,3,4,1,2,3,4,1,2,3,4,1,2,3,4,1,2,3,4" > Q0.clav}.}
  \end{example}

  \subsection{Construction of \texorpdfstring{$\murot$}{mu\_\{rot\}}}

  The most difficult step is modifying $Q_0$ to become triangular by
  adding a third group of vertices, $A_{\bullet}$, also of half-Dynkin
  type. However, since the $A_{\bullet}$, $B_{\bullet}$, and
  $C_{\bullet}$ vertices will be frozen, the $Q_0$ alone is enough to
  describe the mutation sequence $\murot$. It will be composed of three
  main pieces.
  \begin{itemize}
  \item
    The mutation $\murottwist$ transforms coordinates of $\alpha$ to
    coordinates of $\operatorname{rot}(\alpha)$, in some order. The
    component mutation $\mucol(j)$ performs a mutation on column $j$
    that adjusts the minor coordinates by concatenating the Coxeter
    element to each word.
  \item
    The mutation $\muint{O1}$ (``ordering-1'') performs a
    $\sigma_{A_{\ell}}$ action on each row (a Dynkin diagram of
    $A_{\ell}$) using $\murowperm$ (``row $\sigma_{A_{\ell}}$'')
  \item
    The mutation $\muint{O2}$ (``ordering-2'') performs a $\sigma_G$
    action on each column (a Dynkin diagram of $G$) using $\mucolperm$
    (``column $\sigma_G$''), as in
    \cref{rem:sigma_G-action-on-Dynkin-diagram}. When $\sigma_G$ is
    trivial, this can be ignored.
  \end{itemize}

  \begin{definition}
    \label{def:murot-components}
    For any quiver\footnote{We are vague because we will use these
      mutations on $Q_0$ to find $Q$, and then on $Q$ itself as
      $\murot$.} with the same mutable portion as $Q_0$, fix an induced
    Coxeter element $c = \set{c_1, c_2, \dotsc, c_n}$, which admits a
    tree-like partition $\Set{T_1, T_2, \dotsc, T_m}$. Let $\ell =
    \frac{h}{2} - 1$. The sequences $\murot$ and $\murottwist$
    (``twisting rotation'') perform almost the same function, though
    they permute the vertices of the quiver differently.

    \begingroup \allowdisplaybreaks
    \begin{align*}
      \muT(i,j) &= \prod_{z \in T_i} \mu\Set{v_{zj}} \\
      \mucol(j) &= \prod_{i = 1}^{m} \muT(i,j) \\
      \murowperm(i) &= \left[ \prod_{k=1}^{\ell} \prod_{j=0}^{\ell - k} \muT(i, 1 + j) \right] \left[ \prod_{j=1}^{\ell} \muT(i, \ell - j) \right] \\
      \mucolperm(j) &= \left[ \prod_{i=0}^{\ell + 1} \mucol(j) \right] \\
      \murottwist &= \prod_{x=1}^{\ell} \prod_{y = 1}^{\ell + 1 - x} \mucol(y) \\
      \muint{O1} &= \prod_{i=1}^{m} \murowperm(i), \quad \muint{O2} = \begin{cases} \prod_{j=1}^{\ell} \mucolperm(j) & \text{$\sigma_G$ non-trivial} \\ \varnothing & \text{$\sigma_G$ trivial} \end{cases}\\
      \murot &= \left[ \murottwist \right ] \left[ \muint{O1} \right] \left[ \muint{O2} \right]
    \end{align*}
    \endgroup (Recall that, by \cref{not:writing-mutations}, multiplied
    mutations are applied left-to-right, and that products carry an
    ordering.)
  \end{definition}

  \begin{remark}
    An alternate presentation of $\murot$ which has the same effect on
    quivers may be given as roughly ``$\murottwist$ four times'', though
    we will focus on the above definition. \[ \left(\left[
      \prod_{x=1}^{\ell} \prod_{y =1}^{\ell + 1 - x} \mucol(y) \right]
      \left[ \prod_{x=1}^{\ell} \prod_{y=1}^{\ell + 1 - x}
      \mucol(\ell+1-y) \right] \right)^2. \]
  \end{remark}

  \begin{example}
    \label{ex:f4-part-2}
    Continuing from \cref{ex:f4-part-1} gives the following\footnote{The
      elements of the sequence are presented in usual reading order. The
      spacing is to emphasize decomposition into $\mucol$ terms.}:
    \begin{align*}
      \murottwist &= \mu\Set{
        \begin{subarray}{l}
\phantom{
         v_{1 5}, v_{2 5}, v_{3 5}, v_{4 5}, \quad
         v_{1 4}, v_{2 4}, v_{3 4}, v_{4 4}, \quad
         v_{1 3}, v_{2 3}, v_{3 3}, v_{4 3}, \quad
         v_{1 2}, v_{2 2}, v_{3 2}, v_{4 2}, \quad
}
         v_{1 1}, v_{2 1}, v_{3 1}, v_{4 1}, \\
\phantom{
         v_{1 5}, v_{2 5}, v_{3 5}, v_{4 5}, \quad
         v_{1 4}, v_{2 4}, v_{3 4}, v_{4 4}, \quad
         v_{1 3}, v_{2 3}, v_{3 3}, v_{4 3}, \quad
}
         v_{1 2}, v_{2 2}, v_{3 2}, v_{4 2}, \quad
         v_{1 1}, v_{2 1}, v_{3 1}, v_{4 1}, \\
\phantom{
         v_{1 5}, v_{2 5}, v_{3 5}, v_{4 5}, \quad
         v_{1 4}, v_{2 4}, v_{3 4}, v_{4 4}, \quad
}
         v_{1 3}, v_{2 3}, v_{3 3}, v_{4 3}, \quad
         v_{1 2}, v_{2 2}, v_{3 2}, v_{4 2}, \quad
         v_{1 1}, v_{2 1}, v_{3 1}, v_{4 1}, \\
\phantom{
         v_{1 5}, v_{2 5}, v_{3 5}, v_{4 5}, \quad
}
         v_{1 4}, v_{2 4}, v_{3 4}, v_{4 4}, \quad
         v_{1 3}, v_{2 3}, v_{3 3}, v_{4 3}, \quad
         v_{1 2}, v_{2 2}, v_{3 2}, v_{4 2}, \quad
         v_{1 1}, v_{2 1}, v_{3 1}, v_{4 1}, \\
\phantom{
}
         v_{1 5}, v_{2 5}, v_{3 5}, v_{4 5}, \quad
         v_{1 4}, v_{2 4}, v_{3 4}, v_{4 4}, \quad
         v_{1 3}, v_{2 3}, v_{3 3}, v_{4 3}, \quad
         v_{1 2}, v_{2 2}, v_{3 2}, v_{4 2}, \quad
         v_{1 1}, v_{2 1}, v_{3 1}, v_{4 1} \\
        \end{subarray}
      } \\
      \muint{O1} &= \mu\Set{
        \begin{subarray}{l}
w_{41}, w_{42}, w_{43}, w_{44}, w_{45}, \quad
w_{41}, w_{42}, w_{43}, w_{44}, \quad
w_{41}, w_{42}, w_{43}, \quad
w_{41}, w_{42}, \quad
w_{41}, \quad
w_{45}, w_{44}, w_{43}, w_{42}, w_{41}, \\
\\
w_{31}, w_{32}, w_{33}, w_{34}, w_{35}, \quad
w_{31}, w_{32}, w_{33}, w_{34}, \quad
w_{31}, w_{32}, w_{33}, \quad
w_{31}, w_{32}, \quad
w_{31}, \quad
w_{35}, w_{34}, w_{33}, w_{32}, w_{31}, \\
\\
w_{21}, w_{22}, w_{23}, w_{24}, w_{25}, \quad
w_{21}, w_{22}, w_{23}, w_{24}, \quad
w_{21}, w_{22}, w_{23}, \quad
w_{21}, w_{22}, \quad
w_{21}, \quad
w_{25}, w_{24}, w_{23}, w_{22}, w_{21}, \\
\\
w_{11}, w_{12}, w_{13}, w_{14}, w_{15}, \quad
w_{11}, w_{12}, w_{13}, w_{14}, \quad
w_{11}, w_{12}, w_{13}, \quad
w_{11}, w_{12}, \quad
w_{11}, \quad
w_{15}, w_{14}, w_{13}, w_{12}, w_{11} \\
        \end{subarray}
      } \\
      \muint{O2} &= \mu\Set{ }
    \end{align*}
  \end{example}

  \begin{lemma}
    \label{lem:murot-works-for-making-Q}
    $\murot \circ \murot \circ \murot$ induces the identity on $Q_0$.
    Furthermore, $\murot$ induces a graph isomorphism on the mutable
    portion of the $Q_0$ (the sub-quiver containing only the
    $\set{v_{ij}}$ vertices).
  \end{lemma}

  \subsection{The \texorpdfstring{$A_{\bullet}$}{A.} edge}

  \label{sec:the-A-edge}

  With $\murottwist$ in hand, we are ready to construct $A_{\bullet}$.
  Let $Q_1$ be $Q_0$ together with a dummy $A_{\bullet}$, which is of
  Dynkin type $D$ with edges removed. Denote $Q_i' = \murottwist(Q_i)$
  and $Q_i'' = (\murottwist)^{-1}(Q_i)$.

  We would like $Q_1$ to be isomorphic to $Q_1'$ and $Q_1''$ via
  $\varphi$ and $\varphi^{-1}$, with \[ \varphi: \Set{v_{ij} \mapsto
      v_{ij}, \qquad A_i \mapsto C_i, \qquad B_i \mapsto A_i, \qquad C_i
      \mapsto B_i}, \] but unfortunately these are not isomorphisms.
  However, we can correct for this. Let $Q_2$ be the quiver containing
  the vertices of $Q_1$, and with edges defined by \[ \sigma_{Q_2}(v, w)
    = \sum \set{\sigma_{Q_1}(v,w), \sigma_{\varphi^{-1}(Q_1')}(v,w),
      \sigma_{\varphi(Q_1'')}(v,w)}. \] That is, we repeatedly rotate
  $Q_1$, taking the inclusion of all frozen vertices necessary to ensure
  that $\murot$ is of order $3$ on the entire quiver.

  \begin{remark}
    \label{rem:yes-using-murot-is-justified}
    It is not yet obvious that $\murot$ actually does rotate by a third.
    The proof of \cref{lem:fgcs-exists-rot-works-trivial}, however, will
    show that $\murot$ acts with order $3$ on the seed torus for $Q_0$,
    and therefore on the quiver. There is no circular dependency here,
    as that lemma does not depend on the existence of $A_{\bullet}$,
    $B_{\bullet}$, or $C_{\bullet}$.
  \end{remark}

  We therefore define $Q$ to be $Q_2$, and we have constructed $Q$ and
  $\murot$.

  \begin{remark}
    \label{rem:Q-is-visually-triangular}
    To impose \cref{def:triangular-quiver}, we will customarily put
    $A_{\bullet}$ on the $(0,1)$ edge, $B_{\bullet}$ on the $(1,2)$
    edge, and $C_{\bullet}$ on the $(2,0)$ edge, with the $v_{ij}$ as
    face vertices.
  \end{remark}

  \begin{example}
    \label{ex:f4-part-3}

    \begin{figure}
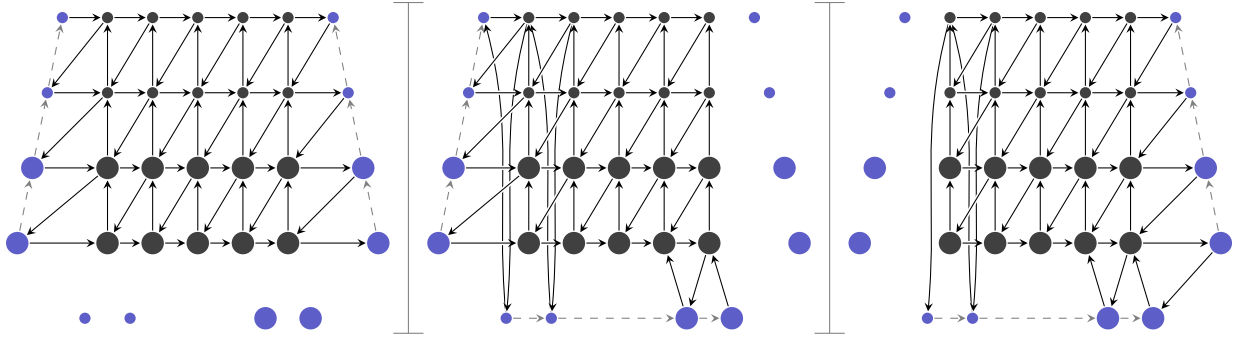

      \centering \includestandalone[mode=image|tex]{fig/F4-Q1s}

      \caption{$Q_1$, $(Q_1')$, and $(Q_1'')$ for $F_4$ (rotated to
        agree with $\id$, $\varphi^{-1}$, and $\varphi$ respectively).
        \label{fig:rect-F4-Q1s}
      }
    \end{figure}

    \begin{figure}
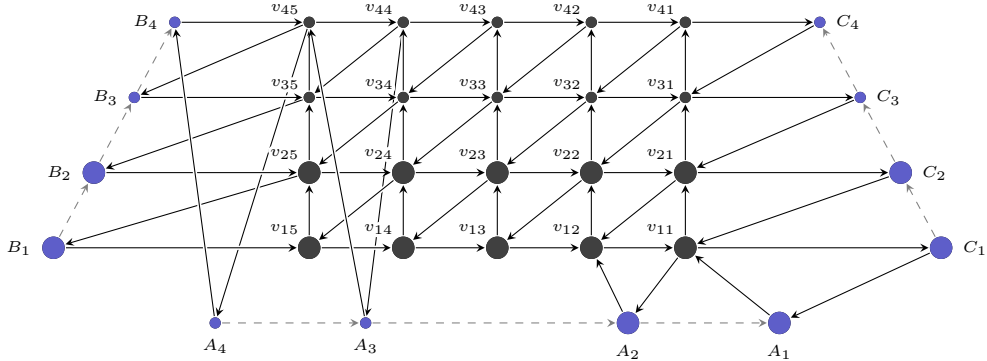

      \centering \includestandalone[mode=image|tex]{fig/QF4}

      \caption{$Q$ for $F_4$.
        \label{fig:QF4}
      }
    \end{figure}

    Continuing from \cref{ex:f4-part-2}, we obtain the quivers of
    \cref{fig:rect-F4-Q1s}. Merging them, we obtain $Q = Q_2$ as in
    \cref{fig:QF4}.

  \end{example}

  \subsection{Construction of \texorpdfstring{$\muflip$}{mu\_\{flip\}}}

  \label{sec:constructing-muflip}

  As in the case of $\murot$, the flip mutation is given by a tedious,
  repetitive sequence which looks like intertwined mutations on each row
  and column. It is again composed of pieces.
  \begin{itemize}
  \item
    The mutation $\muint{P}$ (``Pre-mutation'') rotates the sub-quivers
    on the left and right into a position where all the $w_{i,j}$
    vertices forma rectangle. See \cref{fig:Qdv-naming}.
  \item
    The mutation $\muflipcore$ is the core of the flip. It performs an
    analogous function to $\murottwist$: sequences of $\mucol(j)$
    mutations which adjusts minor coordinates (in the sense of
    \cref{sec:external-results}) in column $j$.
  \item
    The mutation $\muint{O3}$ (``ordering-3'') mirrors the rectangle of
    $w_{i,j}$ coordinates horizontally, again by
    \cref{rem:sigma_G-action-on-Dynkin-diagram}. The mutations
    $\muint{O4}$ and $\muint{O5}$ (``ordering-4'' and ``ordering-5'') do
    the same, but restricted to the left and right triangles. So the
    product $\muint{O3} \muint{O4} \muint{O5}$ switches positions of the
    left and right triangle interiors by translation as a composition of
    reflection.
  \end{itemize}

  To distinguish notation from $Q$, we shall label the edges of $\Qdv$
  by $D_{\bullet}$, $E_{\bullet}$, $F_{\bullet}$, $G_{\bullet}$, and the
  interior, mutable vertices by $w_{ij}$. See \cref{fig:Qdv-naming}.

  For $\mu$ a mutation defined on $Q$, let $\mu^R_{\ast}$ (resp.
  $\mu^L_{\ast}$) be the mutation defined to act on the right (left)
  part of the double quiver. Technically, replace each $v_{ij}$ with
  $w_{i(\ell + 1 - j)}$ (with $w_{i(L + 1 - j)}$) in $\mu_{\bullet}$ to
  obtain $\mu^{R}_{\bullet}$ ($\mu^{L}_{\bullet}$).

  \begin{figure}
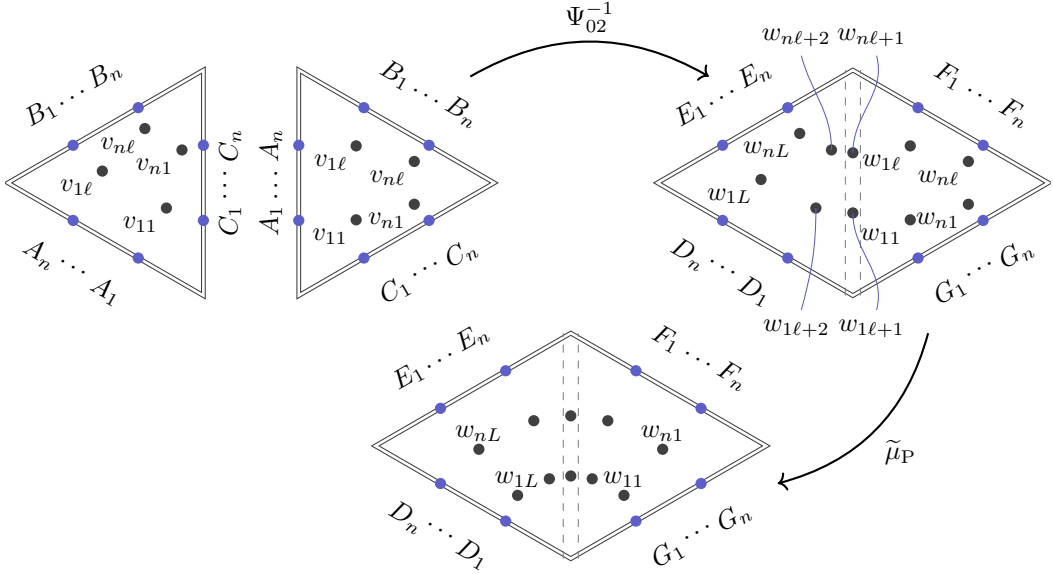

    \centering \includestandalone[mode=image|tex]{fig/Qdv-naming}

    \caption{Vertex naming for $\Qdv$ (with $L = 2\ell + 1$), agreeing
      with \cref{fig:dv-construction} and labeled to induce a mutable
      rectangle in $\muP(\Qdv)$.
      \label{fig:Qdv-naming}
    }
  \end{figure}

  \begin{definition}
    \label{def:muflip-components}
    Construct $\Qdv$ by \cref{def:dv-dh-construction} and labeled as in
    \cref{fig:Qdv-naming}. Recall that $m$ is the number of partitions
    $T_i$ of $c$ from \cref{def:tree-like-quiver}. Some names such as
    $\muT$ are re-used from \cref{def:murot-components} with slightly
    different meanings.

    \begingroup \allowdisplaybreaks
    \begin{align*}
      \muP &= \left( \murottwistL \right)^{-1} \left( \murottwistR \right)^{-1} \tag{See \cref{def:murot-components}} \\
      \muT(i,j) &= \prod_{k \in T_i} \mu\Set{w_{kj}} \quad \text{(possibly empty)} \\
      \mucol(j) &= \prod_{i = 1}^{m} \muT(i,j) \\
      \muflipcore &= \prod_{i=1}^{L} \prod_{j=i}^{L} \mucol(L + i - j) \\
      \murowperm(i, a, b) &= \left[ \prod_{k=0}^{b-a} \prod_{j=0}^{(b-a) - k} \muT(i, a + j) \right] \left[ \prod_{j=0}^{b-a} \muT(i, b - j) \right] \\
      \muint{O3} &= \prod_{i=1}^{m} \murowperm(i, 1, L), \quad \muint{O4} = \prod_{i=1}^{m} \murowperm(i, \ell + 2, L), \quad \muint{O5} = \prod_{i=1}^{m} \murowperm(i, 1, \ell) \\
      \mufliptwist &= \left[ \muP \right] \left[ \muflipcore \right] \\
      \muflip &= \left[ \muP \right] \left[ \muflipcore \right] \left[ \muint{O3} \right] \left[ \muint{O4} \right] \left[ \muint{O5} \right]
    \end{align*}
    \endgroup
  \end{definition}

  \begin{remark}
    Because $\muP$ is composed entirely of rotation mutations, we are
    morally justified in focusing on $\muP(\Qdv)$ instead of $\Qdv$: it
    makes the isomorphism with $\Qdh$ more evident.

    The mutation $\murowperm(i, a, b)$ permutes the
    $\mathcal{A}$-coordinates at those vertices, and is purely used for
    rearranging.
  \end{remark}

  \begin{example}
    \label{ex:f4-part-4}

    \begin{figure}
      \centering \includestandalone[mode=image|tex]{fig/Q0-2F4}

      \caption{$\muP(\Qdv)$ for $F_4$. We apologize for the hexadecimal
        notation; rows have $>9$ vertices.
        \label{fig:Q0-2F4}
      }
    \end{figure}

    Continuing from \cref{ex:f4-part-3} constructs $\Qdv$ such that
    $\muP(\Qdv)$ is as in \cref{fig:Q0-2F4}. The sequence of mutations
    defining the flip is given in \cref{fig:F4-muflip-example}.

    \begin{figure}[h]
      \begin{align*}
        \muP &= \mu\Set{
          \begin{subarray}{l}
\phantom{
         w_{47}, w_{37}, w_{27}, w_{17}, \quad
         w_{48}, w_{38}, w_{28}, w_{18}, \quad
         w_{49}, w_{39}, w_{29}, w_{19}, \quad
         w_{4a}, w_{3a}, w_{2a}, w_{1a}, \quad
}
         w_{4b}, w_{3b}, w_{2b}, w_{1b}, \\
\phantom{
         w_{47}, w_{37}, w_{27}, w_{17}, \quad
         w_{48}, w_{38}, w_{28}, w_{18}, \quad
         w_{49}, w_{39}, w_{29}, w_{19}, \quad
}
         w_{4a}, w_{3a}, w_{2a}, w_{1a}, \quad
         w_{4b}, w_{3b}, w_{2b}, w_{1b}, \\
\phantom{
         w_{47}, w_{37}, w_{27}, w_{17}, \quad
         w_{48}, w_{38}, w_{28}, w_{18}, \quad
}
         w_{49}, w_{39}, w_{29}, w_{19}, \quad
         w_{4a}, w_{3a}, w_{2a}, w_{1a}, \quad
         w_{4b}, w_{3b}, w_{2b}, w_{1b}, \\
\phantom{
         w_{47}, w_{37}, w_{27}, w_{17}, \quad
}
         w_{48}, w_{38}, w_{28}, w_{18}, \quad
         w_{49}, w_{39}, w_{29}, w_{19}, \quad
         w_{4a}, w_{3a}, w_{2a}, w_{1a}, \quad
         w_{4b}, w_{3b}, w_{2b}, w_{1b}, \\
\phantom{
}
         w_{47}, w_{37}, w_{27}, w_{17}, \quad
         w_{48}, w_{38}, w_{28}, w_{18}, \quad
         w_{49}, w_{39}, w_{29}, w_{19}, \quad
         w_{4a}, w_{3a}, w_{2a}, w_{1a}, \quad
         w_{4b}, w_{3b}, w_{2b}, w_{1b}, \\
\phantom{
         w_{41}, w_{31}, w_{21}, w_{11}, \quad
         w_{42}, w_{32}, w_{22}, w_{12}, \quad
         w_{43}, w_{33}, w_{23}, w_{13}, \quad
         w_{44}, w_{34}, w_{24}, w_{14}, \quad
}
         w_{45}, w_{35}, w_{25}, w_{15}, \\
\phantom{
         w_{41}, w_{31}, w_{21}, w_{11}, \quad
         w_{42}, w_{32}, w_{22}, w_{12}, \quad
         w_{43}, w_{33}, w_{23}, w_{13}, \quad
}
         w_{44}, w_{34}, w_{24}, w_{14}, \quad
         w_{45}, w_{35}, w_{25}, w_{15}, \\
\phantom{
         w_{41}, w_{31}, w_{21}, w_{11}, \quad
         w_{42}, w_{32}, w_{22}, w_{12}, \quad
}
         w_{43}, w_{33}, w_{23}, w_{13}, \quad
         w_{44}, w_{34}, w_{24}, w_{14}, \quad
         w_{45}, w_{35}, w_{25}, w_{15}, \\
\phantom{
         w_{41}, w_{31}, w_{21}, w_{11}, \quad
}
         w_{42}, w_{32}, w_{22}, w_{12}, \quad
         w_{43}, w_{33}, w_{23}, w_{13}, \quad
         w_{44}, w_{34}, w_{24}, w_{14}, \quad
         w_{45}, w_{35}, w_{25}, w_{15}, \\
\phantom{
}
         w_{41}, w_{31}, w_{21}, w_{11}, \quad
         w_{42}, w_{32}, w_{22}, w_{12}, \quad
         w_{43}, w_{33}, w_{23}, w_{13}, \quad
         w_{44}, w_{34}, w_{24}, w_{14}, \quad
         w_{45}, w_{35}, w_{25}, w_{15} \\
          \end{subarray}
        } \\
        \muflipcore &= \mu\Set{
          \begin{subarray}{l}
w_{1b}, w_{2b}, w_{3b}, w_{4b}, w_{1a}, w_{2a}, w_{3a}, w_{4a}, w_{19}, w_{29}, w_{39}, w_{49}, w_{18}, w_{28}, w_{38}, w_{48}, w_{17}, w_{27}, w_{37}, w_{47}, w_{16}, w_{26}, w_{36}, w_{46}, \\
        \qquad \qquad w_{15}, w_{25}, w_{35}, w_{45}, w_{14}, w_{24}, w_{34}, w_{44}, w_{13}, w_{23}, w_{33}, w_{43}, w_{12}, w_{22}, w_{32}, w_{42}, w_{11}, w_{21}, w_{31}, w_{41}, \\
w_{1b}, w_{2b}, w_{3b}, w_{4b}, w_{1a}, w_{2a}, w_{3a}, w_{4a}, w_{19}, w_{29}, w_{39}, w_{49}, w_{18}, w_{28}, w_{38}, w_{48}, w_{17}, w_{27}, w_{37}, w_{47}, w_{16}, w_{26}, w_{36}, w_{46}, \\
        \qquad \qquad w_{15}, w_{25}, w_{35}, w_{45}, w_{14}, w_{24}, w_{34}, w_{44}, w_{13}, w_{23}, w_{33}, w_{43}, w_{12}, w_{22}, w_{32}, w_{42}, \\
w_{1b}, w_{2b}, w_{3b}, w_{4b}, w_{1a}, w_{2a}, w_{3a}, w_{4a}, w_{19}, w_{29}, w_{39}, w_{49}, w_{18}, w_{28}, w_{38}, w_{48}, w_{17}, w_{27}, w_{37}, w_{47}, w_{16}, w_{26}, w_{36}, w_{46}, \\
        \qquad \qquad w_{15}, w_{25}, w_{35}, w_{45}, w_{14}, w_{24}, w_{34}, w_{44}, w_{13}, w_{23}, w_{33}, w_{43}, \\
w_{1b}, w_{2b}, w_{3b}, w_{4b}, w_{1a}, w_{2a}, w_{3a}, w_{4a}, w_{19}, w_{29}, w_{39}, w_{49}, w_{18}, w_{28}, w_{38}, w_{48}, w_{17}, w_{27}, w_{37}, w_{47}, w_{16}, w_{26}, w_{36}, w_{46}, \\
        \qquad \qquad w_{15}, w_{25}, w_{35}, w_{45}, w_{14}, w_{24}, w_{34}, w_{44}, \\
w_{1b}, w_{2b}, w_{3b}, w_{4b}, w_{1a}, w_{2a}, w_{3a}, w_{4a}, w_{19}, w_{29}, w_{39}, w_{49}, w_{18}, w_{28}, w_{38}, w_{48}, w_{17}, w_{27}, w_{37}, w_{47}, w_{16}, w_{26}, w_{36}, w_{46}, \\
        \qquad \qquad w_{15}, w_{25}, w_{35}, w_{45}, \\
w_{1b}, w_{2b}, w_{3b}, w_{4b}, w_{1a}, w_{2a}, w_{3a}, w_{4a}, w_{19}, w_{29}, w_{39}, w_{49}, w_{18}, w_{28}, w_{38}, w_{48}, w_{17}, w_{27}, w_{37}, w_{47}, w_{16}, w_{26}, w_{36}, w_{46}, \\
w_{1b}, w_{2b}, w_{3b}, w_{4b}, w_{1a}, w_{2a}, w_{3a}, w_{4a}, w_{19}, w_{29}, w_{39}, w_{49}, w_{18}, w_{28}, w_{38}, w_{48}, w_{17}, w_{27}, w_{37}, w_{47}, \\
w_{1b}, w_{2b}, w_{3b}, w_{4b}, w_{1a}, w_{2a}, w_{3a}, w_{4a}, w_{19}, w_{29}, w_{39}, w_{49}, w_{18}, w_{28}, w_{38}, w_{48}, \\
w_{1b}, w_{2b}, w_{3b}, w_{4b}, w_{1a}, w_{2a}, w_{3a}, w_{4a}, w_{19}, w_{29}, w_{39}, w_{49}, \\
w_{1b}, w_{2b}, w_{3b}, w_{4b}, w_{1a}, w_{2a}, w_{3a}, w_{4a}, \\
w_{1b}, w_{2b}, w_{3b}, w_{4b}, \\
          \end{subarray}
        } \\
        \muint{O3} &= \mu\Set{
          \begin{subarray}{l}
w_{41}, w_{42}, w_{43}, w_{44}, w_{45}, w_{46}, w_{47}, w_{48}, w_{49}, w_{4a}, w_{4b}, \\
w_{41}, w_{42}, w_{43}, w_{44}, w_{45}, w_{46}, w_{47}, w_{48}, w_{49}, w_{4a}, \\
w_{41}, w_{42}, w_{43}, w_{44}, w_{45}, w_{46}, w_{47}, w_{48}, w_{49}, \\
w_{41}, w_{42}, w_{43}, w_{44}, w_{45}, w_{46}, w_{47}, w_{48}, \\
w_{41}, w_{42}, w_{43}, w_{44}, w_{45}, w_{46}, w_{47}, \\
w_{41}, w_{42}, w_{43}, w_{44}, w_{45}, w_{46}, \\
w_{41}, w_{42}, w_{43}, w_{44}, w_{45}, \\
w_{41}, w_{42}, w_{43}, w_{44}, \\
w_{41}, w_{42}, w_{43}, \\
w_{41}, w_{42}, \\
w_{41}, \\
w_{4b}, w_{4a}, w_{49}, w_{48}, w_{47}, w_{46}, w_{45}, w_{44}, w_{43}, w_{42}, w_{41}, \\
\\
w_{31}, w_{32}, w_{33}, w_{34}, w_{35}, w_{36}, w_{37}, w_{38}, w_{39}, w_{3a}, w_{3b}, \\
w_{31}, w_{32}, w_{33}, w_{34}, w_{35}, w_{36}, w_{37}, w_{38}, w_{39}, w_{3a}, \\
w_{31}, w_{32}, w_{33}, w_{34}, w_{35}, w_{36}, w_{37}, w_{38}, w_{39}, \\
w_{31}, w_{32}, w_{33}, w_{34}, w_{35}, w_{36}, w_{37}, w_{38}, \\
w_{31}, w_{32}, w_{33}, w_{34}, w_{35}, w_{36}, w_{37}, \\
w_{31}, w_{32}, w_{33}, w_{34}, w_{35}, w_{36}, \\
w_{31}, w_{32}, w_{33}, w_{34}, w_{35}, \\
w_{31}, w_{32}, w_{33}, w_{34}, \\
w_{31}, w_{32}, w_{33}, \\
w_{31}, w_{32}, \\
w_{31}, \\
w_{3b}, w_{3a}, w_{39}, w_{38}, w_{37}, w_{36}, w_{35}, w_{34}, w_{33}, w_{32}, w_{31}, \\
\\
w_{21}, w_{22}, w_{23}, w_{24}, w_{25}, w_{26}, w_{27}, w_{28}, w_{29}, w_{2a}, w_{2b}, \\
w_{21}, w_{22}, w_{23}, w_{24}, w_{25}, w_{26}, w_{27}, w_{28}, w_{29}, w_{2a}, \\
w_{21}, w_{22}, w_{23}, w_{24}, w_{25}, w_{26}, w_{27}, w_{28}, w_{29}, \\
w_{21}, w_{22}, w_{23}, w_{24}, w_{25}, w_{26}, w_{27}, w_{28}, \\
w_{21}, w_{22}, w_{23}, w_{24}, w_{25}, w_{26}, w_{27}, \\
w_{21}, w_{22}, w_{23}, w_{24}, w_{25}, w_{26}, \\
w_{21}, w_{22}, w_{23}, w_{24}, w_{25}, \\
w_{21}, w_{22}, w_{23}, w_{24}, \\
w_{21}, w_{22}, w_{23}, \\
w_{21}, w_{22}, \\
w_{21}, \\
w_{2b}, w_{2a}, w_{29}, w_{28}, w_{27}, w_{26}, w_{25}, w_{24}, w_{23}, w_{22}, w_{21}, \\
\\
w_{11}, w_{12}, w_{13}, w_{14}, w_{15}, w_{16}, w_{17}, w_{18}, w_{19}, w_{1a}, w_{1b}, \\
w_{11}, w_{12}, w_{13}, w_{14}, w_{15}, w_{16}, w_{17}, w_{18}, w_{19}, w_{1a}, \\
w_{11}, w_{12}, w_{13}, w_{14}, w_{15}, w_{16}, w_{17}, w_{18}, w_{19}, \\
w_{11}, w_{12}, w_{13}, w_{14}, w_{15}, w_{16}, w_{17}, w_{18}, \\
w_{11}, w_{12}, w_{13}, w_{14}, w_{15}, w_{16}, w_{17}, \\
w_{11}, w_{12}, w_{13}, w_{14}, w_{15}, w_{16}, \\
w_{11}, w_{12}, w_{13}, w_{14}, w_{15}, \\
w_{11}, w_{12}, w_{13}, w_{14}, \\
w_{11}, w_{12}, w_{13}, \\
w_{11}, w_{12}, \\
w_{11}, \\
w_{1b}, w_{1a}, w_{19}, w_{18}, w_{17}, w_{16}, w_{15}, w_{14}, w_{13}, w_{12}, w_{11}, \\
          \end{subarray}
        },
        \substack{
          \muint{O4} = \mu\Set{
            \begin{subarray}{l}
w_{47}, w_{48}, w_{49}, w_{4a}, w_{4b}, \\
w_{47}, w_{48}, w_{49}, w_{4a}, \\
w_{47}, w_{48}, w_{49}, \\
w_{47}, w_{48}, \\
w_{47}, \\
w_{4b}, w_{4a}, w_{49}, w_{48}, w_{47}, \\
\\
w_{37}, w_{38}, w_{39}, w_{3a}, w_{3b}, \\
w_{37}, w_{38}, w_{39}, w_{3a}, \\
w_{37}, w_{38}, w_{39}, \\
w_{37}, w_{38}, \\
w_{37}, \\
w_{3b}, w_{3a}, w_{39}, w_{38}, w_{37}, \\
\\
w_{27}, w_{28}, w_{29}, w_{2a}, w_{2b}, \\
w_{27}, w_{28}, w_{29}, w_{2a}, \\
w_{27}, w_{28}, w_{29}, \\
w_{27}, w_{28}, \\
w_{27}, \\
w_{2b}, w_{2a}, w_{29}, w_{28}, w_{27}, \\
\\
w_{17}, w_{18}, w_{19}, w_{1a}, w_{1b}, \\
w_{17}, w_{18}, w_{19}, w_{1a}, \\
w_{17}, w_{18}, w_{19}, \\
w_{17}, w_{18}, \\
w_{17}, \\
w_{1b}, w_{1a}, w_{19}, w_{18}, w_{17}, \\
            \end{subarray}
          } \\
          \muint{O5} = \mu\Set{
            \begin{subarray}{l}
w_{41}, w_{42}, w_{43}, w_{44}, w_{45}, \\
w_{41}, w_{42}, w_{43}, w_{44}, \\
w_{41}, w_{42}, w_{43}, \\
w_{41}, w_{42}, \\
w_{41}, \\
w_{45}, w_{44}, w_{43}, w_{42}, w_{41}, \\
\\
w_{31}, w_{32}, w_{33}, w_{34}, w_{35}, \\
w_{31}, w_{32}, w_{33}, w_{34}, \\
w_{31}, w_{32}, w_{33}, \\
w_{31}, w_{32}, \\
w_{31}, \\
w_{35}, w_{34}, w_{33}, w_{32}, w_{31}, \\
\\
w_{21}, w_{22}, w_{23}, w_{24}, w_{25}, \\
w_{21}, w_{22}, w_{23}, w_{24}, \\
w_{21}, w_{22}, w_{23}, \\
w_{21}, w_{22}, \\
w_{21}, \\
w_{25}, w_{24}, w_{23}, w_{22}, w_{21}, \\
\\
w_{11}, w_{12}, w_{13}, w_{14}, w_{15}, \\
w_{11}, w_{12}, w_{13}, w_{14}, \\
w_{11}, w_{12}, w_{13}, \\
w_{11}, w_{12}, \\
w_{11}, \\
w_{15}, w_{14}, w_{13}, w_{12}, w_{11}, \\
            \end{subarray}
          } \\
        } \\
      \end{align*}
      \caption{The flip for $F_4$.
        \label{fig:F4-muflip-example}
      }
    \end{figure}
  \end{example}

  \begin{remark}
    At intermediate stages of $\murot$ and $\muflip$, the quiver may
    contain edges with weights higher than $1$. This is one of several
    phenomena which do not appear in the $A_n$ case. It is possible that
    alternate presentations of mutations exist which do not exhibit
    this.
  \end{remark}

  \subsection{The map \texorpdfstring{$\conftoquiv$}{M}}

  We now have $Q$ and the mutations, at least graphically. What remains
  is to finalize the association to the Lie group $G$. More
  specifically, the vertices of $Q$ should be associated to coordinates
  on $\Conf_3^{\ast}(G/N_{+})$.

  Recall that $\Conf_3^{\ast}(G/N_{+}) \cong H^3 \times_{H} (N_{-} \cap
  \lifttoG{w_0}G_0)$ by \cref{lem:canonical-form-for-conf_3_star}. So
  let $\alpha \in \Conf_3^{\ast}(G/N_{+})$ be given by $\alpha = (h_1,
  h_2, h_3, u)$. By abuse of notation, we denote the coordinate in
  $\seedAtorus_Q$ by the name of the vertex to which it is associated,
  and regard \[ \seedAtorus_Q =
    \overbrace{\set{A_i}}^{\partialtorus_{A}} \times
    \overbrace{\set{B_i}}^{\partialtorus_B} \times
    \overbrace{\set{C_i}}^{\partialtorus_C} \times
    \overbrace{\set{v_{ij}}}^{\partialtorus_V}. \] We will define
  $\minorcoordmap : \Conf_3^{\ast}(G/N_{+}) \to \seedAtorus_Q$ and a let
  $\conftoquiv$ be in terms of $\minorcoordmap$.

  \begin{definition}
    \label{def:conf_3_star-to-seed-torus}
    Let $k(i,j) = i + rj$. (This has the property that $v_{ij}$ is the
    $k(i,j)^{\text{th}}$ vertex added to $Q_0$ in the algorithm of
    \cref{sec:Q0-construction}.) Alternately, the prefix of $w_0^{-1}$
    used to place vertex $v_{ij}$ in $Q_0$ via the algorithm of
    \cite[Section~2]{fominzelevinsky2005} has length $(m + 1) - k$, and
    is equal to $w_{k(i,j)}$ by \cref{def:gamma-k}. Using
    \cref{def:generalized-minor},

    \begin{align*}
      \minorcoordmap_{A} &: \alpha \mapsto \set{A_{i} = \genminor{}{i}(h_1)} \\
      \minorcoordmap_{B} &: \alpha \mapsto \set{B_{i} = \genminor{}{i}(h_2)} \\
      \minorcoordmap_{C} &: \alpha \mapsto \set{C_{i} = \genminor{}{i}(h_3)} \\
      \minorcoordmap_{V} &: \alpha \mapsto \set{ v_{ij} = \genminor{w_{k(i,j)}}{i}(u)} \\
      \minorcoordmap &: \alpha \mapsto \minorcoordmap_{A}(\alpha) \oplus \minorcoordmap_{B}(\alpha) \oplus \minorcoordmap_{C}(\alpha) \oplus \minorcoordmap_{V}(\alpha) \\
    \end{align*}

    We then define a monomial map $\monomialmap : \seedAtorus_Q \to
    \seedAtorus_Q$ (using
    \cref{def:sigma_G,def:non-negative-generalized-minor}) as
    \begin{align*}
      \monomialmap &: A_i \mapsto A_i \\
      \monomialmap &: B_i \mapsto B_i \\
      \monomialmap &: C_i \mapsto C_i \\
      \monomialmap &: v_{ij} \mapsto v_{ij} \cdot \genminor{}{i}(\sigma_G(h_1^{-1}) h_2) \cdot \genminornumerator{w_{k(i,j)}}{\fundamentalweight_i}(\sigma_G(h_1)) \\
    \end{align*}

    We define $\conftoquiv = \monomialmap \circ \minorcoordmap$.
  \end{definition}

  \begin{remark}
    \label{rem:M-is-birational-equivalence}
    The map $\conftoquiv$ is a birational equivalence, following from
    \cite[Section~2.7]{fominzelevinsky1999}.
  \end{remark}

  \begin{example}
    \label{ex:f4-part-6}

    Continuing from \cref{ex:f4-part-4}, let $\alpha = (h_1, h_2, h_3,
    u)$. Let $h_{ij} = \genminor{}{j}(h_i)$. Then $\minorcoordmap_{V}$
    produces the following: \[ A_1 = h_{11}, \quad A_2 = h_{12}, \quad
      \dotsc, \quad C_4 = h_{34}, \quad v_{11} = \genminor{w_5}{1}(u),
      \quad v_{12} = \genminor{w_6}{2}(u), \quad \dotsc, \quad v_{45} =
      \genminor{w_{24}}{4}(u). \]

    To compute the monomial map, refer to figure
    \cref{fig:abstract-gen-minor}. For example, to calculate
    $\monomialmap(v_{23})$, we have $i = 2$, so
    $\genminor{}{i}(\sigma_G(h_1^{-1}) h_2) = \frac{h_{22}}{h_{12}}$.
    For the second factor, $k = k(2,3) = 2 + 4\cdot 3 = 14$. Taking the
    numerator of the $i = 2$, $k = 14$ entry gives $t_1^2 t_2$, so
    $\genminornumerator{w_{14}}{\fundamentalweight_i}(\sigma_G(h_1)) =
    h_{11}^2 h_{12}$. This gives $\monomialmap(v_{23}) = v_{23} \cdot
    \frac{h_{22}}{h_{12}} \cdot h_{11}^2 h_{12}$, and
    $\conftoquiv(\alpha)$ has $v_{23} = \genminor{w_{14}}{2}(u) \cdot
    h_{11}^2 h_{22}$.

    Repeating this for all others, $\conftoquiv(\alpha)$ has coordinates
    as given in \cref{fig:f4-M-assignment}.
  \end{example}
  \begin{figure}
    \centering \tabulinesep=1.1mm

    \begin{tabu}{r|llll p{0.5in} r|llll p{0.5in} r|llll}
      $i$ & $1$ & $2$ & $3$ & $4$ && & $1$ & $2$ & $3$ & $4$ && & $1$ & $2$ & $3$ & $4$ \\ \hline
      $k=1$ & $\frac{1}{t_1}$ & $\frac{1}{t_2}$ & $\frac{1}{t_3}$ & $\frac{1}{t_4}$ && $9$ & $\frac{t_2}{t_3}$ & $\frac{t_1 t_2}{t_3 t_4}$ & $\frac{t_1^2 t_2^2}{t_3^2 t_4}$ & $\frac{t_2^2}{t_3 t_4}$ && $17$ & $\frac{t_1 t_2}{t_3}$ & $\frac{t_1 t_2^2}{t_3 t_4}$ & $\frac{t_1^2 t_2^2}{t_3 t_4}$ & $\frac{t_2^2}{t_3}$ \\
      $2$ & $\frac{t_1}{t_2}$ & $\frac{1}{t_2}$ & $\frac{1}{t_3}$ & $\frac{1}{t_4}$ && $10$ & $\frac{t_1}{t_4}$ & $\frac{t_1 t_2}{t_3 t_4}$ & $\frac{t_1^2 t_2^2}{t_3^2 t_4}$ & $\frac{t_2^2}{t_3 t_4}$ && $18$ & $\frac{t_2}{t_4}$ & $\frac{t_1 t_2^2}{t_3 t_4}$ & $\frac{t_1^2 t_2^2}{t_3 t_4}$ & $\frac{t_2^2}{t_3}$ \\
      $3$ & $\frac{t_1}{t_2}$ & $\frac{t_1}{t_3}$ & $\frac{1}{t_3}$ & $\frac{1}{t_4}$ && $11$ & $\frac{t_1}{t_4}$ & $\frac{t_1^2 t_2}{t_3 t_4}$ & $\frac{t_1^2 t_2^2}{t_3^2 t_4}$ & $\frac{t_2^2}{t_3 t_4}$ && $19$ & $\frac{t_2}{t_4}$ & $\frac{t_1 t_2}{t_4}$ & $\frac{t_1^2 t_2^2}{t_3 t_4}$ & $\frac{t_2^2}{t_3}$ \\
      $4$ & $\frac{t_1}{t_2}$ & $\frac{t_1}{t_3}$ & $\frac{t_1^2}{t_3 t_4}$ & $\frac{1}{t_4}$ && $12$ & $\frac{t_1}{t_4}$ & $\frac{t_1^2 t_2}{t_3 t_4}$ & $\frac{t_1^2 t_2^2}{t_3 t_4^2}$ & $\frac{t_2^2}{t_3 t_4}$ && $20$ & $\frac{t_2}{t_4}$ & $\frac{t_1 t_2}{t_4}$ & $\frac{t_2^2}{t_4}$ & $\frac{t_2^2}{t_3}$ \\
      $5$ & $\frac{t_1}{t_2}$ & $\frac{t_1}{t_3}$ & $\frac{t_1^2}{t_3 t_4}$ & $\frac{t_1^2}{t_3}$ && $13$ & $\frac{t_1}{t_4}$ & $\frac{t_1^2 t_2}{t_3 t_4}$ & $\frac{t_1^2 t_2^2}{t_3 t_4^2}$ & $\frac{t_1^2}{t_4}$ && $21$ & $\frac{t_2}{t_4}$ & $\frac{t_1 t_2}{t_4}$ & $\frac{t_2^2}{t_4}$ & $\frac{t_3}{t_4}$ \\
      $6$ & $\frac{t_2}{t_3}$ & $\frac{t_1}{t_3}$ & $\frac{t_1^2}{t_3 t_4}$ & $\frac{t_1^2}{t_3}$ && $14$ & $\frac{t_1 t_2}{t_3}$ & $\frac{t_1^2 t_2}{t_3 t_4}$ & $\frac{t_1^2 t_2^2}{t_3 t_4^2}$ & $\frac{t_1^2}{t_4}$ && $22$ & $t_1$ & $\frac{t_1 t_2}{t_4}$ & $\frac{t_2^2}{t_4}$ & $\frac{t_3}{t_4}$ \\
      $7$ & $\frac{t_2}{t_3}$ & $\frac{t_1 t_2}{t_3 t_4}$ & $\frac{t_1^2}{t_3 t_4}$ & $\frac{t_1^2}{t_3}$ && $15$ & $\frac{t_1 t_2}{t_3}$ & $\frac{t_1 t_2^2}{t_3 t_4}$ & $\frac{t_1^2 t_2^2}{t_3 t_4^2}$ & $\frac{t_1^2}{t_4}$ && $23$ & $t_1$ & $t_2$ & $\frac{t_2^2}{t_4}$ & $\frac{t_3}{t_4}$ \\
      $8$ & $\frac{t_2}{t_3}$ & $\frac{t_1 t_2}{t_3 t_4}$ & $\frac{t_1^2 t_2^2}{t_3^2 t_4}$ & $\frac{t_1^2}{t_3}$ && $16$ & $\frac{t_1 t_2}{t_3}$ & $\frac{t_1 t_2^2}{t_3 t_4}$ & $\frac{t_1^2 t_2^2}{t_3 t_4}$ & $\frac{t_1^2}{t_4}$ && $24$ & $t_1$ & $t_2$ & $t_3$ & $\frac{t_3}{t_4}$ \\
    \end{tabu}

    \caption{For type $F_4$, $\genminoryz{w_k}{w_k}{i}(A)$ for abstract
      $A = \prod_{i=1}^{r} \torush_i(t_i)$.
      \label{fig:abstract-gen-minor}
    }
  \end{figure}

  \begin{figure}
    \begin{equation*}
      \begin{gathered}
        \begin{aligned}
          A_1 &= h_{11} &\qquad A_2 &= h_{12} &\qquad A_3 &= h_{13} &\qquad A_4 &= h_{14} \\
          B_1 &= h_{11} & B_2 &= h_{12} & B_3 &= h_{13} & B_4 &= h_{14} \\
          C_1 &= h_{11} & C_2 &= h_{12} & C_3 &= h_{13} & C_4 &= h_{14} \\
        \end{aligned}
        \\
        \begin{aligned}
          v_{11} &= \genminor{w_{ 5}}{1}(u) \cdot h_{21} & v_{21} &= \genminor{w_{ 6}}{2}(u) \cdot \frac{ h_{11} h_{22} }{ h_{12} } & v_{31} &= \genminor{w_{ 7}}{3}(u) \cdot \frac{ h_{11}^2 h_{23} }{ h_{13} } & v_{41} &= \genminor{w_{ 8}}{4}(u) \cdot \frac{ h_{11}^2 h_{24} }{ h_{14} } \\
          v_{12} &= \genminor{w_{ 9}}{1}(u) \cdot \frac{ h_{12} h_{21} }{ h_{11} } & v_{22} &= \genminor{w_{10}}{2}(u) \cdot h_{11} h_{22} & v_{32} &= \genminor{w_{11}}{3}(u) \cdot \frac{ h_{11}^2 h_{12}^2 h_{23} }{ h_{13} } & v_{42} &= \genminor{w_{12}}{4}(u) \cdot \frac{ h_{12}^2 h_{24} }{ h_{14} } \\
          v_{13} &= \genminor{w_{13}}{1}(u) \cdot h_{21} & v_{23} &= \genminor{w_{14}}{2}(u) \cdot h_{11}^2 h_{22} & v_{33} &= \genminor{w_{15}}{3}(u) \cdot \frac{ h_{11}^2 h_{12}^2 h_{23} }{ h_{13} } & v_{43} &= \genminor{w_{16}}{4}(u) \cdot \frac{ h_{11}^2 h_{24} }{ h_{14} } \\
          v_{14} &= \genminor{w_{17}}{1}(u) \cdot h_{12} h_{21} & v_{24} &= \genminor{w_{18}}{2}(u) \cdot h_{11} h_{12} h_{22} & v_{34} &= \genminor{w_{19}}{3}(u) \cdot \frac{ h_{11}^2 h_{12}^2 h_{23} }{ h_{13} } & v_{44} &= \genminor{w_{20}}{4}(u) \cdot \frac{ h_{12}^2 h_{24} }{ h_{14} } \\
          v_{15} &= \genminor{w_{21}}{1}(u) \cdot \frac{ h_{12} h_{21} }{ h_{11} } & v_{25} &= \genminor{w_{22}}{2}(u) \cdot h_{11} h_{22} & v_{35} &= \genminor{w_{23}}{3}(u) \cdot \frac{ h_{12}^2 h_{23} }{ h_{13} } & v_{45} &= \genminor{w_{24}}{4}(u) \cdot \frac{ h_{13} h_{24} }{ h_{14} }
        \end{aligned}
      \end{gathered}
    \end{equation*}
    \caption{$\conftoquiv(\alpha)$ for type $F_4$.
      \label{fig:f4-M-assignment}
    }
  \end{figure}

  \begin{remark}
    There is good reason to suspect that this construction is not unique
    up to mutation equivalency of the quivers. It is certainly not
    unique if one relaxes the construction of $\conftoquiv$ as a
    monomial map applied to generalized minors. For example, during our
    investigation $\conftoquiv$ was considered as
    $\Conf_3^{\ast}(G/N_{+}) \times \gen{\sigma_G} \to \seedAtorus_Q
    \times \gen{\sigma_G}$, where $\murot^{\ast}$ acted by $+1$ on the
    second factor. An infinite family of $\conftoquiv$ maps were found,
    depending on this factor of $[\tau] \in \gen{\sigma_G}$ to varying
    extent.
  \end{remark}

  \subsection{The significance of the ``twisting'' mutations}

  \label{sec:what-does-the-twist-do}

  The quiver mutations $\murot$ and $\muflip$ are intended to compose
  easily without requiring renaming the vertices of the quiver. For
  example, $\murot^2(Q)$ should be identical to $\murot^{-1}(Q)$.

  Unfortunately, quiver mutations that yield equivalent seeds do not
  necessarily preserve the positions of those seeds. For example,
  consider the classic ``pentagon recurrence'', which happens to be
  equivalent to $\murowperm$ on a row which is of Dynkin type $A_2$. The
  switching of positions is exactly the the action of $\sigma_{A_2}$ on
  the Dynkin diagram.

  \begin{figure}[h]
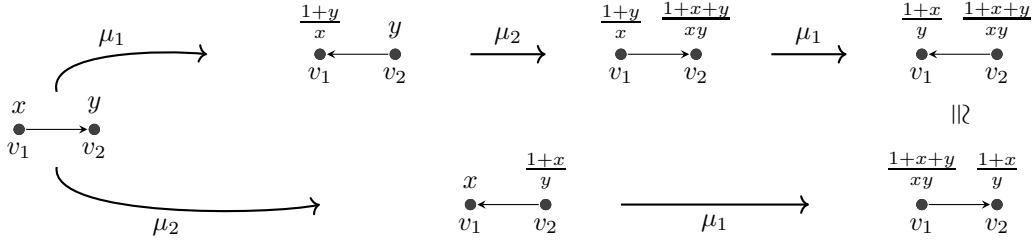

    \centering
    \includestandalone[mode=image|tex]{fig/pentagon-recurrence}

    \caption{The pentagon recurrence.
      \label{fig:pentagon-recurrence}
    }
  \end{figure}

  The $\muint{Oi}$ parts of $\murot$ and $\muflip$, distinguishing them
  from $\murottwist$ and $\mufliptwist$, are exactly to address these
  actions of $\sigma_G$. If desired, shorter mutations may be used at
  the expense of slightly more complicated identifications between
  variables.

  \section{Construction for \texorpdfstring{$A_n$}{A\_n}}

  \label{sec:An-case}

  The $A_n$ case was studied in detail in \cite{gtz2015}. We merely
  restate the conclusions in language consistent with the above.

  \begin{remark}
    That the algorithm of \cref{sec:main-result} does not work for
    $A_{2n}$ can be seen in a few different ways, which are
    interconnected.
    \begin{itemize}
    \item
      There is no general formula for a Coxeter element $c$ that yields
      a longest-word presentation $w_0 = c c \dotsb c$.
    \item
      The Coxeter number for $A_{2n}$ is odd.
    \item
      The action of $\sigma_G$ for $A_{2n}$ preserves no simple roots,
      therefore there can be no tree partitioning with a unique root
      node.
    \end{itemize}
  \end{remark}

  \begin{proposition}
    \label{prop:fgcs-exists-for-An}
    For $G$ of type $A_n$, \cref{thm:fgcs-exists} holds.
  \end{proposition}
  \begin{proof}
    For the construction of $Q$, take the quiver consisting of mutable
    vertices $\set{v_{ij} : 1 \le j \le n, 1 \le i \le n - j}$ and
    frozen vertices $\set{A_i, B_i, C_i : 1 \le i \le n}$, all of weight
    $1$. We consider \[ A_{i} = v_{0,i}, \quad B_{i} = v_{i,n+1-j},
      \quad C_{i} = v_{i,0}. \] The edges are given by \[ \sigma(a,b) =
      \begin{cases} c(a,b) & a = v_{i,j}, b = v_{i+1,j} \text{ or } a =
        v_{i,j}, b = v_{i-1,j+1} \text{ or } a = v_{i,j}, b = v_{i,j-1}
        \\ -c(a,b) & a = v_{i+1,j}, b = v_{i,j} \text{ or } a =
        v_{i-1,j+1}, b = v_{i,j} \text{ or } a = v_{i,j-1}, b = v_{i,j}
        \\ 0 & \text{else}, \end{cases} \] and $c(a,b)$ is $\frac{1}{2}$
    if $a$ and $b$ are both $A_i$s, both $B_i$s, or both $C_i$s, and is
    $1$ otherwise.

    The mutation $\murot = \mu\Set{}$ is trivial.

    Since the mutable portion of $Q_{A_n}$ is not rectangular, we need
    another convention for describing coordinates of $\Qdv$ in order to
    describe $\muflip$. We present this by example in
    \cref{fig:amalgamation-QAn}.
  \end{proof}

  \begin{figure}[h]
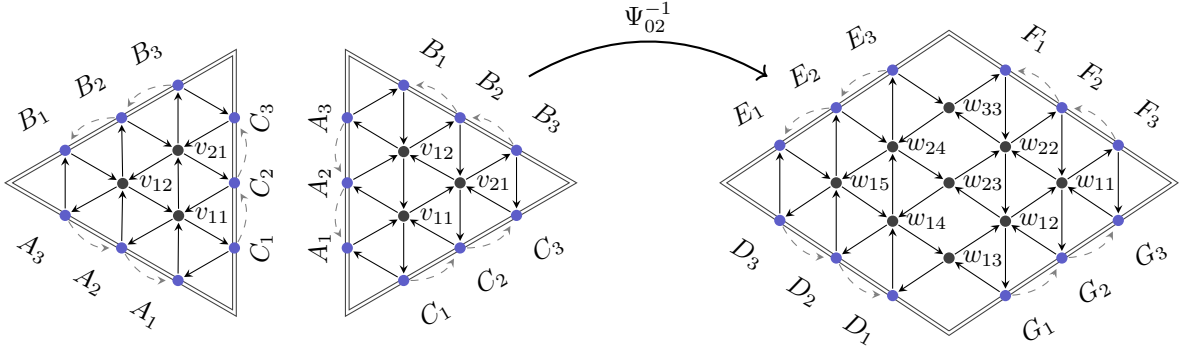

    \centering \includestandalone[mode=image|tex]{fig/amalgamation-QAn}

    \caption{Quivers for type $A_3$, also showing naming convention for
      $\Qdv$.
      \label{fig:amalgamation-QAn}
    }
  \end{figure}

  Mechanically, the inclusion of the left $Q_{A_n}$ into $\Qdv$ is the
  following: \[ A_i \mapsto D_i, \qquad B_i \mapsto E_i, \qquad C_i
    \mapsto w_{i,n}, \qquad v_{i,j} \mapsto w_{i,j+n}, \] and for the
  right $Q_{A_n}$ into $\Qdv$, the following: \[ A_i \mapsto w_{i,n},
    \qquad B_i \mapsto F_i, \qquad C_i \mapsto G_i, \qquad v_{i,j}
    \mapsto w_{j,n+1-i}, \]

  The mutation $\muflip$ is then the ``diamond sequence'':
  \begin{align*}
    t(\ell, j) &= \frac{\ell + 1}{2} - \abs*{j - \frac{\ell + 1}{2}} \\
    \muint{Rect}(\ell) &= \prod_{k = 1}^{\ell} \prod_{j=0}^{n-\ell} \mu\Set{w_{t(\ell,k) + j,2k + n - (\ell + 1)}} \\
    \muflip &= \prod_{\ell = 1}^{n} \muint{Rect}(\ell)
  \end{align*}

  The map $\conftoquiv$ is constructed as in
  \cref{def:conf_3_star-to-seed-torus}. Where needed, the presentation
  of the longest word is given by \[ \word{i} = (1, 2, \dotsc, n, \quad
    1, 2, \dotsc, n-1, \quad \dotsc, \quad 1, 2, \quad 1). \]

  \section{Proof of \cref{thm:fgcs-exists} for simple
    \texorpdfstring{$G$}{G}}

  \label{sec:proof}

  \subsection{Overview}

  Our goal is to prove \cref{thm:fgcs-exists}. We will only consider the
  case where $G$ is simple, as products are handled by
  \cref{sec:products}. We need to show that, for a fixed $G$, the
  results $Q$, $\murot$, $\muflip$, and $\conftoquiv$ satisfy all the
  requirements of \cref{def:tcs,def:fgcs}.

  \begin{proof}[Proof of \cref{thm:fgcs-exists}]
    If $G$ is not simple, then we may appeal to
    \cref{lem:fgcs-exists-for-products} and recurse, so assume $G$ is
    simple. For $G$ of type $A_n$, \cref{prop:fgcs-exists-for-An} is
    sufficient, so assume $G$ is not of type $A_n$.

    The algorithm of \cref{sec:main-result} produces $Q$, $\murot$,
    $\muflip$, and $\conftoquiv$. We must show these satisfy
    \cref{def:fgcs}.

    \begin{itemize}
    \item
      \cref{def:fgcs-1} (triangularity) is handled last. As described in
      \cref{sec:the-A-edge}, we need to know that $\murot$ is of order
      $3$ on the $\mathcal{A}$-coordinates before concluding that the
      quiver is triangular.
    \item
      \Cref{def:fgcs-2} (that the edges are half-Dynkin) is evident by
      construction.
    \item
      \Cref{def:fgcs-M-is-birational-equivalence} (that $\conftoquiv$ is
      a birational equivalence) is given by
      \cref{rem:M-is-birational-equivalence}.
    \item
      We handle \cref{def:fgcs-rot-works} (the rotation) with \cref{lem:fgcs-exists-rot-works-trivial,lem:fgcs-exists-rot-works-nontrivial}.
    \item
      We handle \cref{def:fgcs-flip-works} (the flip) with
      \cref{lem:fgcs-exists-flip-works}.
    \item
      Now, recalling \cref{rem:yes-using-murot-is-justified}, since
      $\murotind$ is of order $3$ acting on $\seedAtorus_Q$ and the seed
      torus of a cluster determines the quiver, $\murot$ is of order $3$
      on $Q$. Likewise, $\muflip$ must transform $\Qdv$ to $\Qdh$.
      Therefore, $Q$ has triangulation-compatible symmetry and
      \cref{def:fgcs-1} is satisfied.
    \end{itemize}
  \end{proof}

  In what follows, we assume $G$ a simple Lie group over $\mathbb{C}$ of
  type other than $A_n$, and that the $Q$, $\murot$, $\muflip$,
  $\conftoquiv$ are given by \cref{sec:main-result} for $G$.

  \subsection{Ordering mutations}

  The ends of both $\murot$ and $\muflip$ are sequences of $\murowperm$
  and $\mucolperm$. We must show that these permute Dynkin sub-quivers
  of $Q$ by switching certain vertices without changing the cluster
  seed, as in \cref{rem:sigma_G-action-on-Dynkin-diagram}. These are
  used in \cref{lem:fgcs-exists-rot-works-trivial,lem:fgcs-exists-rot-works-nontrivial,lem:fgcs-exists-flip-works}
  to justify the $\muint{Oi}$ parts.

  \begin{lemma}
    \label{lem:colsigmaG-does-sigmaG}
    The permutation $\mucolperm$ of
    \cref{def:murot-components,def:muflip-components} acts on quivers
    (and cluster ensembles) by permuting the vertices of the appropriate
    column according to $\sigma_G$.
  \end{lemma}
  \begin{lemma}
    \label{lem:rowsigmaG-does-sigmaG}
    The permutation $\murowperm$ of
    \cref{def:murot-components,def:muflip-components} acts on quivers
    (and cluster ensembles) by permuting the vertices of the mutated
    sub-quiver according to $\sigma_{A_{\ell}}$.
  \end{lemma}

  These proofs rely heavily on results of \cite{yangzelevinsky2008}
  described in \cref{sec:yz-dynkin-identity}. In short, we shall
  restrict the quivers on which $\mucolperm$ and $\murowperm$ act to
  those which Yang--Zelevinsky can strongly analyze. We then use
  counting arguments to establish that the mutation acts by permutation,
  and use classic results to apply these results to larger quivers.

  \begin{proof}[Proof of \cref{lem:colsigmaG-does-sigmaG}]
    In our construction, whenever $\mucolperm$ is applied, the vertices
    at which it mutates form a subquiver of Dynkin type. Therefore,
    first, we shall show that the result holds when $\mucolperm$ is
    applied to a quiver which is of Dynkin type. Then we shall show that
    the result holds for larger quivers which contain a sub-quiver of
    Dynkin type.

    First, suppose $\mucolperm$ acts on quiver of Dynkin type (non
    $A_{2n}$). There exists an element $g$ of $L^{c,c^{-1}}$ such that
    the initial cluster coordinate at $v_k$ is $\genminoryz{}{}{k}(g)$.

    Now we note that $\mucolperm$ is exactly $\ell + 2 = \frac{h}{2} +
    1$ iterations of a mutation following the Coxeter element $c$, with
    each mutation replacing $\genminoryz{c^m}{c^m}{k}$ with
    $\genminoryz{c^{m+1}}{c^{m+1}}{k}$. Since the number $h(i;c)$ is
    always $\frac{h}{2}$ by our construction of $c$, from
    \cref{prop:compressed-yang-zelevinsky} we obtain that the final
    cluster coordinate at $v_k$ is \[ \genminoryz{c^{h(k;c) +
          1}}{c^{h(k;c) + 1}}{k}(g) = \genminoryz{}{}{k^{\ast}}(g), \]
    which was the initial coordinate at $v_{k^{\ast}}$.

    Thus $\mucolperm$ acts on the cluster by permuting the associations
    between vertices and cluster variables according to $\sigma_G$. In
    this way we obtain the desired result without ever having to rely
    directly facts of $g \in L^{c,c^{-1}}$.

    Now it must be shown that this holds when $\mucolperm$ acts on a
    larger quiver. Since the action of $\mucolperm$ is solely a
    permutation on the mutable portion of the quiver, it preserves the
    set of cluster variables. By
    \cite[Theorem~1.12]{fominzelevinsky2003} and the finiteness of
    Dynkin-type quivers, the cluster variables determine the exchange
    matrix and therefore the quiver. Therefore the action of
    $\mucolperm$ on the graph must be trivial up to renaming, and
    therefore must be exactly the renaming that we have created.
  \end{proof}

  \begin{proof}[Proof of \cref{lem:rowsigmaG-does-sigmaG}]
    We would like to apply exactly the same argument as in the proof of
    \cref{lem:colsigmaG-does-sigmaG}. However, since rows are of type
    $A_{\ell}$ and our construction algorithm for $Q$ does not allow
    $\murowperm$ to be repeated iterations of a single Coxeter-style
    mutation, we must be a bit more careful.

    For a type $A_{\ell}$, define $\Q{Lin} = Q_{A_{\ell},\text{Lin}}$ as
    the Dynkin quiver where the graph describes a linear ordering from
    left to right. Also, let $\Q{Alt} = Q_{A_{\ell},\text{Alt}}$ be the
    Dynkin quiver in source-sink position, see
    \cref{fig:dynkin-type-Al}.

    \begin{figure}[h]
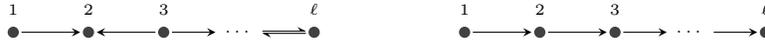

      \centering \includestandalone[mode=image|tex]{fig/dynkin-type-Al}

      \caption{$\Q{Alt}$ on the left, $\Q{Lin}$ on the right, both
        quivers of Dynkin type $A_{\ell}$.
        \label{fig:dynkin-type-Al}
      }
    \end{figure}

    We also introduce one more piece of notation.
    \begin{fact}
      \label{fac:Y-sub-parity}
      Let $Y(x)$ be defined as $\frac{1}{2}$ when $x$ is odd and $0$
      when $x$ is even. Then \[ Y(a) - Y(b) =
        (-1)^{\operatorname{parity} a}Y(a-b). \]
    \end{fact}

    As above, we will interpret the cluster coordinates of the mutable
    Dynkin-type quiver as minor coordinates on some $g \in
    L^{c,c^{-1}}$, but we cannot interpret them as the initial minors.

    The mutation $\murowperm$ acts on $\ell = b - a$ vertices, always
    arranged as in $\Q{Lin}$ of \cref{fig:dynkin-type-Al}. Recalling
    notation of \cref{prop:def-h(i;c)}, by \cite{yangzelevinsky2008} the
    quiver $\Q{Alt}$ (together with cluster variables
    $\genminoryz{}{}{i}$ at $v_{i}$) is an initial seed, governed by an
    element of $L^{t_{+}t_{-},(t_{+}t_{-})^{-1}}$.

    It is straightforward that (up to Langlands dualizing, which has no
    effect on coordinates) the mutation
    \begin{equation*}
      \begin{aligned}
        \muint{Lin} &= \begin{cases}\mu\Set{v_1, v_3, \dotsc, v_{\ell - 0}, \quad v_2, v_4, \dotsc, v_{\ell - 1}, \quad \dotsc, \quad v_1, v_3, \quad v_2, \quad v_1} & \text{$\ell$ odd} \\ \mu\Set{v_1, v_3, \dotsc, v_{\ell - 1}, \quad v_2, v_4, \dotsc, v_{\ell - 2}, \quad \dotsc, \quad v_1, v_3, \quad v_2, \quad v_1} & \text{$\ell$ even} \end{cases} \\
        &= \prod_{j= 2 \lfloor \ell / 2 \rfloor - (\ell - 1)}^{\ell - 1} \begin{cases} \prod_{k=1}^{\lceil (\ell - j) / 2 \rceil} \mu\set{v_{2k - 1}} & \text{$\ell - j$ odd} \\ \prod_{k=1}^{(\ell - j) / 2} \mu\set{v_{2k}} & \text{$\ell - j$ even} \end{cases}
      \end{aligned}
    \end{equation*}
    transforms $\Q{Alt}$ to $\Q{Lin}$, using only source or sink
    mutations (thus governed by primitive exchange relations). By
    considering these relations, and counting mutations, it is also
    straightforward that the cluster coordinates of $\Q{Lin}$ are given
    by the following.
    \begin{equation*}
      \text{coordinate at $v_k$ is $\genminoryz{c^{(\ell - k)/2 +
            Y(\ell) + Y(k)}}{c^{(\ell - k)/2 + Y(\ell) + Y(k)}}{k}$}
    \end{equation*}
    For example, coordinates of $Q_{A_4,\text{Lin}}$ are
    $\Set{\genminoryz{c^2}{c^2}{1}, \genminoryz{c^1}{c^1}{2},
      \genminoryz{c^1}{c^1}{3}, \genminoryz{c^0}{c^0}{4}}$, and the
    coordinates for $Q_{A_3,\text{Lin}}$ are
    $\Set{\genminoryz{c^2}{c^2}{1}, \genminoryz{c^1}{c^1}{2},
      \genminoryz{c^1}{c^1}{3}}$. To prove the lemma, we must show that
    $\murowperm$ permutes these by $\sigma_G$, which in the case of
    $A_{\ell}$ replaces each $\fundamentalweight_k$ with
    $\fundamentalweight_{\ell + 1 - k}$.

    Since $\murowperm$ also consists entirely of source-or-sink
    mutations, it is also governed by primitive exchange mutations. Let
    \[ \muint{A} = \prod_{k=1}^{\ell} \prod_{j=0}^{\ell - k} \muT(i, 1 +
      j), \qquad \muint{B} = \prod_{j=1}^{\ell} \muT(i, \ell - j), \] so
    that $\murowperm = \muint{A} \muint{B}$. We make the following
    observations, which are all easily checked by induction:
    \begin{itemize}
    \item
      During $\muint{A}$, whenever a vertex $v_{k}$ is mutated at with
      $k$ odd, $v_{k}$ has coordinate $\genminoryz{c^m}{c^m}{k}$ while
      its neighbors have coordinates $\genminoryz{c^{m -
          1}}{c^{m-1}}{k-1}$ and $\genminoryz{c^{m - 1}}{c^{m-1}}{k+1}$.
      Since $k \prec_{c} k \pm 1$, the coordinate at $v_{k}$ becomes
      $\genminoryz{c^{m-1}}{c^{m-1}}{k+1}$.
    \item
      During $\muint{A}$, whenever a vertex $v_{k}$ is mutated at with
      $k$ even, $v_{k}$ has coordinate $\genminoryz{c^m}{c^m}{k}$ and
      its neighbors have coordinates $\genminoryz{c^m}{c^m}{k-1}$ and
      $\genminoryz{c^m}{c^m}{k+1}$. Since $k \pm 1 \prec_{c} k$, the
      coordinate at $v_{k}$ becomes
      $\genminoryz{c^{m-1}}{c^{m-1}}{k+1}$.
    \item
      By the same logic as the above statements every mutation in
      $\muint{B}$ takes the coordinate at $v_{k}$ from $\genminoryz{c^{m
        }}{c^{m}}{k}$ to $\genminoryz{c^{m+1}}{c^{m+1}}{k}$.
    \end{itemize}

    Therefore, since $\muint{A}$ touches vertex $k$ a total of $\ell - k
    + 1$ times and $\muint{B}$ touches each vertex once, we may conclude
    via \cref{prop:compressed-yang-zelevinsky} that after
    $\muint{A}\muint{B}$ the coordinates are given by
    \begin{equation*}
      \text{coordinate at $v_k$ is $\genminoryz{c^{(\ell - k)/2 +
            Y(\ell) + Y(k) - (\ell - k)}}{c^{(\ell - k)/2 + Y(\ell) +
            Y(k) - (\ell - k)}}{k}$.}
    \end{equation*}

    Now, let us consider the difference between the exponents of $c$ in
    the final coordinate of $v_{k}$ and in the initial coordinate of
    $v_{k^{\ast}}$. If the difference is $-h(k^{\ast};c) - 1$, then
    \cref{prop:compressed-yang-zelevinsky} will show that $\murowperm$
    acts by permuting vertices according to $\sigma_G$.

    \begin{align*}
      \text{final at $k$} - \text{initial at $k^{\ast}$} &= \left[ \frac{\ell - k}{2} + Y(\ell) + Y(k) - (\ell - k) \right] - \left[\frac{\ell - k^{\ast}}{2} + Y(\ell) + Y(k^{\ast}) \right] \\
      &= \frac{k^{\ast} - k}{2} - (\ell - k) + Y(k) - Y(k^{\ast}) \\
      &= \frac{(\ell + 1 - k) - k}{2} - (\ell - k) + Y(k) - Y(\ell + 1 - k) \\
      &= -\left[\frac{\ell + 1}{2} - Y(k) + Y(\ell + 1 - k)\right] - 1 \\
      &= \begin{cases} -\left[ \frac{\ell + 1}{2} + Y(\ell + 1)\right] - 1 & \text{$\ell + 1 - k$ odd} \\ -\left[ \frac{\ell + 1}{2} - Y(\ell + 1)\right] - 1 & \text{$\ell + 1 - k$ even}\end{cases} \tag{\Cref{fac:Y-sub-parity}} \\
      &= \begin{cases} -\lceil\frac{\ell + 1}{2}\rceil - 1 & \text{$k^{\ast}$ odd} \\ -\lfloor \frac{\ell + 1}{2}\rfloor - 1 & \text{$k^{\ast}$ even}\end{cases} \\
    \end{align*}
    Consulting \cref{prop:def-h(i;c)}, the exponent difference is indeed
    $-h(k^{\ast};c) - 1$, so the action of $\murowperm$ is to rearrange
    the vertices of $\Q{Lin}$ according to $\sigma_G$. From here, the
    action of $\murowperm$ generalizes to larger quivers as in the proof
    for $\mucolperm$.
  \end{proof}

  \subsection{Rotation mutations}

  \label{sec:proof-of-rot}

  We now turn to verifying the longer mutations, starting with $\murot$.
  First, we recall work of Zickert to describe the effect of
  $\operatorname{rot}$ in terms of canonical forms, taking $(h_1, h_2,
  h_3, u)$ to $(h_3, h_1, h_2, \widetilde{u})$. We show, using an
  identity of Fomin--Zelevinsky, that when all $h_i$ are trivial,
  $\murottwistind$ takes coordinates of $u$ to those of $\widetilde{u}$.
  Then we show that the monomial map $\monomialmap$ of $\conftoquiv$ is
  exactly what is necessary to extend to the general case. The final
  pieces $\muint{O1}$ and $\muint{O2}$, rearrange the vertices without
  changing their coordinates according to
  \cref{lem:colsigmaG-does-sigmaG,lem:rowsigmaG-does-sigmaG}.

  Recall that by \cref{prop:Phi-Psi-and-rot}, we have \[ \widetilde{u} =
    h_2^{-1} (w_0(h_1))^{-1} (\Phi \Psi \Phi \Psi)(u) (w_0(h_1)) h_2. \]
  So we must consider $(\Phi \Psi)^{2}(u)$ in terms of minor
  coordinates, then show that the result agrees in general with the
  cluster action of $\murot$. In other words, we must show that
  $\murotind$ fits into the diagram of \cref{fig:Psi-Phi-murot-cd}. We
  interpret the actions of the maps $\Phi$ and $\Psi$ on generalized
  minors, then appeal to an identity of Fomin--Zelevinsky that applies
  at every step of the mutation sequence $\murottwist$.

  \begin{figure}[h]
    \centering \includestandalone[mode=image|tex]{fig/Psi-Phi-murot-cd}

    \caption{Maps $\Phi$ and $\Psi$ by \cref{prop:Phi-Psi-and-rot},
      $\minorcoordmap_V$ by \cref{def:conf_3_star-to-seed-torus}.
      \label{fig:Psi-Phi-murot-cd}
    }
  \end{figure}

  \begin{lemma}
    \label{lem:fgcs-exists-rot-works-trivial}
    \Cref{def:fgcs-rot-works} of \cref{def:fgcs} holds for $h_1 = h_2 =
    h_3$ are all trivial.

  \end{lemma}
  \begin{proof}
    By \cref{rem:p-commutes-with-mutation}, we need only consider the
    commuting diagrams for $\mathcal{A}$-coordinates; the diagrams for
    $\mathcal{X}$-coordinates will then follow immediately by applying
    $\TtoZ$.

    Let $\widetilde{u} = (w_0(h_1))^{-1} (\Phi \Psi)^2(u) w_0(h_1) h_2$
    for notation. The core idea is that, modulo coordinates of $h_i$,
    $\genminoryz{w_k}{e}{i_k}(u)$ is approximately
    $\genminoryz{w_0}{w_k}{i_k}(\widetilde{u})$. The mutation sequence
    $\murot$ transforms $\genminoryz{w_0}{w_k}{i_k}$ to
    $\genminoryz{w_k}{e}{i_k}$ by
    \cite[Theorem~1.17]{fominzelevinsky1999}, therefore changing
    coordinates of $u$ into coordinates of $\widetilde{u}$. Later, we
    will show that the monomial map exactly compensates for the
    coordinates of $h_i$.

    To be more precise, the coordinate assigned to interior vertex
    $v_{i,j}$ is (letting $k = k(i,j)$ as in
    \cref{def:conf_3_star-to-seed-torus}) \[ v_{i,j} =
      \genminoryz{w_k}{e}{i_k}(u) \cdot
      \genminor{}{i}(\sigma_G(h_1^{-1}) h_2) \cdot \genminornumerator{w_{k(i,j)}}{\fundamentalweight_i}(\sigma_G(h_1))
    \] Focusing on the first term, consider $(\Phi \Psi)^2(u) = [ \Psi([
        \Psi(u) \lifttoG{w_0} ]_{-}) \lifttoG{w_0} ]_{-}$. Letting
    $x_{\bullet} \in N_{+}$, $h_{\bullet} \in H$, and $y_{\bullet} \in
    N_{-}$,
    \begin{align*}
      y_0 &= u & x_1 &= \Psi(y_0) & y_2 H_2 x_2 &= x_1 \lifttoG{w_0} \\
      y_2 &= (\Phi \Psi)(u) & x_3 &= \Psi(y_2) & y_4 H_4 x_4 &= x_3 \lifttoG{w_0} \\
      y_4 &= (\Phi \Psi)^2(u).
    \end{align*}
    Expanding those terms, we have
    \begin{align*}
      y_4 H_4 x_4 &= x_3 \lifttoG{w_0} \\
      &= \Psi(y_2) \lifttoG{w_0} \\
      &= \Psi(x_1 \lifttoG{w_0} x_2^{-1} H_2^{-1}) \lifttoG{w_0} \\
      &= \Psi(\Psi(y_0) \lifttoG{w_0} x_2^{-1} H_2^{-1}) \lifttoG{w_0} \\
      &= \Psi(H_2^{-1}) \Psi(x_2^{-1}) \Psi(\lifttoG{w_0}) y_0 \lifttoG{w_0} \\
      &= H_2^{-1} \Psi(x_2^{-1}) \lifttoG{w_0} y_0 \lifttoG{w_0} \\
      &= \overbrace{H_2^{-1} \Psi(x_2^{-1}) H_2}^{\in N_{-}} \cdot \overbrace{H_2^{-1} s_G}^{\in H} \cdot \overbrace{\lifttoG{w_0}^{-1} y_0 \lifttoG{w_0}}^{\in N_{+}}
    \end{align*}
    Extracting the $N_{-}$ term, $(\Phi \Psi)^2(u) = y_4 = H_2^{-1}
    \Psi(x_2^{-1}) H_2$, which we can rewrite via \[ H_2 = [ \Psi(y_0)
        \lifttoG{w_0} ]_{0}, \qquad x_2 = \Psi([\Psi(y_0) \lifttoG{w_0}
      ]_{+}^{-1}). \]

    Now, consider the minor $\genminoryz{w_0}{w_k}{i_k}(y_4)$. Denote by
    $w_k^{\ast}$ the word defined by \cref{def:gamma-k} using fixed
    presentation $\word{i}^{\ast} = \sigma_G(\word{i})$. By relying on a
    number of identities from \cite{fominzelevinsky1999},
    \begin{align*}
      \genminoryz{w_0}{w_k}{i_k}(H_2^{-1} \Psi(x_2^{-1}) H_2) &= \frac{\genminoryz{w_0}{w_k}{i_k}(\Psi(x_2^{-1})) \cdot \genminoryz{w_k}{w_k}{i_k}(H_2)}{\genminoryz{w_0}{w_0}{i_k}(H_2)} \tag{\cite[Equation~2.14]{fominzelevinsky1999}} \\
      \genminoryz{w_0}{w_k}{i_k}(\Psi(x_2^{-1})) &= \genminoryz{w_k}{w_0}{i_k}(x_2^{-1}) \tag{\cite[Equation~2.25]{fominzelevinsky1999}} \\
      &= \genminoryz{w_k}{w_0}{i_k}( \Psi([\Psi(y_0)\lifttoG{w_0} ]_{+}^{-1})^{-1} ) \\
      &= \genminoryz{w_k^{\ast}}{w_0}{i_k^{\ast}}( \lifttoG{w_0} [\Psi(y_0)\lifttoG{w_0} ]_{+}^{-1} \lifttoG{w_0}^{-1} ) \\
      &= \genminoryz{w_0w_k^{\ast}}{e}{i_k^{\ast}}( [\Psi(y_0)\lifttoG{w_0} ]_{+}^{-1} ) \\
      &= \genminoryz{w_0w_k^{\ast}}{e}{i_k^{\ast}}( [\lifttoG{w_0} y_0]_{-} ) \\
      &= \left(\genminoryz{w_0w_k^{\ast}}{e}{i_k^{\ast}}( \lifttoG{w_0} y_0 )\middle) \middle/ \middle(\genminoryz{e}{e}{i_k^{\ast}}(\lifttoG{w_0} y_0) \right) \tag{\cite[Equation~2.23]{fominzelevinsky1999}} \\
      &= \left(\genminoryz{w_k^{\ast}}{e}{i_k^{\ast}}( y_0 ) \middle) \middle/ \middle( \genminor{}{i_k^{\ast}}([\lifttoG{w_0}^{-1} y_0]_{0}) \right)
    \end{align*}
    Also, by \cref{lem:canonical-form-for-conf_3_star}, $H_2 =
    [\Psi(y_0) \lifttoG{w_0}]_{0} = [\lifttoG{w_0}^{-1} y_0]_{0} =
    (w_0(h_3 h_1) h_2)^{-1}$, where the $h_i$ in the last term refer to
    the elements of $H$ in the canonical form for
    $\Conf_{3}^{\ast}(G/N_{+})$. Writing solely in terms of $(h_1, h_2,
    h_3, u)$), we have
    \begin{align*}
      \genminoryz{w_0}{w_k}{i_k}((\Phi \Psi)^2(u)) &= \genminoryz{w_k^{\ast}}{e}{i_k^{\ast}}(u) \cdot \frac{ \genminoryz{w_0}{w_0}{i_k}( w_0(h_3 h_1)h_2 ) \genminoryz{e}{e}{i_k^{\ast}}( w_0(h_3 h_1)h_2 ) }{ \genminoryz{w_k}{w_k}{i_k}( w_0(h_3 h_1) h_2 ) } \\
      &= \genminoryz{w_k^{\ast}}{e}{i_k^{\ast}}(u) \cdot \frac{ 1 }{ \genminoryz{w_k}{w_k}{i_k}( w_0(h_3 h_1) h_2 ) } \\
    \end{align*}
    Now, introducing the factor of $w_0(h_1) h_2$, we have
    \begin{align*}
      \genminoryz{w_0}{w_k}{i_k}(\widetilde{u}) &= \genminoryz{w_0}{w_k}{i_k}((w_0(h_1))^{-1} (\Phi \Psi)^2(u) w_0(h_1) h_2) \\
      &= \genminoryz{w_0}{w_k}{i_k}((\Phi \Psi)^2(u)) \cdot \genminoryz{w_0}{w_0}{i_k}(w_0(h_1)^{-1}h_2^{-1}) \cdot \genminoryz{w_k}{w_k}{i_k}(w_0(h_1)h_2) \\
      &= \genminoryz{w_k^{\ast}}{e}{i_k^{\ast}}(u) \cdot \frac{1}{ \genminoryz{w_k}{w_k}{i_k}(w_0(h_3)) \genminoryz{w_0}{w_0}{i_k}(w_0(h_1) h_2) }
    \end{align*}

    Thus, including the assignment of the monomial map $\monomialmap$,
    the coordinate at $v_{i,j}$ is given by
    \begin{align*}
      v_{i,j} &= \genminoryz{w_0}{w_k}{i_k}(\widetilde{u}) \cdot \overbrace{\genminoryz{e}{e}{i_i}(w_0(h_1) h_2) \cdot \genminoryz{w_0}{w_0}{i_k}(w_0(h_1)h_2)}^{1} \cdot \genminoryz{w_k}{w_k}{i_k}(w_0(h_3)) \cdot \genminornumerator{w_k}{\fundamentalweight_{i_k}}(w_0(h_1^{-1})) \\
      &= \genminoryz{w_0}{w_k}{i_k}(\widetilde{u} \cdot w_0(h_3)) \cdot \genminornumerator{w_k}{\fundamentalweight_{i_k}}(w_0(h_1^{-1}))
    \end{align*}

    In the case that $h_1 = h_2 = h_3 = e$, we may ignore the frozen
    vertices of the quiver. The proof is completed by appealing to the
    identity of \cite[Theorem~1.17]{fominzelevinsky1999}, which states
    that when $\length(u s_i) = \length(u) + 1$ and $\length(v s_i) =
    \length(v) + 1$, \[ \genminoryz{u}{v}{i} \genminoryz{us_i}{vs_i}{i}
      = \genminoryz{u s_i}{v}{i} \genminoryz{v}{u s_i}{i} + \prod_{j \ne
        i} \genminoryz{u}{v}{i}^{-a_{ji}}. \]

    This identity exactly matches the quiver mutation relation at each
    step of $\murottwist$, so each $\mucol(i)$ in $\murot$ transforms
    the coordinates of $v_{i,\bullet}$ from \[
      \genminoryz{c^n}{c^m}{j}(\widetilde{u})\] to \[
      \genminoryz{c^{n-1}}{c^{m-1}}{j}(\widetilde{u}).\] Thus, after
    applying $\murottwist$ the vertex $v_{i,j}$ has coordinate \[
      \genminoryz{w_{m+r-k}^{\ast}}{e}{i_k^{\ast}}(\widetilde{u}). \] By
    \cref{lem:colsigmaG-does-sigmaG,lem:rowsigmaG-does-sigmaG},
    $\muint{O1}$ converts the $w_{m+r-k}$ to $w_k$, and $\muint{O2}$
    removes the $\cdot^{\ast}$. Thus $\murot$ changes coordinates for
    $u$ into those for $h_2^{-1} w_0(h_1)^{-1} (\Phi \Psi)^2(u) w_0(h_1)
    h_2$ as desired.
  \end{proof}

  \begin{proof}[Proof of \cref{lem:murot-works-for-making-Q}]
    Since we ignored frozen vertices above, the result applies to $Q_0$.
    Since the quiver mutation changes coordinates of $\alpha$ to those
    of $\operatorname{rot}(\alpha)$, it is of order 3 and induces a
    graph isomorphism.
  \end{proof}

  \begin{lemma}
    \label{lem:fgcs-exists-rot-works-nontrivial}
    \Cref{def:fgcs-rot-works} of \cref{def:fgcs} holds for arbitrary
    $h_1, h_2, h_3$.
  \end{lemma}
  \begin{proof}
    In the case that $h_1, h_2, h_3$ are not trivial, we must verify
    that the equation of \cite[Theorem~1.17]{fominzelevinsky1999} still
    holds with the frozen vertices considered, which introduce factors
    not included in the equation. Again, we only consider
    $\mathcal{A}$-coordinates by \cref{rem:p-commutes-with-mutation}. We
    note the following:

    \begin{itemize}
    \item
      The vertices $B_k$ take the place of
      $\genminoryz{w_0}{e}{i_k}(\widetilde{u}) =
      \chi_{i_k^{\ast}}([\lifttoG{w_0}^{-1} \widetilde{u}]_{0}) =
      \chi_{i_k^{\ast}}(w_0(h_2 h_3) h_1)^{-1}$ in the equation.
      Therefore, they should be considered as \[ B_k =
        \genminoryz{w_0}{e}{i_k}(\widetilde{u}) \cdot
        \frac{A_{k^{\ast}}}{C_k} =
        \genminoryz{w_0}{e}{i_k}(\widetilde{u} \cdot w_0(h_3)) \cdot
        A_{k^{\ast}}. \]
    \item
      If, at each step of the mutation, the frozen vertices and the
      monomial map $\monomialmap$ induce equal extra factors on the
      $\genminoryz{u s_i}{v}{i} \genminoryz{v}{u s_i}{i}$ and $\prod_{j
        \ne i} \genminoryz{u}{v}{i}^{-a_{ji}}$ terms, then the
      difference of the final terms from coordinates of $\widetilde{u}$
      will also be a monomial map.
    \item
      If this final difference monomial map matches $\monomialmap$, the
      result will be proven.
    \item
      Since the extra factors at each coordinate are given by a monomial
      map, each frozen vertex may be checked individually.
    \end{itemize}

    For any particular group, these results may be verified by a few
    numerical calculations: to verify the monomial identity $x_1^{a_1}
    \dotsb x_m^{a_m} \stackrel{?}{=} x_1^{b_1} \dotsb x_m^{b_m}$, taking
    $\log$ of both sides reduces to $\sum_{i} \log(x_i) (a_i - b_i)
    \stackrel{?}{=} 0$. Therefore, evaluating at $m+1$ linearly
    independent choices of vector $\langle\log(x_i)\rangle$ verifies the
    result. This has been carried out for the exceptional groups. The
    \texttt{test-murot} sub-program of \cite{gilles2020sw} may be used
    for this purpose.

    We present an argument that the coordinates of $h_1$ are treated
    correctly by the monomial map and $\murot$ for type $D_n$; other
    arguments are similar.

    We will ignore all terms other than $h_{1,j} = A_j$, and annotate
    vertices with these. Half the Coxeter number minus one will be
    denoted $\ell$, which in the case of $D_n$ is equal to $n - 2$. Then
    we will trace the effects of $\murottwist$ in general. The objective
    is to show that, starting with the assignments given by the monomial
    factors attached to $\genminoryz{w_0}{w_k}{i_k}(\widetilde{u})$
    above, we end with those factors of given by the monomial map
    $\monomialmap$ following $\operatorname{rot}$ (i.e.\ the powers
    associated to $h_2$ by $\monomialmap$). Before applying any
    mutations, $Q_{D_n}$ is as in \cref{fig:murot-QD7-stage0}. The
    formula for $A_j$ factors is, following
    $\genminornumerator{w_k}{\fundamentalweight_{i_k}}(w_0(h_1^{-1}))$,
    \[ \text{at }v_{i,j} : \begin{cases} A_{j} & i \ge n-1 \\ A_{j} & i
        < n-1, i + j < n \\ A_{j} A_{i+j+1-n} & i < n-1, i + j \ge n
      \end{cases} \qquad \text{at }A_k: A_k, \qquad \text{at }B_k:
      A_{k^{\ast}}, \qquad \text{at }C_k: 1 \]

    \begin{figure}[h]
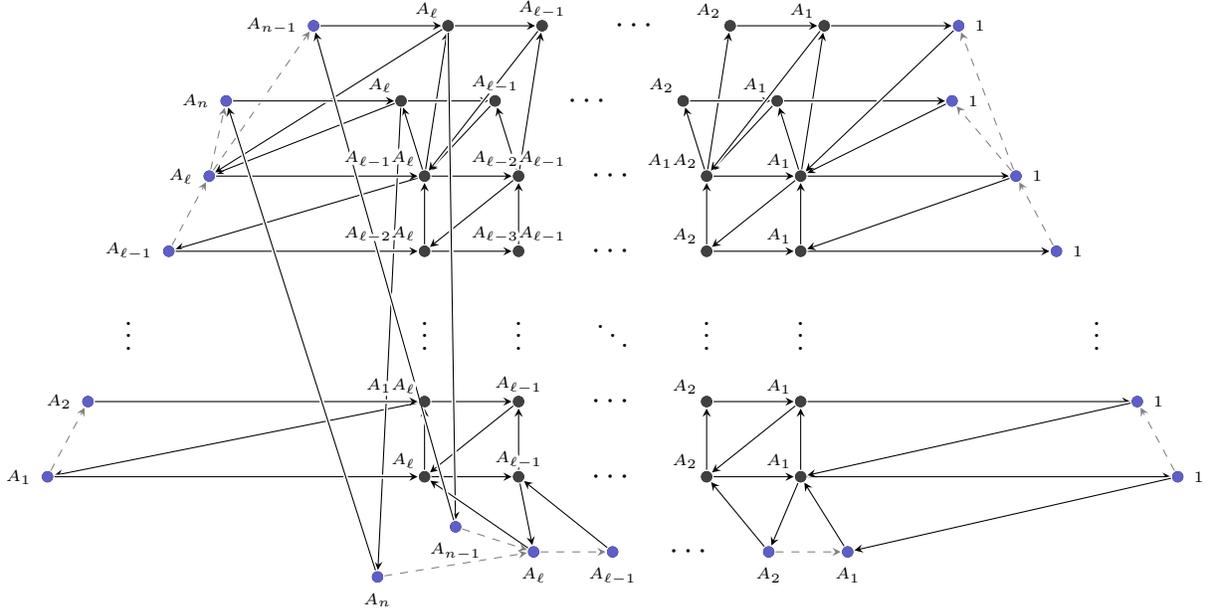

      \centering
      \includestandalone[mode=image|tex]{fig/murot-QD7-stage0}

      \caption{$Q_{D_n}$ before any mutation.
        \label{fig:murot-QD7-stage0}
      }
    \end{figure}

    After one iteration of $\mucol(1)$, by straight-forward induction
    the factor at $v_{i,j}$ matches that at $v_{i,j+1}$, as in
    \cref{fig:murot-QD7-stage1}.

    \begin{figure}[h]
      \centering
      \includestandalone[mode=image|tex]{fig/murot-QD7-stage1}

      \caption{$Q_{D_n}$ after $\mucol(1)$.
        \label{fig:murot-QD7-stage1}
      }
    \end{figure}

    After $\prod_{y=1}^{\ell} \mucol(y)$, this pattern continues as
    shown in \cref{fig:murot-QD7-stage2}: for $j < \ell$, we have the
    factor at $v_{i,j}$ matching the original factor at $v_{i,j+1}$. At
    $v_{i,\ell}$, however, the factor matches the original factor
    associated to $B_i$. Therefore, the new formula for $A_j$ factors at
    $v_{i,j}$ is \[ \text{at }v_{i,j} : \begin{cases} A_{i^{\ast}} & j =
        \ell \\ A_{j+1} & j < \ell, i \ge n-1 \\ A_{j+1} & j < \ell, i <
        n-1, i + j < n \\ A_{j+1} A_{i+j+2-n} & j < \ell, i < n-1, i + j
        \ge n \end{cases} \]

    \begin{figure}[h]
      \centering
      \includestandalone[mode=image|tex]{fig/murot-QD7-stage2}

      \caption{$Q_{D_n}$ after $\prod_{y=1}^{\ell}\mucol(y)$.
        \label{fig:murot-QD7-stage2}
      }
    \end{figure}

    Each successive $\prod_{y=1}^{\dotsc} \mucol(y)$ performs the same
    adjustment: factors of $A_j$ are copied from left to right. As can
    be shown by induction, after $\prod_{x=1}^{k} \prod_{y=1}^{\ell + 1
      - x} \mucol(y)$, the formula for $A_j$ factors at $v_{i,j}$ is \[
      \text{at }v_{i,j} : \begin{cases} A_{i^{\ast}} & j > \ell - k \\
        A_{j+k} & j \le \ell - k, i \ge n-1 \\ A_{j+k} & j \le \ell - k,
        i < n-1, i + j < n \\ A_{j+k} A_{i+j+1+k-n} & j \le \ell - k, i
        < n-1, i + j \ge n \end{cases} \]

    Thus, after applying $\murottwist$, each $v_{i,j}$ has a factor of
    $A_{i^{\ast}}$ attached. Since the monomial map $\monomialmap$
    associates a factor of $\genminor{}{i}(h_2) = B_{i}$, rotation
    renames $A_{\bullet}$ to $B_{\bullet}$, and $\muint{O1}$ renames
    $v_{i,j}$ to $v_{i^{\ast},j}$, the result is proven.

    Proofs that the monomial map works correctly for factors of $h_2$
    and $h_3$ for type $D$, and for all factors of type $A$, $B$, and
    $C$, have similar forms, and the result has been numerically checked
    for exceptional types. This proves the result in all cases.
  \end{proof}

  \subsection{Flip mutations}

  We now verify the longest flip mutation: the flip. The idea of the
  proof is similar to that of $\murot$. First, we put the coordinates of
  $\Qdv$ into a rectangular pattern, then apply mutations that respect
  \cite[Theorem~1.17]{fominzelevinsky1999}, showing that the resulting
  coordinates are of $\Qdh$. The most significant difficulty is to
  construct the element of $G$ holding the appropriate generalized
  minors, which we obtain by a result of Zickert.

  \begin{lemma}
    \label{lem:fgcs-exists-flip-works}
    \Cref{def:fgcs-flip-works} of \cref{def:fgcs} holds.
  \end{lemma}
  \begin{proof}
    This proof almost entirely focuses on $\muflipcore$. The prefix,
    $\muint{P}$, is a sequence of rotations as depicted in
    \cref{fig:Qdv-naming}. The suffix, $\muint{O3} \muint{O4}
    \muint{O5}$, is a reordering that switches the left and right sides.

    Let $\alpha = (g_0 N_{+}, g_1 N_{+}, g_2 N_{+}, g_3 N_{+}) \in
    \Conf_4^{\ast}(G/N_{+})$, and recall
    \cref{def:gluing-configurations}. For notation, let
    \begin{align*}
      \alpha_{012} &= (g_0 s_G N_{+}, g_1 N_{+}, g_2 N_{+}) & \alpha_{023} &= (g_0 N_{+}, g_2 N_{+}, g_3 N_{+}) \\
      \alpha_{123} &= (g_1 N_{+}, g_2 N_{+}, g_3 N_{+}) & \alpha_{013} &= (g_0 N_{+}, g_1 s_G N_{+}, g_3 N_{+}),
    \end{align*}
    with other $\alpha_{ijk}$ defined by rotation, and let $u_{ijk}$ be
    the element of $N_{-}$ in the canonical form of $\alpha_{ijk}$. As
    usual, let $h_{ij} = [ \lifttoG{w_0}^{-1} g_i^{-1} g_j]_{0}$.

    We must show that $\muflip$ takes $\Qdv$ to $\Qdh$. By an argument
    entirely analogous to that of
    \cref{lem:fgcs-exists-rot-works-nontrivial}, we assume that all edge
    coordinates of $\alpha_{123}$ and $\alpha_{013}$ are trivial. That
    is, every $h_{ij}$ is trivial except $h_{20}$ and $h_{02}$. We also
    appeal to \cref{rem:p-commutes-with-mutation}, and only consider
    $\mathcal{A}$-coordinates.

    By \cref{def:dv-dh-construction}, coordinates for $\Qdv$ come from
    $\alpha_{012}$ and $\alpha_{023}$. After applying $\muP$, these are
    rotated to $\alpha_{120}$ and $\alpha_{230}$. Treating
    $\alpha_{230}$ as the rotation of $\alpha_{302}$, the coordinates
    are given by \[ \genminoryz{w_k}{}{\word{i}_k}(u_{120}) \quad
      \text{and} \quad \genminoryz{w_k}{}{\word{i}_k}(u_{230}). \]

    We must show that $\muflipcore$ takes these to \[
      \genminoryz{w_k}{}{\word{i}_k}(u_{123}) \quad \text{and} \quad
      \genminoryz{w_k}{}{\word{i}_k}(u_{013}). \]

    Now consider $g = \Phi^{-1}(u_{123}) u_{130} h_{31}^{-1} h_{30}$. We
    also have $g = u_{120} h_{20} w_0(h_{12}) \Phi^{-1}(u_{023})$ by
    \cite[Proposition~5.14]{zickert2016}, and this will be identity
    which relates the two sides of the flip. All the coordinates we need
    are, up to edge coordinates, minors of $g$.

    \begin{align*}
      \genminoryz{w_k}{}{\word{i}_k}(u_{120}) &= \chi_{\fundamentalweight_{\word{i}_k}}\left(\left[ \dlifttoG{w_k^{-1}} [g]_{-} \right]_{0}\right) \\
      &= \chi_{\fundamentalweight_{\word{i}_k}}\left(\left[ \dlifttoG{w_k^{-1}} g [g]_{0}^{-1} \left( [g]_{0}^{-1} [g]_{+}^{-1} [g]_{0} \right) \right]_{0}\right) \\
      &= \chi_{\fundamentalweight_{\word{i}_k}}\left(\left[ \dlifttoG{w_k^{-1}} g \right]_{0} \right) / \chi_{\fundamentalweight_{\word{i}_k}} \left( [g]_{0} \right) \\
      &= \genminoryz{w_k}{}{\word{i}_k}(g) \frac{1}{\chi_{\fundamentalweight_{\word{i}_k}}(h_{20} w_0(h_{12}))} \\
      &= \genminoryz{w_k}{}{\word{i}_k}(g) \frac{1}{\chi_{\fundamentalweight_{\word{i}_k}}(h_{20}) \chi_{\fundamentalweight_{\word{i}_k}}(w_0(h_{12}))}
    \end{align*}
    \begin{align*}
      \genminoryz{w_k}{}{\word{i}_k}(u_{302}) &= \genminoryz{w_k}{}{\word{i}_k}( (\Phi \Psi \Phi \Psi)(u_{023}) ) \frac{\genminoryz{}{}{\word{i}_k}(w_0(h_{02}) h_{23})}{\genminoryz{w_k}{w_k}{\word{i}_k}(w_0(h_{02}) h_{23})} \\
      &= \genminoryz{w_k}{}{\word{i}_k}( [[\lifttoG{w_0} u_{023}]_{+} \lifttoG{w_0}]_{-}) \frac{\genminoryz{w_k}{w_k}{\word{i}_k}(w_0(h_{02}) h_{23})}{\genminoryz{w_0}{w_0}{\word{i}_k}(w_0(h_{02}) h_{23})} \\
      &= \genminoryz{w_k}{}{\word{i}_k}([\lifttoG{w_0} u_{023}]_{-}^{-1}) \frac{\genminoryz{}{}{\word{i}_k}([\lifttoG{w_0} u_{023}]_{0}) \genminoryz{w_k}{w_k}{\word{i}_k}(w_0(h_{02}) h_{23})}{\genminoryz{w_k}{w_k}{\word{i}_k}([\lifttoG{w_0} u_{023}]_{0}) \genminoryz{w_0}{w_0}{\word{i}_k}(w_0(h_{02}) h_{23})} \\
      &= \genminoryz{}{w_k^{\ast}}{\word{i}_k^{\ast}}([g]_{+}) \frac{\genminoryz{}{}{\word{i}_k}([\lifttoG{w_0} u_{023}]_{0}) \genminoryz{w_k}{w_k}{\word{i}_k}(w_0(h_{02}) h_{23})}{\genminoryz{w_k}{w_k}{\word{i}_k}([\lifttoG{w_0} u_{023}]_{0}) \genminoryz{w_0}{w_0}{\word{i}_k}(w_0(h_{02}) h_{23})} \\
      &= \genminoryz{}{w_k^{\ast}}{\word{i}_k^{\ast}}(g) \frac{\genminoryz{}{}{\word{i}_k}([\lifttoG{w_0} u_{023}]_{0}) \genminoryz{w_k}{w_k}{\word{i}_k}(w_0(h_{02}) h_{23})}{ \genminoryz{}{}{\word{i}_{k}^{\ast}}([g]_{0}) \genminoryz{w_k}{w_k}{\word{i}_k}([\lifttoG{w_0} u_{023}]_{0}) \genminoryz{w_0}{w_0}{\word{i}_k}(w_0(h_{02}) h_{23})} \\
      &= \genminoryz{}{w_k^{\ast}}{\word{i}_k^{\ast}}(g) \frac{\chi_{\fundamentalweight_{\word{i}_k}}\left(w_0(h_{20} h_{30}^{-1}) h_{12}\right)}{\genminoryz{w_k}{w_k}{\word{i}_k}(h_{03})} \\
      &= \genminoryz{}{w_k^{\ast}}{\word{i}_k^{\ast}}(g) \frac{\chi_{\fundamentalweight_{\word{i}_k}}\left(w_0(h_{30}^{-1}) h_{12}\right)}{\chi_{\fundamentalweight_{\word{i}_k}}(h_{02}) \genminoryz{w_k}{w_k}{\word{i}_k}(h_{03})}
    \end{align*}
    \begin{align*}
      \genminoryz{w_k}{}{\word{i}_k}(u_{123}) &= \genminoryz{w_k}{}{\word{i}_k}([g \lifttoG{w_0}]_{-}) \\
      &= \chi_{\fundamentalweight_{\word{i}_k}}\left(\left[ \dlifttoG{w_k^{-1}} g \lifttoG{w_0} [g \lifttoG{w_0}]_{+}^{-1} [g \lifttoG{w_0}]_{0}^{-1} \right]_{0}\right) \\
      &= \genminoryz{w_k}{w_0}{\word{i}_k}(g) \frac{1}{\genminoryz{}{w_0}{\word{i}_k}(g)} \\
      &= \genminoryz{w_k}{w_0}{\word{i}_k}(g) \frac{1}{\chi_{\fundamentalweight_{\word{i}_k}}([\lifttoG{w_0} [\lifttoG{w_0}^{-1} u_{123}]_{-}]_{0} w_0(h_{31}^{-1} h_{30}))} \\
      &= \genminoryz{w_k}{w_0}{\word{i}_k}(g) \frac{1}{\chi_{\fundamentalweight_{\word{i}_k}}([\lifttoG{w_0}^{-1} u_{123}]_{0}^{-1} w_0(h_{31}^{-1} h_{30}))} \\
      &= \genminoryz{w_k}{w_0}{\word{i}_k}(g) \frac{1}{\chi_{\fundamentalweight_{\word{i}_k}}(w_0(h_{12} h_{30}) h_{23})}
    \end{align*}
    \begin{align*}
      \genminoryz{w_0}{w_k}{\word{i}_k}(u_{130}) &= \genminoryz{w_0}{w_k}{\word{i}_k}(\Phi^{-1}(u_{123}) u_{130} h_{31}^{-1} h_{30}) \frac{1}{\genminoryz{w_k}{w_k}{\word{i}_k}(h_{31}^{-1} h_{30})} \\
      &= \genminoryz{w_0}{w_k}{\word{i}_k}(g) \frac{1}{\genminoryz{w_k}{w_k}{\word{i}_k}(h_{31}^{-1} h_{30})}
    \end{align*}

    Amalgamating $\alpha_{120}$ and $\alpha_{302}$ and labeling
    coordinates by minors of $g$ (up to edge coordinates) gives
    \cref{fig:flip-C3-120-302}. This includes the central edge because
    $[g]_{0} = h_{20} w_0(h_{12})$. This is not $\Qdv$, but the
    difference is only a matter of rotations and twistings.

    Note that the factors of
    $\chi_{\fundamentalweight_{\word{i}_k}}(h_{20})$ and
    $\chi_{\fundamentalweight_{\word{i}_k}}(h_{02})$ in the assignments
    for $\genminoryz{w_k}{}{\word{i}_k}(u_{120})$ and
    $\genminoryz{w_k}{}{\word{i}_k}(u_{302})$ are exactly cancelled out
    by the $\genminor{}{i}(h_2)$ factor of $\monomialmap$ from
    \cref{def:conf_3_star-to-seed-torus}. Therefore we are justified in
    ignoring edge coordinates for the remainder of the proof since we
    assume the others to be trivial.

    \begin{figure}
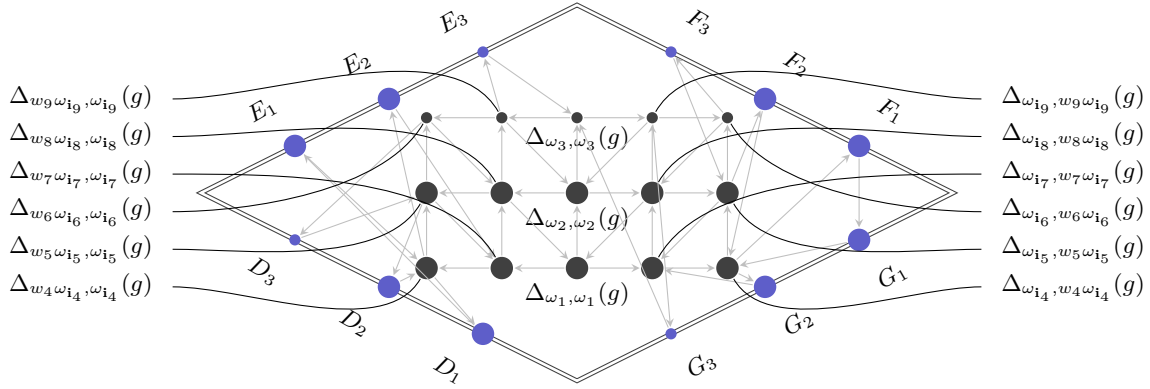

      \centering \includestandalone[mode=image|tex]{fig/flip-C3-120-302}

      \caption{Coordinates from $\alpha_{120}$ and $\alpha_{302}$.
        \label{fig:flip-C3-120-302}
      }
    \end{figure}

    After applying $\left( \murottwistR \right)$ the result is
    $\muP(\Qdv)$, with the coordinates of $\alpha_{120}$ and
    $\alpha_{230}$, as in \cref{fig:flip-C3-120-230}. In this figure,
    since each $\genminoryz{w_j}{w_k}{\word{i}_k}(g)$ is equivalent to
    some $\genminoryz{c^n}{c^m}{\word{i}_k}(g)$, we label vertices by
    the pair $(n,m)$. Further, taking into account that $c^{h} = w_0^2$
    which lifts to an element of $H$, we may consider $n$ and $m$ modulo
    $h$.

    \begin{figure}
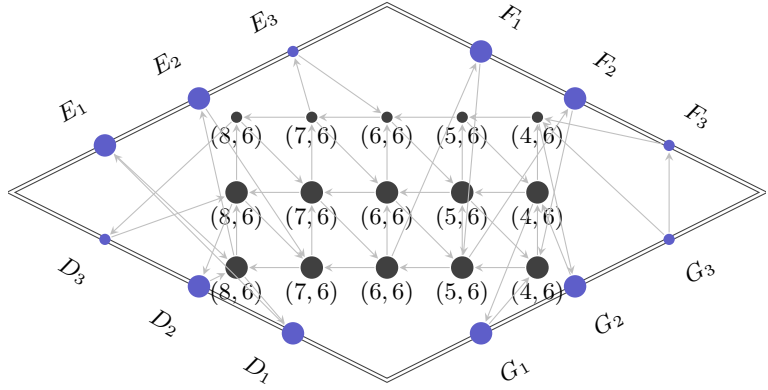

      \centering \includestandalone[mode=image|tex]{fig/flip-C3-120-230}

      \caption{Coordinates from $\alpha_{120}$ and $\alpha_{230}$,
        labeled by $(n,m)$.
        \label{fig:flip-C3-120-230}
      }
    \end{figure}

    As with $\murot$, each mutation in $\muflipcore$ takes
    $\genminoryz{c^a}{c^b}{\word{i}_k}(g)$ to
    $\genminoryz{c^{a-1}}{c^{b-1}}{\word{i}_k}(g)$ by
    \cite[Theorem~1.17]{fominzelevinsky1999} and
    \cref{lem:thm-1.17-for-repeated-c}. That results in
    \cref{fig:flip-C3-120-230-post-flipcore}.

    \begin{figure}
      \centering \includestandalone[mode=image|tex]{fig/flip-C3-120-230-post-flipcore}

      \caption{Coordinates after applying $\muflipcore$, labeled by
        $(n,m)$.
        \label{fig:flip-C3-120-230-post-flipcore}
      }
    \end{figure}

    Converting $(n,m)$ back to $\genminoryz{c^n}{c^m}{\word{i}_k}(g)$,
    we obtain \cref{fig:flip-C3-123-013}. By the calculations above,
    these are coordinates of $u_{123}$ on the right and $u_{013}$ (as
    rotated $u_{130}$) on the left.

    \begin{figure}
      \centering \includestandalone[mode=image|tex]{fig/flip-C3-123-013}

      \caption{Coordinates after applying $\muflipcore$.
        \label{fig:flip-C3-123-013}
      }
    \end{figure}

    The final combination of $\muint{O3} \muint{O4} \muint{O5}$ serve to
    swap the left and right halves by \cref{lem:rowsigmaG-does-sigmaG}.
    This completes the proof.
  \end{proof}

  \section{Proof of \cref{thm:fgcs-exists} for semisimple
    \texorpdfstring{$G$}{G}}

  \label{sec:products}

  We now expand from ``simple'' to ``semisimple''. Since the algorithm
  of \cref{sec:main-result} really only needed the information of $G$'s
  Dynkin diagram, and Dynkin diagrams behave very nicely with respect to
  products, the construction easily generalizes.

  \begin{lemma}
    \label{lem:fgcs-exists-for-products}
    Let $G$, a semisimple Lie group over $\mathbb{C}$, be $G = G_1
    \oplus G_2 \oplus \dotsb \oplus G_n$, with each $G_i$ a simple Lie
    group over $\mathbb{C}$ admitting quivers $Q_i$ with Fock--Goncharov
    coordinate structures $(\murot)_i$, $(\muflip)_i$, and
    $\conftoquiv_i$.

    Then for $G$, there is a quiver $Q$ which also carries a
    Fock--Goncharov coordinate structure.
  \end{lemma}
  \begin{proof}
    We must construct $Q$, $\murot$, $\muflip$, $\conftoquiv$, then show
    that all requirements of \cref{def:fgcs} are satisfied. This follows
    entirely from the fact that the Dynkin diagram of the direct product
    of groups is the disjoint union of Dynkin diagrams of the factors.

    For $Q$, take the disjoint union of each $Q_i$, with each
    $\facemap_{\Delta}(Q)$ being the disjoint union of all
    $\facemap_{\Delta}(Q_i)$. This is triangular since each component
    is.

    The mutations $\murot$ and $\muflip$ are concatenations of all
    $(\murot)_i$ and $(\muflip)_i$ respectively. The ordering of these
    components are irrelevant, because quiver mutations at disconnected
    vertices commute and $Q$ is exactly a disjoint union. Any
    concatenation will operate componentwise on the pieces of the
    disjoint union. This satisfies \cref{def:fgcs-1}.

    By the aforementioned fact of Dynkin diagrams, \cref{def:fgcs-2} is
    also satisfied.

    The map $\conftoquiv$ is defined as a product of each
    $\conftoquiv_i$ in the following manner. Let $\alpha_i$ and
    $\alpha_j$ be any two elements of the roots system for $G$
    associated to different factors. Then the nodes corresponding to
    those roots are disconnected in the Dynkin diagram for $G$, so the
    $A_{ij} = 0$ in the Cartan matrix.

    By e.g. \cite[Proposition~2.95]{knapp1996}, the bracket of any
    $\set{e_i, f_i, h_i}$ and $\set{e_j, f_j, h_j}$ vanishes, so any
    $\set{x_i, y_i, \torush_i}$ and $\set{x_j, y_j, \torush_j}$ commute.
    Letting $\alpha = (g_0 N_{+}, g_1 N_{+}, g_2 N_{+}, g_3 N_{+})$, by
    sufficient genericity we may factorize any $g_i$ as a product of
    $x_k, y_k, \torush_k$ elements, and may therefore rearrange these
    factors so that \[ g_i = (g_i)_1 (g_i)_2 \dotsb (g_i)_n, \qquad
      (g_i)_j \in G_j. \] Then the maps $\operatorname{rot}$ and
    $\Psi_{ij}$ trivially factor through the decomposition of $\alpha$
    into \[ \prod_{i = 1}^{n} \left[ \strut \alpha_i = ((g_0)_i, (g_1)_i
      N_{+}, (g_2)_i N_{+}, (g_3)_i N_{+}) \right]. \] That is, it is
    obvious that $\prod_i \operatorname{rot}(\alpha_i) =
    \operatorname{rot}(\prod_i \alpha_i)$. This, together with the
    construction of $Q$, shows that \cref{def:fgcs-M-is-birational-equivalence,def:fgcs-rot-works,def:fgcs-flip-works}
    are satisfied since they are satisfied for each $G_i$.
  \end{proof}

  \section{Examples}

  \label{sec:examples}

  Here we provide examples, in consistent notation, of $Q$, $\murot$,
  $\muflip$, and $\conftoquiv$ for various types (for $F_4$ see the
  examples of \cref{sec:main-result}). Our notation is intended to agree
  with that of \cite{zickert2016}, with the caveat that the mutations we
  present are longer but do not require vertex renaming (see
  \cref{sec:what-does-the-twist-do}). All Dynkin diagrams are taken from
  \cite[Figure~2.4]{knapp1996}, equivalently
  \cite[Plates~II--IX]{bourbaki2002_46} (ignoring the extending nodes).
  In each example, $\alpha = (h_1, h_2, h_3, u)$, where the coordinates
  of $h_i$ are $\set{h_{ij}}$, as in \cref{ex:f4-part-6}.

  These results were obtained by the \texttt{latex-Q},
  \texttt{print-murot}, \texttt{print-muflip}, and \texttt{print-M}
  commands of \cite{gilles2020sw}.

  %
  %
  %
  %

  \subsection{\texorpdfstring{$A_5$}{A\_5}}

  We use the Dynkin diagram \raisebox{-0.7ex}{\includestandalone[mode=image|tex,height=3ex]{fig/dynkin-A5}}. The Coxeter element is $c =
  \set{3, 2, 4, 1, 5}$, and the partitions are $T_0 = \set{3}$, $T_1 =
  \set{2, 4}$, $T_2 = \set{1, 5}$. The Dynkin-type quiver follows.

  {\centering \includestandalone[mode=image|tex]{fig/dynkin-type-A5}

  } $Q_{A_5}$ is given in \cref{fig:QA5}.

  \begin{figure}
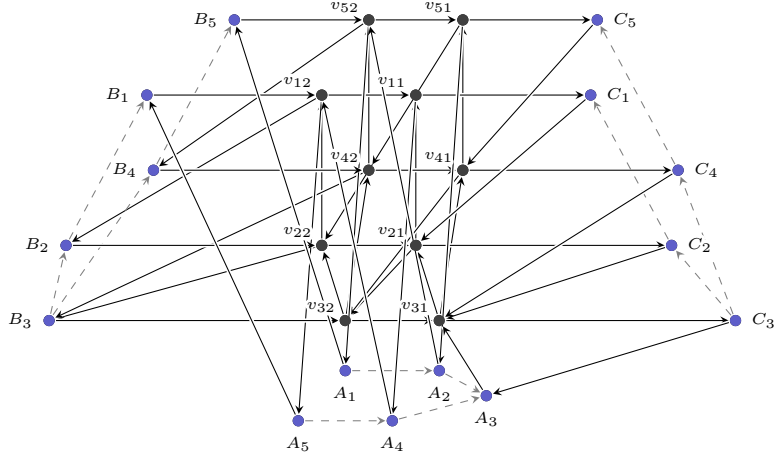

    \centering \includestandalone[mode=image|tex]{fig/QA5}

    \caption{$Q$ for $A_5$.
      \label{fig:QA5}
    }
  \end{figure}

  \inlmutseq{\murot}{ $v_{3 1}$, $v_{2 1}$, $v_{4 1}$, $v_{1 1}$, $v_{5
      1}$, $v_{3 2}$, $v_{2 2}$, $v_{4 2}$, $v_{1 2}$, $v_{5 2}$, $v_{3
      1}$, $v_{2 1}$, $v_{4 1}$, $v_{1 1}$, $v_{5 1}$, $v_{3 1}$, $v_{2
      1}$, $v_{4 1}$, $v_{1 1}$, $v_{5 1}$, $v_{3 1}$, $v_{2 1}$, $v_{4
      1}$, $v_{1 1}$, $v_{5 1}$, $v_{3 1}$, $v_{2 1}$, $v_{4 1}$, $v_{1
      1}$, $v_{5 1}$, $v_{3 1}$, $v_{2 1}$, $v_{4 1}$, $v_{1 1}$, $v_{5
      1}$, $v_{3 2}$, $v_{2 2}$, $v_{4 2}$, $v_{1 2}$, $v_{5 2}$, $v_{3
      2}$, $v_{2 2}$, $v_{4 2}$, $v_{1 2}$, $v_{5 2}$, $v_{3 2}$, $v_{2
      2}$, $v_{4 2}$, $v_{1 2}$, $v_{5 2}$, $v_{3 2}$, $v_{2 2}$, $v_{4
      2}$, $v_{1 2}$, $v_{5 2}$, $v_{3 1}$, $v_{3 2}$, $v_{3 1}$, $v_{3
      2}$, $v_{3 1}$, $v_{2 1}$, $v_{4 1}$, $v_{2 2}$, $v_{4 2}$, $v_{2
      1}$, $v_{4 1}$, $v_{2 2}$, $v_{4 2}$, $v_{2 1}$, $v_{4 1}$, $v_{1
      1}$, $v_{5 1}$, $v_{1 2}$, $v_{5 2}$, $v_{1 1}$, $v_{5 1}$, $v_{1
      2}$, $v_{5 2}$, $v_{1 1}$, $v_{5 1}$ }

  \inlmutseq{\muflip}{$w_{1 5}$, $w_{5 5}$, $w_{2 5}$, $w_{4 5}$, $w_{3
      5}$, $w_{1 4}$, $w_{5 4}$, $w_{2 4}$, $w_{4 4}$, $w_{3 4}$, $w_{1
      5}$, $w_{5 5}$, $w_{2 5}$, $w_{4 5}$, $w_{3 5}$, $w_{1 2}$, $w_{5
      2}$, $w_{2 2}$, $w_{4 2}$, $w_{3 2}$, $w_{1 1}$, $w_{5 1}$, $w_{2
      1}$, $w_{4 1}$, $w_{3 1}$, $w_{1 2}$, $w_{5 2}$, $w_{2 2}$, $w_{4
      2}$, $w_{3 2}$, $w_{3 5}$, $w_{2 5}$, $w_{4 5}$, $w_{1 5}$, $w_{5
      5}$, $w_{3 4}$, $w_{2 4}$, $w_{4 4}$, $w_{1 4}$, $w_{5 4}$, $w_{3
      3}$, $w_{2 3}$, $w_{4 3}$, $w_{1 3}$, $w_{5 3}$, $w_{3 2}$, $w_{2
      2}$, $w_{4 2}$, $w_{1 2}$, $w_{5 2}$, $w_{3 1}$, $w_{2 1}$, $w_{4
      1}$, $w_{1 1}$, $w_{5 1}$, $w_{3 5}$, $w_{2 5}$, $w_{4 5}$, $w_{1
      5}$, $w_{5 5}$, $w_{3 4}$, $w_{2 4}$, $w_{4 4}$, $w_{1 4}$, $w_{5
      4}$, $w_{3 3}$, $w_{2 3}$, $w_{4 3}$, $w_{1 3}$, $w_{5 3}$, $w_{3
      2}$, $w_{2 2}$, $w_{4 2}$, $w_{1 2}$, $w_{5 2}$, $w_{3 5}$, $w_{2
      5}$, $w_{4 5}$, $w_{1 5}$, $w_{5 5}$, $w_{3 4}$, $w_{2 4}$, $w_{4
      4}$, $w_{1 4}$, $w_{5 4}$, $w_{3 3}$, $w_{2 3}$, $w_{4 3}$, $w_{1
      3}$, $w_{5 3}$, $w_{3 5}$, $w_{2 5}$, $w_{4 5}$, $w_{1 5}$, $w_{5
      5}$, $w_{3 4}$, $w_{2 4}$, $w_{4 4}$, $w_{1 4}$, $w_{5 4}$, $w_{3
      5}$, $w_{2 5}$, $w_{4 5}$, $w_{1 5}$, $w_{5 5}$, $w_{1 1}$, $w_{5
      1}$, $w_{1 2}$, $w_{5 2}$, $w_{1 3}$, $w_{5 3}$, $w_{1 4}$, $w_{5
      4}$, $w_{1 5}$, $w_{5 5}$, $w_{1 1}$, $w_{5 1}$, $w_{1 2}$, $w_{5
      2}$, $w_{1 3}$, $w_{5 3}$, $w_{1 4}$, $w_{5 4}$, $w_{1 1}$, $w_{5
      1}$, $w_{1 2}$, $w_{5 2}$, $w_{1 3}$, $w_{5 3}$, $w_{1 1}$, $w_{5
      1}$, $w_{1 2}$, $w_{5 2}$, $w_{1 1}$, $w_{5 1}$, $w_{1 5}$, $w_{5
      5}$, $w_{1 4}$, $w_{5 4}$, $w_{1 3}$, $w_{5 3}$, $w_{1 2}$, $w_{5
      2}$, $w_{1 1}$, $w_{5 1}$, $w_{2 1}$, $w_{4 1}$, $w_{2 2}$, $w_{4
      2}$, $w_{2 3}$, $w_{4 3}$, $w_{2 4}$, $w_{4 4}$, $w_{2 5}$, $w_{4
      5}$, $w_{2 1}$, $w_{4 1}$, $w_{2 2}$, $w_{4 2}$, $w_{2 3}$, $w_{4
      3}$, $w_{2 4}$, $w_{4 4}$, $w_{2 1}$, $w_{4 1}$, $w_{2 2}$, $w_{4
      2}$, $w_{2 3}$, $w_{4 3}$, $w_{2 1}$, $w_{4 1}$, $w_{2 2}$, $w_{4
      2}$, $w_{2 1}$, $w_{4 1}$, $w_{2 5}$, $w_{4 5}$, $w_{2 4}$, $w_{4
      4}$, $w_{2 3}$, $w_{4 3}$, $w_{2 2}$, $w_{4 2}$, $w_{2 1}$, $w_{4
      1}$, $w_{3 1}$, $w_{3 2}$, $w_{3 3}$, $w_{3 4}$, $w_{3 5}$, $w_{3
      1}$, $w_{3 2}$, $w_{3 3}$, $w_{3 4}$, $w_{3 1}$, $w_{3 2}$, $w_{3
      3}$, $w_{3 1}$, $w_{3 2}$, $w_{3 1}$, $w_{3 5}$, $w_{3 4}$, $w_{3
      3}$, $w_{3 2}$, $w_{3 1}$, $w_{1 4}$, $w_{5 4}$, $w_{1 5}$, $w_{5
      5}$, $w_{1 4}$, $w_{5 4}$, $w_{1 5}$, $w_{5 5}$, $w_{1 4}$, $w_{5
      4}$, $w_{2 4}$, $w_{4 4}$, $w_{2 5}$, $w_{4 5}$, $w_{2 4}$, $w_{4
      4}$, $w_{2 5}$, $w_{4 5}$, $w_{2 4}$, $w_{4 4}$, $w_{3 4}$, $w_{3
      5}$, $w_{3 4}$, $w_{3 5}$, $w_{3 4}$, $w_{1 1}$, $w_{5 1}$, $w_{1
      2}$, $w_{5 2}$, $w_{1 1}$, $w_{5 1}$, $w_{1 2}$, $w_{5 2}$, $w_{1
      1}$, $w_{5 1}$, $w_{2 1}$, $w_{4 1}$, $w_{2 2}$, $w_{4 2}$, $w_{2
      1}$, $w_{4 1}$, $w_{2 2}$, $w_{4 2}$, $w_{2 1}$, $w_{4 1}$, $w_{3
      1}$, $w_{3 2}$, $w_{3 1}$, $w_{3 2}$, $w_{3 1}$ }

  $\conftoquiv(\alpha)$ gives
  \begin{equation*}
    \begin{gathered}
      \begin{aligned}
        A_1 &= h_{11} &\quad A_2 &= h_{12} &\quad A_3 &= h_{13} &\quad A_4 &= h_{14} &\quad A_5 &= h_{15} \\
        B_1 &= h_{21} &\quad B_2 &= h_{22} &\quad B_3 &= h_{23} &\quad B_4 &= h_{24} &\quad B_5 &= h_{25} \\
        C_1 &= h_{35} &\quad C_2 &= h_{34} &\quad C_3 &= h_{33} &\quad C_4 &= h_{32} &\quad C_5 &= h_{31} \\
      \end{aligned}
      \\
      \begin{aligned}
        v_{3 1} &= \genminor{w_{6}}{3}(u) \cdot h_{23} &v_{3 2} &= \genminor{w_{11}}{3}(u) \cdot h_{23} \\
        v_{2 1} &= \genminor{w_{7}}{2}(u) \cdot \frac{h_{13} h_{22} }{h_{14}} &v_{2 2} &= \genminor{w_{12}}{2}(u) \cdot \frac{h_{13} h_{22} }{h_{14}} \\
        v_{4 1} &= \genminor{w_{8}}{4}(u) \cdot \frac{h_{13} h_{24} }{h_{12}} &v_{4 2} &= \genminor{w_{13}}{4}(u) \cdot \frac{h_{13} h_{24} }{h_{12}} \\
        v_{1 1} &= \genminor{w_{9}}{1}(u) \cdot \frac{h_{13} h_{21} }{h_{15}} &v_{1 2} &= \genminor{w_{14}}{1}(u) \cdot \frac{h_{14} h_{21} }{h_{15}} \\
        v_{5 1} &= \genminor{w_{10}}{5}(u) \cdot \frac{h_{13} h_{25} }{h_{11}} &v_{5 2} &= \genminor{w_{15}}{5}(u) \cdot \frac{h_{12} h_{25} }{h_{11}}
      \end{aligned}
    \end{gathered}
  \end{equation*}

  \subsection{\texorpdfstring{$B_3$}{B\_3}}

  We use the Dynkin diagram \raisebox{-0.1ex}{\includestandalone[mode=image|tex,height=1.5ex]{fig/dynkin-B3}}. This is linear, so the
  Coxeter element and partitions are trivially $c = \set{1, 2, 3}$ and
  $T_0 = \set{1}, T_1 = \set{2}, T_2 = \set{3}$. The Dynkin-type quiver
  follows (recall \cref{rem:sizes-get-switched}).

  {\centering \includestandalone[mode=image|tex]{fig/dynkin-type-B3}

  } $Q_{B_3}$ is given in \cref{fig:QB3}.

  \begin{figure}
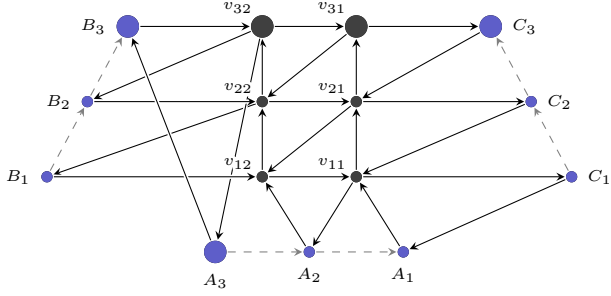

    \centering \includestandalone[mode=image|tex]{fig/QB3}

    \caption{$Q$ for $B_3$.
      \label{fig:QB3}
    }
  \end{figure}

  \inlmutseq{\murot}{ $v_{1 1}$, $v_{2 1}$, $v_{3 1}$, $v_{1 2}$, $v_{2
      2}$, $v_{3 2}$, $v_{1 1}$, $v_{2 1}$, $v_{3 1}$, $v_{1 1}$, $v_{1
      2}$, $v_{1 1}$, $v_{1 2}$, $v_{1 1}$, $v_{2 1}$, $v_{2 2}$, $v_{2
      1}$, $v_{2 2}$, $v_{2 1}$, $v_{3 1}$, $v_{3 2}$, $v_{3 1}$, $v_{3
      2}$, $v_{3 1}$ }

  \inlmutseq{\muflip}{ $w_{3 5}$, $w_{2 5}$, $w_{1 5}$, $w_{3 4}$, $w_{2
      4}$, $w_{1 4}$, $w_{3 5}$, $w_{2 5}$, $w_{1 5}$, $w_{3 2}$, $w_{2
      2}$, $w_{1 2}$, $w_{3 1}$, $w_{2 1}$, $w_{1 1}$, $w_{3 2}$, $w_{2
      2}$, $w_{1 2}$, $w_{1 5}$, $w_{2 5}$, $w_{3 5}$, $w_{1 4}$, $w_{2
      4}$, $w_{3 4}$, $w_{1 3}$, $w_{2 3}$, $w_{3 3}$, $w_{1 2}$, $w_{2
      2}$, $w_{3 2}$, $w_{1 1}$, $w_{2 1}$, $w_{3 1}$, $w_{1 5}$, $w_{2
      5}$, $w_{3 5}$, $w_{1 4}$, $w_{2 4}$, $w_{3 4}$, $w_{1 3}$, $w_{2
      3}$, $w_{3 3}$, $w_{1 2}$, $w_{2 2}$, $w_{3 2}$, $w_{1 5}$, $w_{2
      5}$, $w_{3 5}$, $w_{1 4}$, $w_{2 4}$, $w_{3 4}$, $w_{1 3}$, $w_{2
      3}$, $w_{3 3}$, $w_{1 5}$, $w_{2 5}$, $w_{3 5}$, $w_{1 4}$, $w_{2
      4}$, $w_{3 4}$, $w_{1 5}$, $w_{2 5}$, $w_{3 5}$, $w_{3 1}$, $w_{3
      2}$, $w_{3 3}$, $w_{3 4}$, $w_{3 5}$, $w_{3 1}$, $w_{3 2}$, $w_{3
      3}$, $w_{3 4}$, $w_{3 1}$, $w_{3 2}$, $w_{3 3}$, $w_{3 1}$, $w_{3
      2}$, $w_{3 1}$, $w_{3 5}$, $w_{3 4}$, $w_{3 3}$, $w_{3 2}$, $w_{3
      1}$, $w_{2 1}$, $w_{2 2}$, $w_{2 3}$, $w_{2 4}$, $w_{2 5}$, $w_{2
      1}$, $w_{2 2}$, $w_{2 3}$, $w_{2 4}$, $w_{2 1}$, $w_{2 2}$, $w_{2
      3}$, $w_{2 1}$, $w_{2 2}$, $w_{2 1}$, $w_{2 5}$, $w_{2 4}$, $w_{2
      3}$, $w_{2 2}$, $w_{2 1}$, $w_{1 1}$, $w_{1 2}$, $w_{1 3}$, $w_{1
      4}$, $w_{1 5}$, $w_{1 1}$, $w_{1 2}$, $w_{1 3}$, $w_{1 4}$, $w_{1
      1}$, $w_{1 2}$, $w_{1 3}$, $w_{1 1}$, $w_{1 2}$, $w_{1 1}$, $w_{1
      5}$, $w_{1 4}$, $w_{1 3}$, $w_{1 2}$, $w_{1 1}$, $w_{3 4}$, $w_{3
      5}$, $w_{3 4}$, $w_{3 5}$, $w_{3 4}$, $w_{2 4}$, $w_{2 5}$, $w_{2
      4}$, $w_{2 5}$, $w_{2 4}$, $w_{1 4}$, $w_{1 5}$, $w_{1 4}$, $w_{1
      5}$, $w_{1 4}$, $w_{3 1}$, $w_{3 2}$, $w_{3 1}$, $w_{3 2}$, $w_{3
      1}$, $w_{2 1}$, $w_{2 2}$, $w_{2 1}$, $w_{2 2}$, $w_{2 1}$, $w_{1
      1}$, $w_{1 2}$, $w_{1 1}$, $w_{1 2}$, $w_{1 1}$ }

  $\conftoquiv(\alpha)$ gives
  \begin{equation*}
    \begin{gathered}
      \begin{aligned}
        A_1 &= h_{11} &\qquad A_2 &= h_{12} &\qquad A_3 &= h_{13} \\
        B_1 &= h_{11} &\qquad B_2 &= h_{12} &\qquad B_3 &= h_{13} \\
        C_1 &= h_{11} &\qquad C_2 &= h_{12} &\qquad C_3 &= h_{13} \\
      \end{aligned}
      \\
      \begin{aligned}
        v_{1 1} &= \genminor{w_{4}}{1}(u) \cdot h_{21} & v_{2 1} &= \genminor{w_{5}}{2}(u) \cdot \frac{h_{11} h_{22} }{h_{12}} & v_{3 1} &= \genminor{w_{6}}{3}(u) \cdot \frac{h_{11} h_{23} }{h_{13}} \\
        v_{1 2} &= \genminor{w_{7}}{1}(u) \cdot \frac{h_{12} h_{21} }{h_{11}} & v_{2 2} &= \genminor{w_{8}}{2}(u) \cdot h_{11} h_{22} & v_{3 2} &= \genminor{w_{9}}{3}(u) \cdot \frac{h_{12} h_{23} }{h_{13}} \\
      \end{aligned}
    \end{gathered}
  \end{equation*}

  \subsection{\texorpdfstring{$D_5$}{D\_5}}

  Based on the Dynkin diagram \raisebox{-0.7ex}{\includestandalone[mode=image|tex,height=3ex]{fig/dynkin-D5}}, the following quiver is
  of well-rooted Dynkin type for $D_5$ (recall that $\sigma_G$ is
  trivial for type $D_{2n}$ and non-trivial for type $D_{2n+1}$). It
  admits an induced Coxeter element $c = \set{1, 2, 3, 4, 5}$, with
  partitions \[ T_0 = \set{1}, \quad T_1 = \set{2}, \quad T_2 = \set{3},
    \quad T_3 = \set{4, 5}. \]

  {\centering \includestandalone[mode=image|tex]{fig/dynkin-type-D5}

  } $Q_{D_5}$ is given in \cref{fig:QD5}.

  \begin{figure}
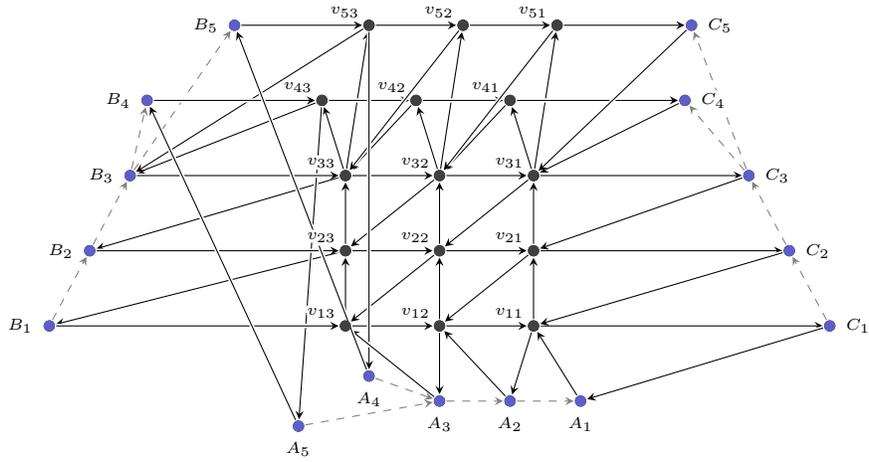

    \centering \includestandalone[mode=image|tex]{fig/QD5}

    \caption{$Q$ for $D_5$.
      \label{fig:QD5}
    }
  \end{figure}

  \inlmutseq{\murot}{ $v_{1 1}$, $v_{2 1}$, $v_{3 1}$, $v_{4 1}$, $v_{5
      1}$, $v_{1 2}$, $v_{2 2}$, $v_{3 2}$, $v_{4 2}$, $v_{5 2}$, $v_{1
      1}$, $v_{2 1}$, $v_{3 1}$, $v_{4 1}$, $v_{5 1}$, $v_{1 3}$, $v_{2
      3}$, $v_{3 3}$, $v_{4 3}$, $v_{5 3}$, $v_{1 2}$, $v_{2 2}$, $v_{3
      2}$, $v_{4 2}$, $v_{5 2}$, $v_{1 1}$, $v_{2 1}$, $v_{3 1}$, $v_{4
      1}$, $v_{5 1}$, $v_{1 1}$, $v_{2 1}$, $v_{3 1}$, $v_{4 1}$, $v_{5
      1}$, $v_{1 1}$, $v_{2 1}$, $v_{3 1}$, $v_{4 1}$, $v_{5 1}$, $v_{1
      1}$, $v_{2 1}$, $v_{3 1}$, $v_{4 1}$, $v_{5 1}$, $v_{1 1}$, $v_{2
      1}$, $v_{3 1}$, $v_{4 1}$, $v_{5 1}$, $v_{1 1}$, $v_{2 1}$, $v_{3
      1}$, $v_{4 1}$, $v_{5 1}$, $v_{1 2}$, $v_{2 2}$, $v_{3 2}$, $v_{4
      2}$, $v_{5 2}$, $v_{1 2}$, $v_{2 2}$, $v_{3 2}$, $v_{4 2}$, $v_{5
      2}$, $v_{1 2}$, $v_{2 2}$, $v_{3 2}$, $v_{4 2}$, $v_{5 2}$, $v_{1
      2}$, $v_{2 2}$, $v_{3 2}$, $v_{4 2}$, $v_{5 2}$, $v_{1 2}$, $v_{2
      2}$, $v_{3 2}$, $v_{4 2}$, $v_{5 2}$, $v_{1 3}$, $v_{2 3}$, $v_{3
      3}$, $v_{4 3}$, $v_{5 3}$, $v_{1 3}$, $v_{2 3}$, $v_{3 3}$, $v_{4
      3}$, $v_{5 3}$, $v_{1 3}$, $v_{2 3}$, $v_{3 3}$, $v_{4 3}$, $v_{5
      3}$, $v_{1 3}$, $v_{2 3}$, $v_{3 3}$, $v_{4 3}$, $v_{5 3}$, $v_{1
      3}$, $v_{2 3}$, $v_{3 3}$, $v_{4 3}$, $v_{5 3}$, $v_{1 1}$, $v_{1
      2}$, $v_{1 3}$, $v_{1 1}$, $v_{1 2}$, $v_{1 1}$, $v_{1 3}$, $v_{1
      2}$, $v_{1 1}$, $v_{2 1}$, $v_{2 2}$, $v_{2 3}$, $v_{2 1}$, $v_{2
      2}$, $v_{2 1}$, $v_{2 3}$, $v_{2 2}$, $v_{2 1}$, $v_{3 1}$, $v_{3
      2}$, $v_{3 3}$, $v_{3 1}$, $v_{3 2}$, $v_{3 1}$, $v_{3 3}$, $v_{3
      2}$, $v_{3 1}$, $v_{4 1}$, $v_{5 1}$, $v_{4 2}$, $v_{5 2}$, $v_{4
      3}$, $v_{5 3}$, $v_{4 1}$, $v_{5 1}$, $v_{4 2}$, $v_{5 2}$, $v_{4
      1}$, $v_{5 1}$, $v_{4 3}$, $v_{5 3}$, $v_{4 2}$, $v_{5 2}$, $v_{4
      1}$, $v_{5 1}$ }

  \inlmutseq{\muflip}{ $w_{4 7}$, $w_{5 7}$, $w_{3 7}$, $w_{2 7}$, $w_{1
      7}$, $w_{4 6}$, $w_{5 6}$, $w_{3 6}$, $w_{2 6}$, $w_{1 6}$, $w_{4
      7}$, $w_{5 7}$, $w_{3 7}$, $w_{2 7}$, $w_{1 7}$, $w_{4 5}$, $w_{5
      5}$, $w_{3 5}$, $w_{2 5}$, $w_{1 5}$, $w_{4 6}$, $w_{5 6}$, $w_{3
      6}$, $w_{2 6}$, $w_{1 6}$, $w_{4 7}$, $w_{5 7}$, $w_{3 7}$, $w_{2
      7}$, $w_{1 7}$, $w_{4 3}$, $w_{5 3}$, $w_{3 3}$, $w_{2 3}$, $w_{1
      3}$, $w_{4 2}$, $w_{5 2}$, $w_{3 2}$, $w_{2 2}$, $w_{1 2}$, $w_{4
      3}$, $w_{5 3}$, $w_{3 3}$, $w_{2 3}$, $w_{1 3}$, $w_{4 1}$, $w_{5
      1}$, $w_{3 1}$, $w_{2 1}$, $w_{1 1}$, $w_{4 2}$, $w_{5 2}$, $w_{3
      2}$, $w_{2 2}$, $w_{1 2}$, $w_{4 3}$, $w_{5 3}$, $w_{3 3}$, $w_{2
      3}$, $w_{1 3}$, $w_{1 7}$, $w_{2 7}$, $w_{3 7}$, $w_{4 7}$, $w_{5
      7}$, $w_{1 6}$, $w_{2 6}$, $w_{3 6}$, $w_{4 6}$, $w_{5 6}$, $w_{1
      5}$, $w_{2 5}$, $w_{3 5}$, $w_{4 5}$, $w_{5 5}$, $w_{1 4}$, $w_{2
      4}$, $w_{3 4}$, $w_{4 4}$, $w_{5 4}$, $w_{1 3}$, $w_{2 3}$, $w_{3
      3}$, $w_{4 3}$, $w_{5 3}$, $w_{1 2}$, $w_{2 2}$, $w_{3 2}$, $w_{4
      2}$, $w_{5 2}$, $w_{1 1}$, $w_{2 1}$, $w_{3 1}$, $w_{4 1}$, $w_{5
      1}$, $w_{1 7}$, $w_{2 7}$, $w_{3 7}$, $w_{4 7}$, $w_{5 7}$, $w_{1
      6}$, $w_{2 6}$, $w_{3 6}$, $w_{4 6}$, $w_{5 6}$, $w_{1 5}$, $w_{2
      5}$, $w_{3 5}$, $w_{4 5}$, $w_{5 5}$, $w_{1 4}$, $w_{2 4}$, $w_{3
      4}$, $w_{4 4}$, $w_{5 4}$, $w_{1 3}$, $w_{2 3}$, $w_{3 3}$, $w_{4
      3}$, $w_{5 3}$, $w_{1 2}$, $w_{2 2}$, $w_{3 2}$, $w_{4 2}$, $w_{5
      2}$, $w_{1 7}$, $w_{2 7}$, $w_{3 7}$, $w_{4 7}$, $w_{5 7}$, $w_{1
      6}$, $w_{2 6}$, $w_{3 6}$, $w_{4 6}$, $w_{5 6}$, $w_{1 5}$, $w_{2
      5}$, $w_{3 5}$, $w_{4 5}$, $w_{5 5}$, $w_{1 4}$, $w_{2 4}$, $w_{3
      4}$, $w_{4 4}$, $w_{5 4}$, $w_{1 3}$, $w_{2 3}$, $w_{3 3}$, $w_{4
      3}$, $w_{5 3}$, $w_{1 7}$, $w_{2 7}$, $w_{3 7}$, $w_{4 7}$, $w_{5
      7}$, $w_{1 6}$, $w_{2 6}$, $w_{3 6}$, $w_{4 6}$, $w_{5 6}$, $w_{1
      5}$, $w_{2 5}$, $w_{3 5}$, $w_{4 5}$, $w_{5 5}$, $w_{1 4}$, $w_{2
      4}$, $w_{3 4}$, $w_{4 4}$, $w_{5 4}$, $w_{1 7}$, $w_{2 7}$, $w_{3
      7}$, $w_{4 7}$, $w_{5 7}$, $w_{1 6}$, $w_{2 6}$, $w_{3 6}$, $w_{4
      6}$, $w_{5 6}$, $w_{1 5}$, $w_{2 5}$, $w_{3 5}$, $w_{4 5}$, $w_{5
      5}$, $w_{1 7}$, $w_{2 7}$, $w_{3 7}$, $w_{4 7}$, $w_{5 7}$, $w_{1
      6}$, $w_{2 6}$, $w_{3 6}$, $w_{4 6}$, $w_{5 6}$, $w_{1 7}$, $w_{2
      7}$, $w_{3 7}$, $w_{4 7}$, $w_{5 7}$, $w_{4 1}$, $w_{5 1}$, $w_{4
      2}$, $w_{5 2}$, $w_{4 3}$, $w_{5 3}$, $w_{4 4}$, $w_{5 4}$, $w_{4
      5}$, $w_{5 5}$, $w_{4 6}$, $w_{5 6}$, $w_{4 7}$, $w_{5 7}$, $w_{4
      1}$, $w_{5 1}$, $w_{4 2}$, $w_{5 2}$, $w_{4 3}$, $w_{5 3}$, $w_{4
      4}$, $w_{5 4}$, $w_{4 5}$, $w_{5 5}$, $w_{4 6}$, $w_{5 6}$, $w_{4
      1}$, $w_{5 1}$, $w_{4 2}$, $w_{5 2}$, $w_{4 3}$, $w_{5 3}$, $w_{4
      4}$, $w_{5 4}$, $w_{4 5}$, $w_{5 5}$, $w_{4 1}$, $w_{5 1}$, $w_{4
      2}$, $w_{5 2}$, $w_{4 3}$, $w_{5 3}$, $w_{4 4}$, $w_{5 4}$, $w_{4
      1}$, $w_{5 1}$, $w_{4 2}$, $w_{5 2}$, $w_{4 3}$, $w_{5 3}$, $w_{4
      1}$, $w_{5 1}$, $w_{4 2}$, $w_{5 2}$, $w_{4 1}$, $w_{5 1}$, $w_{4
      7}$, $w_{5 7}$, $w_{4 6}$, $w_{5 6}$, $w_{4 5}$, $w_{5 5}$, $w_{4
      4}$, $w_{5 4}$, $w_{4 3}$, $w_{5 3}$, $w_{4 2}$, $w_{5 2}$, $w_{4
      1}$, $w_{5 1}$, $w_{3 1}$, $w_{3 2}$, $w_{3 3}$, $w_{3 4}$, $w_{3
      5}$, $w_{3 6}$, $w_{3 7}$, $w_{3 1}$, $w_{3 2}$, $w_{3 3}$, $w_{3
      4}$, $w_{3 5}$, $w_{3 6}$, $w_{3 1}$, $w_{3 2}$, $w_{3 3}$, $w_{3
      4}$, $w_{3 5}$, $w_{3 1}$, $w_{3 2}$, $w_{3 3}$, $w_{3 4}$, $w_{3
      1}$, $w_{3 2}$, $w_{3 3}$, $w_{3 1}$, $w_{3 2}$, $w_{3 1}$, $w_{3
      7}$, $w_{3 6}$, $w_{3 5}$, $w_{3 4}$, $w_{3 3}$, $w_{3 2}$, $w_{3
      1}$, $w_{2 1}$, $w_{2 2}$, $w_{2 3}$, $w_{2 4}$, $w_{2 5}$, $w_{2
      6}$, $w_{2 7}$, $w_{2 1}$, $w_{2 2}$, $w_{2 3}$, $w_{2 4}$, $w_{2
      5}$, $w_{2 6}$, $w_{2 1}$, $w_{2 2}$, $w_{2 3}$, $w_{2 4}$, $w_{2
      5}$, $w_{2 1}$, $w_{2 2}$, $w_{2 3}$, $w_{2 4}$, $w_{2 1}$, $w_{2
      2}$, $w_{2 3}$, $w_{2 1}$, $w_{2 2}$, $w_{2 1}$, $w_{2 7}$, $w_{2
      6}$, $w_{2 5}$, $w_{2 4}$, $w_{2 3}$, $w_{2 2}$, $w_{2 1}$, $w_{1
      1}$, $w_{1 2}$, $w_{1 3}$, $w_{1 4}$, $w_{1 5}$, $w_{1 6}$, $w_{1
      7}$, $w_{1 1}$, $w_{1 2}$, $w_{1 3}$, $w_{1 4}$, $w_{1 5}$, $w_{1
      6}$, $w_{1 1}$, $w_{1 2}$, $w_{1 3}$, $w_{1 4}$, $w_{1 5}$, $w_{1
      1}$, $w_{1 2}$, $w_{1 3}$, $w_{1 4}$, $w_{1 1}$, $w_{1 2}$, $w_{1
      3}$, $w_{1 1}$, $w_{1 2}$, $w_{1 1}$, $w_{1 7}$, $w_{1 6}$, $w_{1
      5}$, $w_{1 4}$, $w_{1 3}$, $w_{1 2}$, $w_{1 1}$, $w_{4 5}$, $w_{5
      5}$, $w_{4 6}$, $w_{5 6}$, $w_{4 7}$, $w_{5 7}$, $w_{4 5}$, $w_{5
      5}$, $w_{4 6}$, $w_{5 6}$, $w_{4 5}$, $w_{5 5}$, $w_{4 7}$, $w_{5
      7}$, $w_{4 6}$, $w_{5 6}$, $w_{4 5}$, $w_{5 5}$, $w_{3 5}$, $w_{3
      6}$, $w_{3 7}$, $w_{3 5}$, $w_{3 6}$, $w_{3 5}$, $w_{3 7}$, $w_{3
      6}$, $w_{3 5}$, $w_{2 5}$, $w_{2 6}$, $w_{2 7}$, $w_{2 5}$, $w_{2
      6}$, $w_{2 5}$, $w_{2 7}$, $w_{2 6}$, $w_{2 5}$, $w_{1 5}$, $w_{1
      6}$, $w_{1 7}$, $w_{1 5}$, $w_{1 6}$, $w_{1 5}$, $w_{1 7}$, $w_{1
      6}$, $w_{1 5}$, $w_{4 1}$, $w_{5 1}$, $w_{4 2}$, $w_{5 2}$, $w_{4
      3}$, $w_{5 3}$, $w_{4 1}$, $w_{5 1}$, $w_{4 2}$, $w_{5 2}$, $w_{4
      1}$, $w_{5 1}$, $w_{4 3}$, $w_{5 3}$, $w_{4 2}$, $w_{5 2}$, $w_{4
      1}$, $w_{5 1}$, $w_{3 1}$, $w_{3 2}$, $w_{3 3}$, $w_{3 1}$, $w_{3
      2}$, $w_{3 1}$, $w_{3 3}$, $w_{3 2}$, $w_{3 1}$, $w_{2 1}$, $w_{2
      2}$, $w_{2 3}$, $w_{2 1}$, $w_{2 2}$, $w_{2 1}$, $w_{2 3}$, $w_{2
      2}$, $w_{2 1}$, $w_{1 1}$, $w_{1 2}$, $w_{1 3}$, $w_{1 1}$, $w_{1
      2}$, $w_{1 1}$, $w_{1 3}$, $w_{1 2}$, $w_{1 1}$ }

  $\conftoquiv(\alpha)$ gives
  \begin{equation*}
    \begin{gathered}
      \begin{aligned}
        A_1 &= h_{11} &\qquad A_2 &= h_{12} &\qquad A_3 &= h_{13} &\qquad A_4 &= h_{14} &\qquad A_5 &= h_{15} \\
        B_1 &= h_{21} & B_2 &= h_{22} & B_3 &= h_{23} & B_4 &= h_{24} & B_5 &= h_{25} \\
        C_1 &= h_{31} & C_2 &= h_{32} & C_3 &= h_{33} & C_4 &= h_{35} & C_5 &= h_{34} \\
      \end{aligned}
      \\
      \begin{aligned}
        v_{1 1} &= \genminor{w_{6}}{1}(u) \cdot h_{21} & v_{1 2} &= \genminor{w_{11}}{1}(u) \cdot \frac{h_{12} h_{21} }{h_{11}} & v_{1 3} &= \genminor{w_{16}}{1}(u) \cdot \frac{h_{13} h_{21} }{h_{11}} \\
        v_{2 1} &= \genminor{w_{7}}{2}(u) \cdot \frac{h_{11} h_{22} }{h_{12}} & v_{2 2} &= \genminor{w_{12}}{2}(u) \cdot h_{22} & v_{2 3} &= \genminor{w_{17}}{2}(u) \cdot \frac{h_{11} h_{13} h_{22} }{h_{12}} \\
        v_{3 1} &= \genminor{w_{8}}{3}(u) \cdot \frac{h_{11} h_{23} }{h_{13}} & v_{3 2} &= \genminor{w_{13}}{3}(u) \cdot \frac{h_{11} h_{12} h_{23} }{h_{13}} & v_{3 3} &= \genminor{w_{18}}{3}(u) \cdot h_{12} h_{23} \\
        v_{4 1} &= \genminor{w_{9}}{4}(u) \cdot \frac{h_{11} h_{24} }{h_{15}} & v_{4 2} &= \genminor{w_{14}}{4}(u) \cdot \frac{h_{12} h_{24} }{h_{15}} & v_{4 3} &= \genminor{w_{19}}{4}(u) \cdot \frac{h_{13} h_{24} }{h_{15}} \\
        v_{5 1} &= \genminor{w_{10}}{5}(u) \cdot \frac{h_{11} h_{25} }{h_{14}} & v_{5 2} &= \genminor{w_{15}}{5}(u) \cdot \frac{h_{12} h_{25} }{h_{14}} & v_{5 3} &= \genminor{w_{20}}{5}(u) \cdot \frac{h_{13} h_{25} }{h_{14}} \\
      \end{aligned}
    \end{gathered}
  \end{equation*}

  \subsection{\texorpdfstring{$G_2$}{G\_2}}

  Based on the Dynkin diagram \raisebox{-0.1ex}{\includestandalone[mode=image|tex,height=1.5ex]{fig/dynkin-G2}}, (with $c = \set{1, 2}$ and
  trivial linear partitioning) we use the following quiver for Dynkin
  type.

  {\centering \includestandalone[mode=image|tex]{fig/dynkin-type-G2}

  } $Q_{G_2}$ is given in \cref{fig:QG2}.

  \begin{figure}
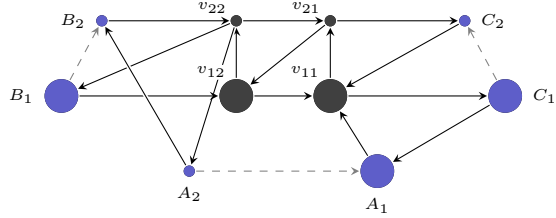

    \centering \includestandalone[mode=image|tex]{fig/QG2}

    \caption{$Q$ for $G_2$.
      \label{fig:QG2}
    }
  \end{figure}

  \inlmutseq{\murot}{ $v_{1 1}$, $v_{2 1}$, $v_{1 2}$, $v_{2 2}$, $v_{1
      1}$, $v_{2 1}$, $v_{1 1}$, $v_{1 2}$, $v_{1 1}$, $v_{1 2}$, $v_{1
      1}$, $v_{2 1}$, $v_{2 2}$, $v_{2 1}$, $v_{2 2}$, $v_{2 1}$ }

  \inlmutseq{\muflip}{$w_{2 5}$, $w_{1 5}$, $w_{2 4}$, $w_{1 4}$, $w_{2
      5}$, $w_{1 5}$, $w_{2 2}$, $w_{1 2}$, $w_{2 1}$, $w_{1 1}$, $w_{2
      2}$, $w_{1 2}$, $w_{1 5}$, $w_{2 5}$, $w_{1 4}$, $w_{2 4}$, $w_{1
      3}$, $w_{2 3}$, $w_{1 2}$, $w_{2 2}$, $w_{1 1}$, $w_{2 1}$, $w_{1
      5}$, $w_{2 5}$, $w_{1 4}$, $w_{2 4}$, $w_{1 3}$, $w_{2 3}$, $w_{1
      2}$, $w_{2 2}$, $w_{1 5}$, $w_{2 5}$, $w_{1 4}$, $w_{2 4}$, $w_{1
      3}$, $w_{2 3}$, $w_{1 5}$, $w_{2 5}$, $w_{1 4}$, $w_{2 4}$, $w_{1
      5}$, $w_{2 5}$, $w_{2 1}$, $w_{2 2}$, $w_{2 3}$, $w_{2 4}$, $w_{2
      5}$, $w_{2 1}$, $w_{2 2}$, $w_{2 3}$, $w_{2 4}$, $w_{2 1}$, $w_{2
      2}$, $w_{2 3}$, $w_{2 1}$, $w_{2 2}$, $w_{2 1}$, $w_{2 5}$, $w_{2
      4}$, $w_{2 3}$, $w_{2 2}$, $w_{2 1}$, $w_{1 1}$, $w_{1 2}$, $w_{1
      3}$, $w_{1 4}$, $w_{1 5}$, $w_{1 1}$, $w_{1 2}$, $w_{1 3}$, $w_{1
      4}$, $w_{1 1}$, $w_{1 2}$, $w_{1 3}$, $w_{1 1}$, $w_{1 2}$, $w_{1
      1}$, $w_{1 5}$, $w_{1 4}$, $w_{1 3}$, $w_{1 2}$, $w_{1 1}$, $w_{2
      4}$, $w_{2 5}$, $w_{2 4}$, $w_{2 5}$, $w_{2 4}$, $w_{1 4}$, $w_{1
      5}$, $w_{1 4}$, $w_{1 5}$, $w_{1 4}$, $w_{2 1}$, $w_{2 2}$, $w_{2
      1}$, $w_{2 2}$, $w_{2 1}$, $w_{1 1}$, $w_{1 2}$, $w_{1 1}$, $w_{1
      2}$, $w_{1 1}$}

  $\conftoquiv(\alpha)$ gives
  \begin{equation*}
    \begin{gathered}
      \begin{aligned}
        A_1 &= h_{11} &\qquad A_2 &= h_{12} \\
        B_1 &= h_{21} & B_2 &= h_{22} \\
        C_1 &= h_{31} & C_2 &= h_{32} \\
      \end{aligned}
      \\
      \begin{aligned}
        v_{1 1} &= \genminor{w_{3}}{1}(u) \cdot h_{21} & v_{2 1} &= \genminor{w_{4}}{2}(u) \cdot \frac{h_{11}^3 h_{22} }{h_{12}} \\
        v_{1 2} &= \genminor{w_{5}}{1}(u) \cdot h_{11} h_{21} & v_{2 2} &= \genminor{w_{6}}{2}(u) \cdot \frac{h_{11}^3 h_{22} }{h_{12}} \\
      \end{aligned}
    \end{gathered}
  \end{equation*}

  \subsection{\texorpdfstring{$E_6$}{E\_6}}

  Based on the Dynkin diagram \raisebox{-0.7ex}{\includestandalone[mode=image|tex,height=3ex]{fig/dynkin-E6}}, the following quiver is of
  well-rooted Dynkin type for $E_6$. It admits an induced Coxeter
  element $c = \set{2,4,3,5,1,6}$, with partitions $T_0 = \set{2}, T_1 =
  \set{4}, T_3 = \set{3,5}, T_4 = \set{1,6}$.

  {\centering \includestandalone[mode=image|tex]{fig/dynkin-type-E6}

  } $Q_{E_6}$ is given in \cref{fig:QE6}.

  \begin{figure}
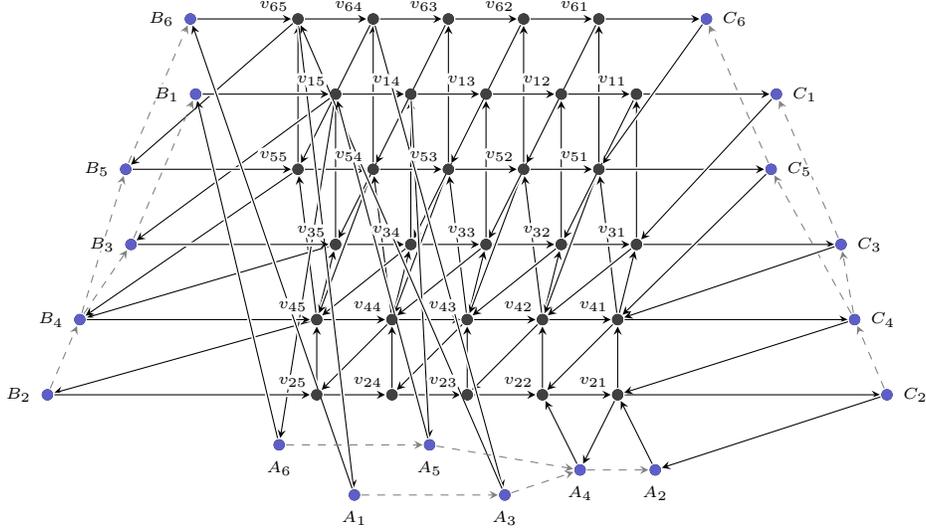

    \centering \includestandalone[mode=image|tex]{fig/QE6}

    \caption{$Q$ for $E_6$.
      \label{fig:QE6}
    }
  \end{figure}

  \inlmutseq{\murot}{ $v_{2 1}$, $v_{4 1}$, $v_{3 1}$, $v_{5 1}$, $v_{1
      1}$, $v_{6 1}$, $v_{2 2}$, $v_{4 2}$, $v_{3 2}$, $v_{5 2}$, $v_{1
      2}$, $v_{6 2}$, $v_{2 1}$, $v_{4 1}$, $v_{3 1}$, $v_{5 1}$, $v_{1
      1}$, $v_{6 1}$, $v_{2 3}$, $v_{4 3}$, $v_{3 3}$, $v_{5 3}$, $v_{1
      3}$, $v_{6 3}$, $v_{2 2}$, $v_{4 2}$, $v_{3 2}$, $v_{5 2}$, $v_{1
      2}$, $v_{6 2}$, $v_{2 1}$, $v_{4 1}$, $v_{3 1}$, $v_{5 1}$, $v_{1
      1}$, $v_{6 1}$, $v_{2 4}$, $v_{4 4}$, $v_{3 4}$, $v_{5 4}$, $v_{1
      4}$, $v_{6 4}$, $v_{2 3}$, $v_{4 3}$, $v_{3 3}$, $v_{5 3}$, $v_{1
      3}$, $v_{6 3}$, $v_{2 2}$, $v_{4 2}$, $v_{3 2}$, $v_{5 2}$, $v_{1
      2}$, $v_{6 2}$, $v_{2 1}$, $v_{4 1}$, $v_{3 1}$, $v_{5 1}$, $v_{1
      1}$, $v_{6 1}$, $v_{2 5}$, $v_{4 5}$, $v_{3 5}$, $v_{5 5}$, $v_{1
      5}$, $v_{6 5}$, $v_{2 4}$, $v_{4 4}$, $v_{3 4}$, $v_{5 4}$, $v_{1
      4}$, $v_{6 4}$, $v_{2 3}$, $v_{4 3}$, $v_{3 3}$, $v_{5 3}$, $v_{1
      3}$, $v_{6 3}$, $v_{2 2}$, $v_{4 2}$, $v_{3 2}$, $v_{5 2}$, $v_{1
      2}$, $v_{6 2}$, $v_{2 1}$, $v_{4 1}$, $v_{3 1}$, $v_{5 1}$, $v_{1
      1}$, $v_{6 1}$, $v_{2 1}$, $v_{4 1}$, $v_{3 1}$, $v_{5 1}$, $v_{1
      1}$, $v_{6 1}$, $v_{2 1}$, $v_{4 1}$, $v_{3 1}$, $v_{5 1}$, $v_{1
      1}$, $v_{6 1}$, $v_{2 1}$, $v_{4 1}$, $v_{3 1}$, $v_{5 1}$, $v_{1
      1}$, $v_{6 1}$, $v_{2 1}$, $v_{4 1}$, $v_{3 1}$, $v_{5 1}$, $v_{1
      1}$, $v_{6 1}$, $v_{2 1}$, $v_{4 1}$, $v_{3 1}$, $v_{5 1}$, $v_{1
      1}$, $v_{6 1}$, $v_{2 1}$, $v_{4 1}$, $v_{3 1}$, $v_{5 1}$, $v_{1
      1}$, $v_{6 1}$, $v_{2 1}$, $v_{4 1}$, $v_{3 1}$, $v_{5 1}$, $v_{1
      1}$, $v_{6 1}$, $v_{2 2}$, $v_{4 2}$, $v_{3 2}$, $v_{5 2}$, $v_{1
      2}$, $v_{6 2}$, $v_{2 2}$, $v_{4 2}$, $v_{3 2}$, $v_{5 2}$, $v_{1
      2}$, $v_{6 2}$, $v_{2 2}$, $v_{4 2}$, $v_{3 2}$, $v_{5 2}$, $v_{1
      2}$, $v_{6 2}$, $v_{2 2}$, $v_{4 2}$, $v_{3 2}$, $v_{5 2}$, $v_{1
      2}$, $v_{6 2}$, $v_{2 2}$, $v_{4 2}$, $v_{3 2}$, $v_{5 2}$, $v_{1
      2}$, $v_{6 2}$, $v_{2 2}$, $v_{4 2}$, $v_{3 2}$, $v_{5 2}$, $v_{1
      2}$, $v_{6 2}$, $v_{2 2}$, $v_{4 2}$, $v_{3 2}$, $v_{5 2}$, $v_{1
      2}$, $v_{6 2}$, $v_{2 3}$, $v_{4 3}$, $v_{3 3}$, $v_{5 3}$, $v_{1
      3}$, $v_{6 3}$, $v_{2 3}$, $v_{4 3}$, $v_{3 3}$, $v_{5 3}$, $v_{1
      3}$, $v_{6 3}$, $v_{2 3}$, $v_{4 3}$, $v_{3 3}$, $v_{5 3}$, $v_{1
      3}$, $v_{6 3}$, $v_{2 3}$, $v_{4 3}$, $v_{3 3}$, $v_{5 3}$, $v_{1
      3}$, $v_{6 3}$, $v_{2 3}$, $v_{4 3}$, $v_{3 3}$, $v_{5 3}$, $v_{1
      3}$, $v_{6 3}$, $v_{2 3}$, $v_{4 3}$, $v_{3 3}$, $v_{5 3}$, $v_{1
      3}$, $v_{6 3}$, $v_{2 3}$, $v_{4 3}$, $v_{3 3}$, $v_{5 3}$, $v_{1
      3}$, $v_{6 3}$, $v_{2 4}$, $v_{4 4}$, $v_{3 4}$, $v_{5 4}$, $v_{1
      4}$, $v_{6 4}$, $v_{2 4}$, $v_{4 4}$, $v_{3 4}$, $v_{5 4}$, $v_{1
      4}$, $v_{6 4}$, $v_{2 4}$, $v_{4 4}$, $v_{3 4}$, $v_{5 4}$, $v_{1
      4}$, $v_{6 4}$, $v_{2 4}$, $v_{4 4}$, $v_{3 4}$, $v_{5 4}$, $v_{1
      4}$, $v_{6 4}$, $v_{2 4}$, $v_{4 4}$, $v_{3 4}$, $v_{5 4}$, $v_{1
      4}$, $v_{6 4}$, $v_{2 4}$, $v_{4 4}$, $v_{3 4}$, $v_{5 4}$, $v_{1
      4}$, $v_{6 4}$, $v_{2 4}$, $v_{4 4}$, $v_{3 4}$, $v_{5 4}$, $v_{1
      4}$, $v_{6 4}$, $v_{2 5}$, $v_{4 5}$, $v_{3 5}$, $v_{5 5}$, $v_{1
      5}$, $v_{6 5}$, $v_{2 5}$, $v_{4 5}$, $v_{3 5}$, $v_{5 5}$, $v_{1
      5}$, $v_{6 5}$, $v_{2 5}$, $v_{4 5}$, $v_{3 5}$, $v_{5 5}$, $v_{1
      5}$, $v_{6 5}$, $v_{2 5}$, $v_{4 5}$, $v_{3 5}$, $v_{5 5}$, $v_{1
      5}$, $v_{6 5}$, $v_{2 5}$, $v_{4 5}$, $v_{3 5}$, $v_{5 5}$, $v_{1
      5}$, $v_{6 5}$, $v_{2 5}$, $v_{4 5}$, $v_{3 5}$, $v_{5 5}$, $v_{1
      5}$, $v_{6 5}$, $v_{2 5}$, $v_{4 5}$, $v_{3 5}$, $v_{5 5}$, $v_{1
      5}$, $v_{6 5}$, $v_{2 1}$, $v_{2 2}$, $v_{2 3}$, $v_{2 4}$, $v_{2
      5}$, $v_{2 1}$, $v_{2 2}$, $v_{2 3}$, $v_{2 4}$, $v_{2 1}$, $v_{2
      2}$, $v_{2 3}$, $v_{2 1}$, $v_{2 2}$, $v_{2 1}$, $v_{2 5}$, $v_{2
      4}$, $v_{2 3}$, $v_{2 2}$, $v_{2 1}$, $v_{4 1}$, $v_{4 2}$, $v_{4
      3}$, $v_{4 4}$, $v_{4 5}$, $v_{4 1}$, $v_{4 2}$, $v_{4 3}$, $v_{4
      4}$, $v_{4 1}$, $v_{4 2}$, $v_{4 3}$, $v_{4 1}$, $v_{4 2}$, $v_{4
      1}$, $v_{4 5}$, $v_{4 4}$, $v_{4 3}$, $v_{4 2}$, $v_{4 1}$, $v_{3
      1}$, $v_{5 1}$, $v_{3 2}$, $v_{5 2}$, $v_{3 3}$, $v_{5 3}$, $v_{3
      4}$, $v_{5 4}$, $v_{3 5}$, $v_{5 5}$, $v_{3 1}$, $v_{5 1}$, $v_{3
      2}$, $v_{5 2}$, $v_{3 3}$, $v_{5 3}$, $v_{3 4}$, $v_{5 4}$, $v_{3
      1}$, $v_{5 1}$, $v_{3 2}$, $v_{5 2}$, $v_{3 3}$, $v_{5 3}$, $v_{3
      1}$, $v_{5 1}$, $v_{3 2}$, $v_{5 2}$, $v_{3 1}$, $v_{5 1}$, $v_{3
      5}$, $v_{5 5}$, $v_{3 4}$, $v_{5 4}$, $v_{3 3}$, $v_{5 3}$, $v_{3
      2}$, $v_{5 2}$, $v_{3 1}$, $v_{5 1}$, $v_{1 1}$, $v_{6 1}$, $v_{1
      2}$, $v_{6 2}$, $v_{1 3}$, $v_{6 3}$, $v_{1 4}$, $v_{6 4}$, $v_{1
      5}$, $v_{6 5}$, $v_{1 1}$, $v_{6 1}$, $v_{1 2}$, $v_{6 2}$, $v_{1
      3}$, $v_{6 3}$, $v_{1 4}$, $v_{6 4}$, $v_{1 1}$, $v_{6 1}$, $v_{1
      2}$, $v_{6 2}$, $v_{1 3}$, $v_{6 3}$, $v_{1 1}$, $v_{6 1}$, $v_{1
      2}$, $v_{6 2}$, $v_{1 1}$, $v_{6 1}$, $v_{1 5}$, $v_{6 5}$, $v_{1
      4}$, $v_{6 4}$, $v_{1 3}$, $v_{6 3}$, $v_{1 2}$, $v_{6 2}$, $v_{1
      1}$, $v_{6 1}$ }

  \inlmutseq{\muflip}{$w_{1 b}$, $w_{6 b}$, $w_{3 b}$, $w_{5 b}$, $w_{4
      b}$, $w_{2 b}$, $w_{1 a}$, $w_{6 a}$, $w_{3 a}$, $w_{5 a}$, $w_{4
      a}$, $w_{2 a}$, $w_{1 b}$, $w_{6 b}$, $w_{3 b}$, $w_{5 b}$, $w_{4
      b}$, $w_{2 b}$, $w_{1 9}$, $w_{6 9}$, $w_{3 9}$, $w_{5 9}$, $w_{4
      9}$, $w_{2 9}$, $w_{1 a}$, $w_{6 a}$, $w_{3 a}$, $w_{5 a}$, $w_{4
      a}$, $w_{2 a}$, $w_{1 b}$, $w_{6 b}$, $w_{3 b}$, $w_{5 b}$, $w_{4
      b}$, $w_{2 b}$, $w_{1 8}$, $w_{6 8}$, $w_{3 8}$, $w_{5 8}$, $w_{4
      8}$, $w_{2 8}$, $w_{1 9}$, $w_{6 9}$, $w_{3 9}$, $w_{5 9}$, $w_{4
      9}$, $w_{2 9}$, $w_{1 a}$, $w_{6 a}$, $w_{3 a}$, $w_{5 a}$, $w_{4
      a}$, $w_{2 a}$, $w_{1 b}$, $w_{6 b}$, $w_{3 b}$, $w_{5 b}$, $w_{4
      b}$, $w_{2 b}$, $w_{1 7}$, $w_{6 7}$, $w_{3 7}$, $w_{5 7}$, $w_{4
      7}$, $w_{2 7}$, $w_{1 8}$, $w_{6 8}$, $w_{3 8}$, $w_{5 8}$, $w_{4
      8}$, $w_{2 8}$, $w_{1 9}$, $w_{6 9}$, $w_{3 9}$, $w_{5 9}$, $w_{4
      9}$, $w_{2 9}$, $w_{1 a}$, $w_{6 a}$, $w_{3 a}$, $w_{5 a}$, $w_{4
      a}$, $w_{2 a}$, $w_{1 b}$, $w_{6 b}$, $w_{3 b}$, $w_{5 b}$, $w_{4
      b}$, $w_{2 b}$, $w_{1 5}$, $w_{6 5}$, $w_{3 5}$, $w_{5 5}$, $w_{4
      5}$, $w_{2 5}$, $w_{1 4}$, $w_{6 4}$, $w_{3 4}$, $w_{5 4}$, $w_{4
      4}$, $w_{2 4}$, $w_{1 5}$, $w_{6 5}$, $w_{3 5}$, $w_{5 5}$, $w_{4
      5}$, $w_{2 5}$, $w_{1 3}$, $w_{6 3}$, $w_{3 3}$, $w_{5 3}$, $w_{4
      3}$, $w_{2 3}$, $w_{1 4}$, $w_{6 4}$, $w_{3 4}$, $w_{5 4}$, $w_{4
      4}$, $w_{2 4}$, $w_{1 5}$, $w_{6 5}$, $w_{3 5}$, $w_{5 5}$, $w_{4
      5}$, $w_{2 5}$, $w_{1 2}$, $w_{6 2}$, $w_{3 2}$, $w_{5 2}$, $w_{4
      2}$, $w_{2 2}$, $w_{1 3}$, $w_{6 3}$, $w_{3 3}$, $w_{5 3}$, $w_{4
      3}$, $w_{2 3}$, $w_{1 4}$, $w_{6 4}$, $w_{3 4}$, $w_{5 4}$, $w_{4
      4}$, $w_{2 4}$, $w_{1 5}$, $w_{6 5}$, $w_{3 5}$, $w_{5 5}$, $w_{4
      5}$, $w_{2 5}$, $w_{1 1}$, $w_{6 1}$, $w_{3 1}$, $w_{5 1}$, $w_{4
      1}$, $w_{2 1}$, $w_{1 2}$, $w_{6 2}$, $w_{3 2}$, $w_{5 2}$, $w_{4
      2}$, $w_{2 2}$, $w_{1 3}$, $w_{6 3}$, $w_{3 3}$, $w_{5 3}$, $w_{4
      3}$, $w_{2 3}$, $w_{1 4}$, $w_{6 4}$, $w_{3 4}$, $w_{5 4}$, $w_{4
      4}$, $w_{2 4}$, $w_{1 5}$, $w_{6 5}$, $w_{3 5}$, $w_{5 5}$, $w_{4
      5}$, $w_{2 5}$, $w_{2 b}$, $w_{4 b}$, $w_{3 b}$, $w_{5 b}$, $w_{1
      b}$, $w_{6 b}$, $w_{2 a}$, $w_{4 a}$, $w_{3 a}$, $w_{5 a}$, $w_{1
      a}$, $w_{6 a}$, $w_{2 9}$, $w_{4 9}$, $w_{3 9}$, $w_{5 9}$, $w_{1
      9}$, $w_{6 9}$, $w_{2 8}$, $w_{4 8}$, $w_{3 8}$, $w_{5 8}$, $w_{1
      8}$, $w_{6 8}$, $w_{2 7}$, $w_{4 7}$, $w_{3 7}$, $w_{5 7}$, $w_{1
      7}$, $w_{6 7}$, $w_{2 6}$, $w_{4 6}$, $w_{3 6}$, $w_{5 6}$, $w_{1
      6}$, $w_{6 6}$, $w_{2 5}$, $w_{4 5}$, $w_{3 5}$, $w_{5 5}$, $w_{1
      5}$, $w_{6 5}$, $w_{2 4}$, $w_{4 4}$, $w_{3 4}$, $w_{5 4}$, $w_{1
      4}$, $w_{6 4}$, $w_{2 3}$, $w_{4 3}$, $w_{3 3}$, $w_{5 3}$, $w_{1
      3}$, $w_{6 3}$, $w_{2 2}$, $w_{4 2}$, $w_{3 2}$, $w_{5 2}$, $w_{1
      2}$, $w_{6 2}$, $w_{2 1}$, $w_{4 1}$, $w_{3 1}$, $w_{5 1}$, $w_{1
      1}$, $w_{6 1}$, $w_{2 b}$, $w_{4 b}$, $w_{3 b}$, $w_{5 b}$, $w_{1
      b}$, $w_{6 b}$, $w_{2 a}$, $w_{4 a}$, $w_{3 a}$, $w_{5 a}$, $w_{1
      a}$, $w_{6 a}$, $w_{2 9}$, $w_{4 9}$, $w_{3 9}$, $w_{5 9}$, $w_{1
      9}$, $w_{6 9}$, $w_{2 8}$, $w_{4 8}$, $w_{3 8}$, $w_{5 8}$, $w_{1
      8}$, $w_{6 8}$, $w_{2 7}$, $w_{4 7}$, $w_{3 7}$, $w_{5 7}$, $w_{1
      7}$, $w_{6 7}$, $w_{2 6}$, $w_{4 6}$, $w_{3 6}$, $w_{5 6}$, $w_{1
      6}$, $w_{6 6}$, $w_{2 5}$, $w_{4 5}$, $w_{3 5}$, $w_{5 5}$, $w_{1
      5}$, $w_{6 5}$, $w_{2 4}$, $w_{4 4}$, $w_{3 4}$, $w_{5 4}$, $w_{1
      4}$, $w_{6 4}$, $w_{2 3}$, $w_{4 3}$, $w_{3 3}$, $w_{5 3}$, $w_{1
      3}$, $w_{6 3}$, $w_{2 2}$, $w_{4 2}$, $w_{3 2}$, $w_{5 2}$, $w_{1
      2}$, $w_{6 2}$, $w_{2 b}$, $w_{4 b}$, $w_{3 b}$, $w_{5 b}$, $w_{1
      b}$, $w_{6 b}$, $w_{2 a}$, $w_{4 a}$, $w_{3 a}$, $w_{5 a}$, $w_{1
      a}$, $w_{6 a}$, $w_{2 9}$, $w_{4 9}$, $w_{3 9}$, $w_{5 9}$, $w_{1
      9}$, $w_{6 9}$, $w_{2 8}$, $w_{4 8}$, $w_{3 8}$, $w_{5 8}$, $w_{1
      8}$, $w_{6 8}$, $w_{2 7}$, $w_{4 7}$, $w_{3 7}$, $w_{5 7}$, $w_{1
      7}$, $w_{6 7}$, $w_{2 6}$, $w_{4 6}$, $w_{3 6}$, $w_{5 6}$, $w_{1
      6}$, $w_{6 6}$, $w_{2 5}$, $w_{4 5}$, $w_{3 5}$, $w_{5 5}$, $w_{1
      5}$, $w_{6 5}$, $w_{2 4}$, $w_{4 4}$, $w_{3 4}$, $w_{5 4}$, $w_{1
      4}$, $w_{6 4}$, $w_{2 3}$, $w_{4 3}$, $w_{3 3}$, $w_{5 3}$, $w_{1
      3}$, $w_{6 3}$, $w_{2 b}$, $w_{4 b}$, $w_{3 b}$, $w_{5 b}$, $w_{1
      b}$, $w_{6 b}$, $w_{2 a}$, $w_{4 a}$, $w_{3 a}$, $w_{5 a}$, $w_{1
      a}$, $w_{6 a}$, $w_{2 9}$, $w_{4 9}$, $w_{3 9}$, $w_{5 9}$, $w_{1
      9}$, $w_{6 9}$, $w_{2 8}$, $w_{4 8}$, $w_{3 8}$, $w_{5 8}$, $w_{1
      8}$, $w_{6 8}$, $w_{2 7}$, $w_{4 7}$, $w_{3 7}$, $w_{5 7}$, $w_{1
      7}$, $w_{6 7}$, $w_{2 6}$, $w_{4 6}$, $w_{3 6}$, $w_{5 6}$, $w_{1
      6}$, $w_{6 6}$, $w_{2 5}$, $w_{4 5}$, $w_{3 5}$, $w_{5 5}$, $w_{1
      5}$, $w_{6 5}$, $w_{2 4}$, $w_{4 4}$, $w_{3 4}$, $w_{5 4}$, $w_{1
      4}$, $w_{6 4}$, $w_{2 b}$, $w_{4 b}$, $w_{3 b}$, $w_{5 b}$, $w_{1
      b}$, $w_{6 b}$, $w_{2 a}$, $w_{4 a}$, $w_{3 a}$, $w_{5 a}$, $w_{1
      a}$, $w_{6 a}$, $w_{2 9}$, $w_{4 9}$, $w_{3 9}$, $w_{5 9}$, $w_{1
      9}$, $w_{6 9}$, $w_{2 8}$, $w_{4 8}$, $w_{3 8}$, $w_{5 8}$, $w_{1
      8}$, $w_{6 8}$, $w_{2 7}$, $w_{4 7}$, $w_{3 7}$, $w_{5 7}$, $w_{1
      7}$, $w_{6 7}$, $w_{2 6}$, $w_{4 6}$, $w_{3 6}$, $w_{5 6}$, $w_{1
      6}$, $w_{6 6}$, $w_{2 5}$, $w_{4 5}$, $w_{3 5}$, $w_{5 5}$, $w_{1
      5}$, $w_{6 5}$, $w_{2 b}$, $w_{4 b}$, $w_{3 b}$, $w_{5 b}$, $w_{1
      b}$, $w_{6 b}$, $w_{2 a}$, $w_{4 a}$, $w_{3 a}$, $w_{5 a}$, $w_{1
      a}$, $w_{6 a}$, $w_{2 9}$, $w_{4 9}$, $w_{3 9}$, $w_{5 9}$, $w_{1
      9}$, $w_{6 9}$, $w_{2 8}$, $w_{4 8}$, $w_{3 8}$, $w_{5 8}$, $w_{1
      8}$, $w_{6 8}$, $w_{2 7}$, $w_{4 7}$, $w_{3 7}$, $w_{5 7}$, $w_{1
      7}$, $w_{6 7}$, $w_{2 6}$, $w_{4 6}$, $w_{3 6}$, $w_{5 6}$, $w_{1
      6}$, $w_{6 6}$, $w_{2 b}$, $w_{4 b}$, $w_{3 b}$, $w_{5 b}$, $w_{1
      b}$, $w_{6 b}$, $w_{2 a}$, $w_{4 a}$, $w_{3 a}$, $w_{5 a}$, $w_{1
      a}$, $w_{6 a}$, $w_{2 9}$, $w_{4 9}$, $w_{3 9}$, $w_{5 9}$, $w_{1
      9}$, $w_{6 9}$, $w_{2 8}$, $w_{4 8}$, $w_{3 8}$, $w_{5 8}$, $w_{1
      8}$, $w_{6 8}$, $w_{2 7}$, $w_{4 7}$, $w_{3 7}$, $w_{5 7}$, $w_{1
      7}$, $w_{6 7}$, $w_{2 b}$, $w_{4 b}$, $w_{3 b}$, $w_{5 b}$, $w_{1
      b}$, $w_{6 b}$, $w_{2 a}$, $w_{4 a}$, $w_{3 a}$, $w_{5 a}$, $w_{1
      a}$, $w_{6 a}$, $w_{2 9}$, $w_{4 9}$, $w_{3 9}$, $w_{5 9}$, $w_{1
      9}$, $w_{6 9}$, $w_{2 8}$, $w_{4 8}$, $w_{3 8}$, $w_{5 8}$, $w_{1
      8}$, $w_{6 8}$, $w_{2 b}$, $w_{4 b}$, $w_{3 b}$, $w_{5 b}$, $w_{1
      b}$, $w_{6 b}$, $w_{2 a}$, $w_{4 a}$, $w_{3 a}$, $w_{5 a}$, $w_{1
      a}$, $w_{6 a}$, $w_{2 9}$, $w_{4 9}$, $w_{3 9}$, $w_{5 9}$, $w_{1
      9}$, $w_{6 9}$, $w_{2 b}$, $w_{4 b}$, $w_{3 b}$, $w_{5 b}$, $w_{1
      b}$, $w_{6 b}$, $w_{2 a}$, $w_{4 a}$, $w_{3 a}$, $w_{5 a}$, $w_{1
      a}$, $w_{6 a}$, $w_{2 b}$, $w_{4 b}$, $w_{3 b}$, $w_{5 b}$, $w_{1
      b}$, $w_{6 b}$, $w_{1 1}$, $w_{6 1}$, $w_{1 2}$, $w_{6 2}$, $w_{1
      3}$, $w_{6 3}$, $w_{1 4}$, $w_{6 4}$, $w_{1 5}$, $w_{6 5}$, $w_{1
      6}$, $w_{6 6}$, $w_{1 7}$, $w_{6 7}$, $w_{1 8}$, $w_{6 8}$, $w_{1
      9}$, $w_{6 9}$, $w_{1 a}$, $w_{6 a}$, $w_{1 b}$, $w_{6 b}$, $w_{1
      1}$, $w_{6 1}$, $w_{1 2}$, $w_{6 2}$, $w_{1 3}$, $w_{6 3}$, $w_{1
      4}$, $w_{6 4}$, $w_{1 5}$, $w_{6 5}$, $w_{1 6}$, $w_{6 6}$, $w_{1
      7}$, $w_{6 7}$, $w_{1 8}$, $w_{6 8}$, $w_{1 9}$, $w_{6 9}$, $w_{1
      a}$, $w_{6 a}$, $w_{1 1}$, $w_{6 1}$, $w_{1 2}$, $w_{6 2}$, $w_{1
      3}$, $w_{6 3}$, $w_{1 4}$, $w_{6 4}$, $w_{1 5}$, $w_{6 5}$, $w_{1
      6}$, $w_{6 6}$, $w_{1 7}$, $w_{6 7}$, $w_{1 8}$, $w_{6 8}$, $w_{1
      9}$, $w_{6 9}$, $w_{1 1}$, $w_{6 1}$, $w_{1 2}$, $w_{6 2}$, $w_{1
      3}$, $w_{6 3}$, $w_{1 4}$, $w_{6 4}$, $w_{1 5}$, $w_{6 5}$, $w_{1
      6}$, $w_{6 6}$, $w_{1 7}$, $w_{6 7}$, $w_{1 8}$, $w_{6 8}$, $w_{1
      1}$, $w_{6 1}$, $w_{1 2}$, $w_{6 2}$, $w_{1 3}$, $w_{6 3}$, $w_{1
      4}$, $w_{6 4}$, $w_{1 5}$, $w_{6 5}$, $w_{1 6}$, $w_{6 6}$, $w_{1
      7}$, $w_{6 7}$, $w_{1 1}$, $w_{6 1}$, $w_{1 2}$, $w_{6 2}$, $w_{1
      3}$, $w_{6 3}$, $w_{1 4}$, $w_{6 4}$, $w_{1 5}$, $w_{6 5}$, $w_{1
      6}$, $w_{6 6}$, $w_{1 1}$, $w_{6 1}$, $w_{1 2}$, $w_{6 2}$, $w_{1
      3}$, $w_{6 3}$, $w_{1 4}$, $w_{6 4}$, $w_{1 5}$, $w_{6 5}$, $w_{1
      1}$, $w_{6 1}$, $w_{1 2}$, $w_{6 2}$, $w_{1 3}$, $w_{6 3}$, $w_{1
      4}$, $w_{6 4}$, $w_{1 1}$, $w_{6 1}$, $w_{1 2}$, $w_{6 2}$, $w_{1
      3}$, $w_{6 3}$, $w_{1 1}$, $w_{6 1}$, $w_{1 2}$, $w_{6 2}$, $w_{1
      1}$, $w_{6 1}$, $w_{1 b}$, $w_{6 b}$, $w_{1 a}$, $w_{6 a}$, $w_{1
      9}$, $w_{6 9}$, $w_{1 8}$, $w_{6 8}$, $w_{1 7}$, $w_{6 7}$, $w_{1
      6}$, $w_{6 6}$, $w_{1 5}$, $w_{6 5}$, $w_{1 4}$, $w_{6 4}$, $w_{1
      3}$, $w_{6 3}$, $w_{1 2}$, $w_{6 2}$, $w_{1 1}$, $w_{6 1}$, $w_{3
      1}$, $w_{5 1}$, $w_{3 2}$, $w_{5 2}$, $w_{3 3}$, $w_{5 3}$, $w_{3
      4}$, $w_{5 4}$, $w_{3 5}$, $w_{5 5}$, $w_{3 6}$, $w_{5 6}$, $w_{3
      7}$, $w_{5 7}$, $w_{3 8}$, $w_{5 8}$, $w_{3 9}$, $w_{5 9}$, $w_{3
      a}$, $w_{5 a}$, $w_{3 b}$, $w_{5 b}$, $w_{3 1}$, $w_{5 1}$, $w_{3
      2}$, $w_{5 2}$, $w_{3 3}$, $w_{5 3}$, $w_{3 4}$, $w_{5 4}$, $w_{3
      5}$, $w_{5 5}$, $w_{3 6}$, $w_{5 6}$, $w_{3 7}$, $w_{5 7}$, $w_{3
      8}$, $w_{5 8}$, $w_{3 9}$, $w_{5 9}$, $w_{3 a}$, $w_{5 a}$, $w_{3
      1}$, $w_{5 1}$, $w_{3 2}$, $w_{5 2}$, $w_{3 3}$, $w_{5 3}$, $w_{3
      4}$, $w_{5 4}$, $w_{3 5}$, $w_{5 5}$, $w_{3 6}$, $w_{5 6}$, $w_{3
      7}$, $w_{5 7}$, $w_{3 8}$, $w_{5 8}$, $w_{3 9}$, $w_{5 9}$, $w_{3
      1}$, $w_{5 1}$, $w_{3 2}$, $w_{5 2}$, $w_{3 3}$, $w_{5 3}$, $w_{3
      4}$, $w_{5 4}$, $w_{3 5}$, $w_{5 5}$, $w_{3 6}$, $w_{5 6}$, $w_{3
      7}$, $w_{5 7}$, $w_{3 8}$, $w_{5 8}$, $w_{3 1}$, $w_{5 1}$, $w_{3
      2}$, $w_{5 2}$, $w_{3 3}$, $w_{5 3}$, $w_{3 4}$, $w_{5 4}$, $w_{3
      5}$, $w_{5 5}$, $w_{3 6}$, $w_{5 6}$, $w_{3 7}$, $w_{5 7}$, $w_{3
      1}$, $w_{5 1}$, $w_{3 2}$, $w_{5 2}$, $w_{3 3}$, $w_{5 3}$, $w_{3
      4}$, $w_{5 4}$, $w_{3 5}$, $w_{5 5}$, $w_{3 6}$, $w_{5 6}$, $w_{3
      1}$, $w_{5 1}$, $w_{3 2}$, $w_{5 2}$, $w_{3 3}$, $w_{5 3}$, $w_{3
      4}$, $w_{5 4}$, $w_{3 5}$, $w_{5 5}$, $w_{3 1}$, $w_{5 1}$, $w_{3
      2}$, $w_{5 2}$, $w_{3 3}$, $w_{5 3}$, $w_{3 4}$, $w_{5 4}$, $w_{3
      1}$, $w_{5 1}$, $w_{3 2}$, $w_{5 2}$, $w_{3 3}$, $w_{5 3}$, $w_{3
      1}$, $w_{5 1}$, $w_{3 2}$, $w_{5 2}$, $w_{3 1}$, $w_{5 1}$, $w_{3
      b}$, $w_{5 b}$, $w_{3 a}$, $w_{5 a}$, $w_{3 9}$, $w_{5 9}$, $w_{3
      8}$, $w_{5 8}$, $w_{3 7}$, $w_{5 7}$, $w_{3 6}$, $w_{5 6}$, $w_{3
      5}$, $w_{5 5}$, $w_{3 4}$, $w_{5 4}$, $w_{3 3}$, $w_{5 3}$, $w_{3
      2}$, $w_{5 2}$, $w_{3 1}$, $w_{5 1}$, $w_{4 1}$, $w_{4 2}$, $w_{4
      3}$, $w_{4 4}$, $w_{4 5}$, $w_{4 6}$, $w_{4 7}$, $w_{4 8}$, $w_{4
      9}$, $w_{4 a}$, $w_{4 b}$, $w_{4 1}$, $w_{4 2}$, $w_{4 3}$, $w_{4
      4}$, $w_{4 5}$, $w_{4 6}$, $w_{4 7}$, $w_{4 8}$, $w_{4 9}$, $w_{4
      a}$, $w_{4 1}$, $w_{4 2}$, $w_{4 3}$, $w_{4 4}$, $w_{4 5}$, $w_{4
      6}$, $w_{4 7}$, $w_{4 8}$, $w_{4 9}$, $w_{4 1}$, $w_{4 2}$, $w_{4
      3}$, $w_{4 4}$, $w_{4 5}$, $w_{4 6}$, $w_{4 7}$, $w_{4 8}$, $w_{4
      1}$, $w_{4 2}$, $w_{4 3}$, $w_{4 4}$, $w_{4 5}$, $w_{4 6}$, $w_{4
      7}$, $w_{4 1}$, $w_{4 2}$, $w_{4 3}$, $w_{4 4}$, $w_{4 5}$, $w_{4
      6}$, $w_{4 1}$, $w_{4 2}$, $w_{4 3}$, $w_{4 4}$, $w_{4 5}$, $w_{4
      1}$, $w_{4 2}$, $w_{4 3}$, $w_{4 4}$, $w_{4 1}$, $w_{4 2}$, $w_{4
      3}$, $w_{4 1}$, $w_{4 2}$, $w_{4 1}$, $w_{4 b}$, $w_{4 a}$, $w_{4
      9}$, $w_{4 8}$, $w_{4 7}$, $w_{4 6}$, $w_{4 5}$, $w_{4 4}$, $w_{4
      3}$, $w_{4 2}$, $w_{4 1}$, $w_{2 1}$, $w_{2 2}$, $w_{2 3}$, $w_{2
      4}$, $w_{2 5}$, $w_{2 6}$, $w_{2 7}$, $w_{2 8}$, $w_{2 9}$, $w_{2
      a}$, $w_{2 b}$, $w_{2 1}$, $w_{2 2}$, $w_{2 3}$, $w_{2 4}$, $w_{2
      5}$, $w_{2 6}$, $w_{2 7}$, $w_{2 8}$, $w_{2 9}$, $w_{2 a}$, $w_{2
      1}$, $w_{2 2}$, $w_{2 3}$, $w_{2 4}$, $w_{2 5}$, $w_{2 6}$, $w_{2
      7}$, $w_{2 8}$, $w_{2 9}$, $w_{2 1}$, $w_{2 2}$, $w_{2 3}$, $w_{2
      4}$, $w_{2 5}$, $w_{2 6}$, $w_{2 7}$, $w_{2 8}$, $w_{2 1}$, $w_{2
      2}$, $w_{2 3}$, $w_{2 4}$, $w_{2 5}$, $w_{2 6}$, $w_{2 7}$, $w_{2
      1}$, $w_{2 2}$, $w_{2 3}$, $w_{2 4}$, $w_{2 5}$, $w_{2 6}$, $w_{2
      1}$, $w_{2 2}$, $w_{2 3}$, $w_{2 4}$, $w_{2 5}$, $w_{2 1}$, $w_{2
      2}$, $w_{2 3}$, $w_{2 4}$, $w_{2 1}$, $w_{2 2}$, $w_{2 3}$, $w_{2
      1}$, $w_{2 2}$, $w_{2 1}$, $w_{2 b}$, $w_{2 a}$, $w_{2 9}$, $w_{2
      8}$, $w_{2 7}$, $w_{2 6}$, $w_{2 5}$, $w_{2 4}$, $w_{2 3}$, $w_{2
      2}$, $w_{2 1}$, $w_{1 7}$, $w_{6 7}$, $w_{1 8}$, $w_{6 8}$, $w_{1
      9}$, $w_{6 9}$, $w_{1 a}$, $w_{6 a}$, $w_{1 b}$, $w_{6 b}$, $w_{1
      7}$, $w_{6 7}$, $w_{1 8}$, $w_{6 8}$, $w_{1 9}$, $w_{6 9}$, $w_{1
      a}$, $w_{6 a}$, $w_{1 7}$, $w_{6 7}$, $w_{1 8}$, $w_{6 8}$, $w_{1
      9}$, $w_{6 9}$, $w_{1 7}$, $w_{6 7}$, $w_{1 8}$, $w_{6 8}$, $w_{1
      7}$, $w_{6 7}$, $w_{1 b}$, $w_{6 b}$, $w_{1 a}$, $w_{6 a}$, $w_{1
      9}$, $w_{6 9}$, $w_{1 8}$, $w_{6 8}$, $w_{1 7}$, $w_{6 7}$, $w_{3
      7}$, $w_{5 7}$, $w_{3 8}$, $w_{5 8}$, $w_{3 9}$, $w_{5 9}$, $w_{3
      a}$, $w_{5 a}$, $w_{3 b}$, $w_{5 b}$, $w_{3 7}$, $w_{5 7}$, $w_{3
      8}$, $w_{5 8}$, $w_{3 9}$, $w_{5 9}$, $w_{3 a}$, $w_{5 a}$, $w_{3
      7}$, $w_{5 7}$, $w_{3 8}$, $w_{5 8}$, $w_{3 9}$, $w_{5 9}$, $w_{3
      7}$, $w_{5 7}$, $w_{3 8}$, $w_{5 8}$, $w_{3 7}$, $w_{5 7}$, $w_{3
      b}$, $w_{5 b}$, $w_{3 a}$, $w_{5 a}$, $w_{3 9}$, $w_{5 9}$, $w_{3
      8}$, $w_{5 8}$, $w_{3 7}$, $w_{5 7}$, $w_{4 7}$, $w_{4 8}$, $w_{4
      9}$, $w_{4 a}$, $w_{4 b}$, $w_{4 7}$, $w_{4 8}$, $w_{4 9}$, $w_{4
      a}$, $w_{4 7}$, $w_{4 8}$, $w_{4 9}$, $w_{4 7}$, $w_{4 8}$, $w_{4
      7}$, $w_{4 b}$, $w_{4 a}$, $w_{4 9}$, $w_{4 8}$, $w_{4 7}$, $w_{2
      7}$, $w_{2 8}$, $w_{2 9}$, $w_{2 a}$, $w_{2 b}$, $w_{2 7}$, $w_{2
      8}$, $w_{2 9}$, $w_{2 a}$, $w_{2 7}$, $w_{2 8}$, $w_{2 9}$, $w_{2
      7}$, $w_{2 8}$, $w_{2 7}$, $w_{2 b}$, $w_{2 a}$, $w_{2 9}$, $w_{2
      8}$, $w_{2 7}$, $w_{1 1}$, $w_{6 1}$, $w_{1 2}$, $w_{6 2}$, $w_{1
      3}$, $w_{6 3}$, $w_{1 4}$, $w_{6 4}$, $w_{1 5}$, $w_{6 5}$, $w_{1
      1}$, $w_{6 1}$, $w_{1 2}$, $w_{6 2}$, $w_{1 3}$, $w_{6 3}$, $w_{1
      4}$, $w_{6 4}$, $w_{1 1}$, $w_{6 1}$, $w_{1 2}$, $w_{6 2}$, $w_{1
      3}$, $w_{6 3}$, $w_{1 1}$, $w_{6 1}$, $w_{1 2}$, $w_{6 2}$, $w_{1
      1}$, $w_{6 1}$, $w_{1 5}$, $w_{6 5}$, $w_{1 4}$, $w_{6 4}$, $w_{1
      3}$, $w_{6 3}$, $w_{1 2}$, $w_{6 2}$, $w_{1 1}$, $w_{6 1}$, $w_{3
      1}$, $w_{5 1}$, $w_{3 2}$, $w_{5 2}$, $w_{3 3}$, $w_{5 3}$, $w_{3
      4}$, $w_{5 4}$, $w_{3 5}$, $w_{5 5}$, $w_{3 1}$, $w_{5 1}$, $w_{3
      2}$, $w_{5 2}$, $w_{3 3}$, $w_{5 3}$, $w_{3 4}$, $w_{5 4}$, $w_{3
      1}$, $w_{5 1}$, $w_{3 2}$, $w_{5 2}$, $w_{3 3}$, $w_{5 3}$, $w_{3
      1}$, $w_{5 1}$, $w_{3 2}$, $w_{5 2}$, $w_{3 1}$, $w_{5 1}$, $w_{3
      5}$, $w_{5 5}$, $w_{3 4}$, $w_{5 4}$, $w_{3 3}$, $w_{5 3}$, $w_{3
      2}$, $w_{5 2}$, $w_{3 1}$, $w_{5 1}$, $w_{4 1}$, $w_{4 2}$, $w_{4
      3}$, $w_{4 4}$, $w_{4 5}$, $w_{4 1}$, $w_{4 2}$, $w_{4 3}$, $w_{4
      4}$, $w_{4 1}$, $w_{4 2}$, $w_{4 3}$, $w_{4 1}$, $w_{4 2}$, $w_{4
      1}$, $w_{4 5}$, $w_{4 4}$, $w_{4 3}$, $w_{4 2}$, $w_{4 1}$, $w_{2
      1}$, $w_{2 2}$, $w_{2 3}$, $w_{2 4}$, $w_{2 5}$, $w_{2 1}$, $w_{2
      2}$, $w_{2 3}$, $w_{2 4}$, $w_{2 1}$, $w_{2 2}$, $w_{2 3}$, $w_{2
      1}$, $w_{2 2}$, $w_{2 1}$, $w_{2 5}$, $w_{2 4}$, $w_{2 3}$, $w_{2
      2}$, $w_{2 1}$ }

  $\conftoquiv(\alpha)$ gives
  \begin{equation*}
    \begin{gathered}
      \begin{aligned}
        A_1 &= h_{11} &\qquad A_2 &= h_{12} &\qquad A_3 &= h_{13} &\qquad A_4 &= h_{14} &\qquad A_5 &= h_{15} &\qquad A_6 &= h_{16} \\
        B_1 &= h_{21} & B_2 &= h_{22} & B_3 &= h_{23} & B_4 &= h_{24} & B_5 &= h_{25} & B_6 &= h_{26} \\
        C_1 &= h_{36} & C_2 &= h_{32} & C_3 &= h_{35} & C_4 &= h_{34} & C_5 &= h_{33} & C_6 &= h_{31} \\
      \end{aligned}
      \\
      \begin{aligned}
        v_{2 1} &= \genminor{w_{7}}{2}(u) \cdot h_{22} & v_{4 1} &= \genminor{w_{8}}{4}(u) \cdot \frac{h_{12} h_{24} }{h_{14}} & v_{3 1} &= \genminor{w_{9}}{3}(u) \cdot \frac{h_{12} h_{23} }{h_{15}} \\
        v_{5 1} &= \genminor{w_{10}}{5}(u) \cdot \frac{h_{12} h_{25} }{h_{13}} & v_{1 1} &= \genminor{w_{11}}{1}(u) \cdot \frac{h_{12} h_{21} }{h_{16}} & v_{6 1} &= \genminor{w_{12}}{6}(u) \cdot \frac{h_{12} h_{26} }{h_{11}} \\
        v_{2 2} &= \genminor{w_{13}}{2}(u) \cdot \frac{h_{14} h_{22} }{h_{12}} & v_{4 2} &= \genminor{w_{14}}{4}(u) \cdot h_{12} h_{24} & v_{3 2} &= \genminor{w_{15}}{3}(u) \cdot \frac{h_{12} h_{14} h_{23} }{h_{15}} \\
        v_{5 2} &= \genminor{w_{16}}{5}(u) \cdot \frac{h_{12} h_{14} h_{25} }{h_{13}} & v_{1 2} &= \genminor{w_{17}}{1}(u) \cdot \frac{h_{14} h_{21} }{h_{16}} & v_{6 2} &= \genminor{w_{18}}{6}(u) \cdot \frac{h_{14} h_{26} }{h_{11}} \\
        v_{2 3} &= \genminor{w_{19}}{2}(u) \cdot h_{22} & v_{4 3} &= \genminor{w_{20}}{4}(u) \cdot h_{12}^2 h_{24} & v_{3 3} &= \genminor{w_{21}}{3}(u) \cdot \frac{h_{12} h_{14} h_{23} }{h_{15}} \\
        v_{5 3} &= \genminor{w_{22}}{5}(u) \cdot \frac{h_{12} h_{14} h_{25} }{h_{13}} & v_{1 3} &= \genminor{w_{23}}{1}(u) \cdot \frac{h_{12} h_{21} }{h_{16}} & v_{6 3} &= \genminor{w_{24}}{6}(u) \cdot \frac{h_{12} h_{26} }{h_{11}} \\
        v_{2 4} &= \genminor{w_{25}}{2}(u) \cdot h_{14} h_{22} & v_{4 4} &= \genminor{w_{26}}{4}(u) \cdot h_{12} h_{14} h_{24} & v_{3 4} &= \genminor{w_{27}}{3}(u) \cdot \frac{h_{12} h_{14} h_{23} }{h_{15}} \\
        v_{5 4} &= \genminor{w_{28}}{5}(u) \cdot \frac{h_{12} h_{14} h_{25} }{h_{13}} & v_{1 4} &= \genminor{w_{29}}{1}(u) \cdot \frac{h_{14} h_{21} }{h_{16}} & v_{6 4} &= \genminor{w_{30}}{6}(u) \cdot \frac{h_{14} h_{26} }{h_{11}} \\
        v_{2 5} &= \genminor{w_{31}}{2}(u) \cdot \frac{h_{14} h_{22} }{h_{12}} & v_{4 5} &= \genminor{w_{32}}{4}(u) \cdot h_{12} h_{24} & v_{3 5} &= \genminor{w_{33}}{3}(u) \cdot \frac{h_{14} h_{23} }{h_{15}} \\
        v_{5 5} &= \genminor{w_{34}}{5}(u) \cdot \frac{h_{14} h_{25} }{h_{13}} & v_{1 5} &= \genminor{w_{35}}{1}(u) \cdot \frac{h_{15} h_{21} }{h_{16}} & v_{6 5} &= \genminor{w_{36}}{6}(u) \cdot \frac{h_{13} h_{26} }{h_{11}} \\
      \end{aligned}
    \end{gathered}
  \end{equation*}

  \subsection{\texorpdfstring{$E_7$}{E\_7}}

  Based on the Dynkin diagram \raisebox{-0.7ex}{\includestandalone[mode=image|tex,height=3ex]{fig/dynkin-E7}}, the following quiver is
  of well-rooted Dynkin type for $E_7$. It admits an induced Coxeter
  element $c = \set{1, 2, 3, 4, 5, 6, 7}$, with partitions \[ T_0 =
    \set{1}, \quad T_1 = \set{2}, \quad T_2 = \set{3}, \quad T_3 =
    \set{4, 5}, \quad T_4 = \set{6}, \quad T_5 = \set{7}. \]

  {\centering \includestandalone[mode=image|tex]{fig/dynkin-type-E7}

  } $Q_{E_7}$ is given in \cref{fig:QE7}. The mutations are too long to
  be reasonably presented.

  \begin{figure}
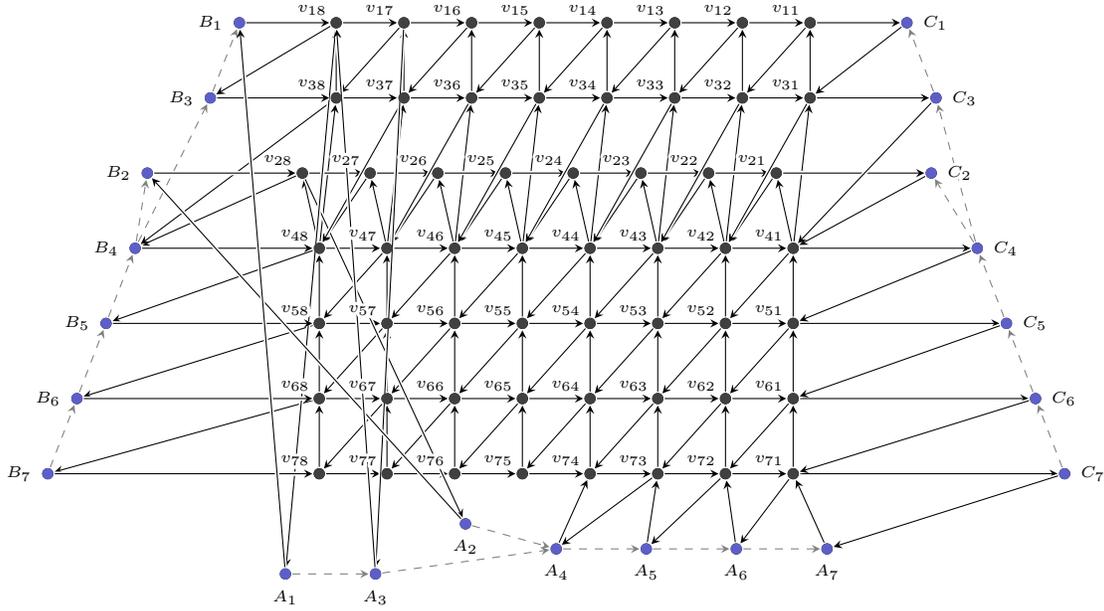

    \centering \includestandalone[mode=image|tex]{fig/QE7}

    \caption{$Q$ for $E_7$.
      \label{fig:QE7}
    }
  \end{figure}

  \subsection{\texorpdfstring{$E_8$}{E\_8}}

  Based on the Dynkin diagram \raisebox{-0.7ex}{\includestandalone[mode=image|tex,height=3ex]{fig/dynkin-E8}}, the following quiver is
  of well-rooted Dynkin type for $E_8$. It admits an induced Coxeter
  element $c = \set{1, 2, 3, 4, 5, 6, 7, 8}$, with partitions \[ T_0 =
    \set{1}, T_1 = \set{2}, T_2 = \set{3}, T_3 = \set{4}, T_4 = \set{5},
    T_5 = \set{6, 7}, T_6 = \set{8}. \]

  {\centering \includestandalone[mode=image|tex]{fig/dynkin-type-E8}

  } $Q_{E_8}$ is given in \cref{fig:QE8}. The mutations are too long to
  be reasonably presented.

  \begin{figure}
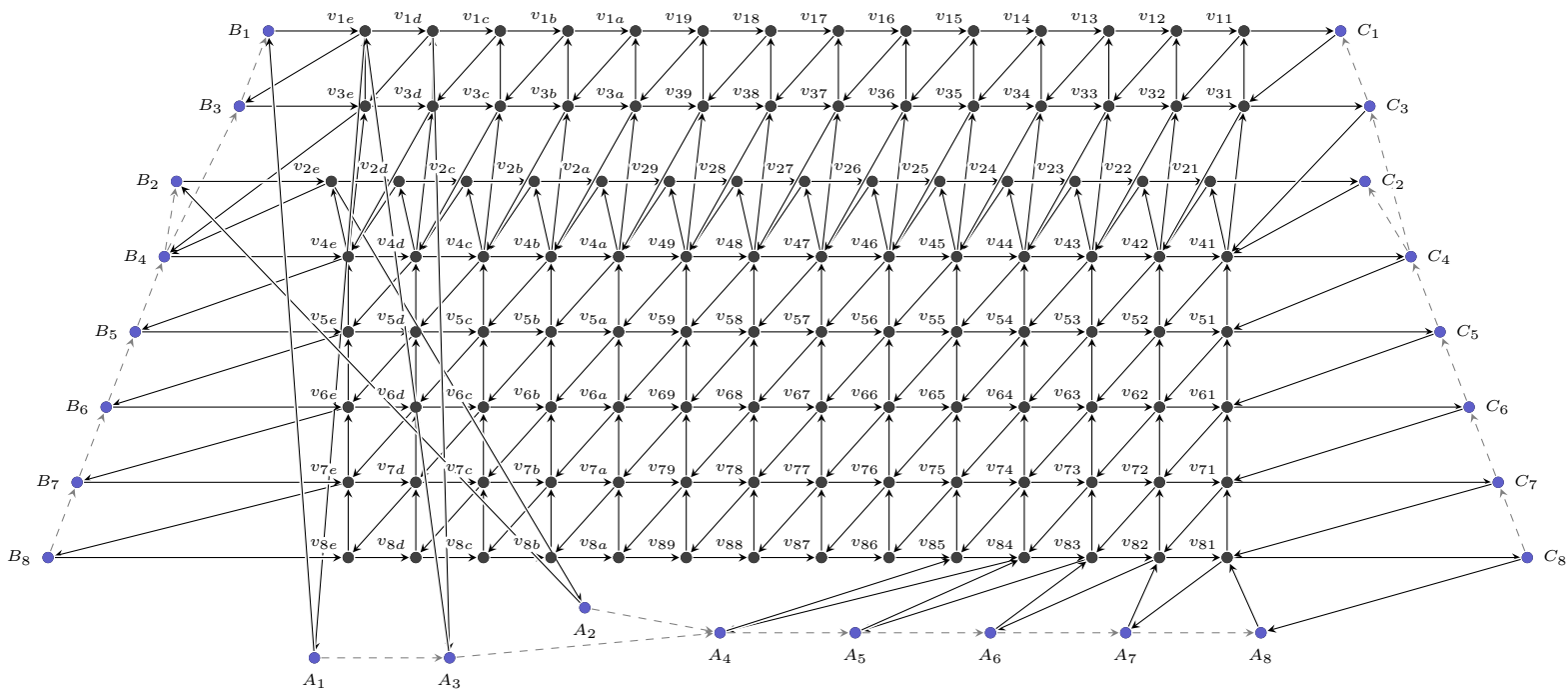

    \centering \includestandalone[angle=270,mode=image|tex]{fig/QE8}

    \caption{$Q$ for $E_8$.
      \label{fig:QE8}
    }
  \end{figure}

  \subsection{\texorpdfstring{$D_2 = A_1 \times A_1$}{D\_2 = A\_1 x
      A\_1}}

  Using the Dynkin diagram \raisebox{-0.0ex}{\includestandalone[mode=image|tex,height=1.1ex]{fig/dynkin-D2}}, we use the following quiver.

  {\centering \includestandalone[mode=image|tex]{fig/dynkin-type-D2}

  } The quiver $Q_G$ is shown in \cref{fig:QD2}.

  \begin{figure}
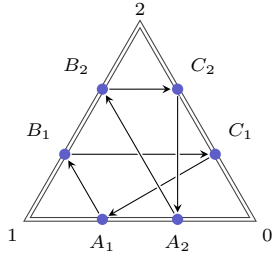

    \centering \includestandalone[mode=image|tex]{fig/QD2}

    \caption{$Q$ for $D_2 = A_1 \times A_1$.
      \label{fig:QD2}
    }
  \end{figure}

  \inlmutseq{\murot}{ }

  \inlmutseq{\muflip}{ $w_{1 1}$, $w_{2 1}$ }

  And $\conftoquiv$ is given the edge coordinate assignment only. Since
  $\sigma_{D_2}$ is trivial, this is direct.

  \begin{equation*}
    \begin{gathered}
      \begin{aligned}
        A_1 &= h_{11} &\quad A_2 &= h_{12} \\
        B_1 &= h_{21} &\quad B_2 &= h_{22} \\
        C_1 &= h_{31} &\quad C_2 &= h_{32}
      \end{aligned}
    \end{gathered}
  \end{equation*}

  \subsection{\texorpdfstring{$A_3 \times C_2$}{A\_3 x C\_2}}

  To demonstrate \cref{lem:fgcs-exists-for-products} for a less trivial
  type, consider $G = G_1 \oplus G_2$, with $G_1$ of type $A_3$ and
  $G_2$ of type $C_2$. The Dynkin diagram for $G$ is
  \raisebox{-0.7ex}{\includestandalone[mode=image|tex,height=3ex]{fig/dynkin-A3-C2}}. The quiver $Q_G$ is the disjoint union of
  $Q_{A_3}$ and $Q_{C_2}$ as shown in \cref{fig:QA3-C2}. We label the
  simple roots as shown in the Dynkin type quiver.

  {\centering \includestandalone[mode=image|tex]{fig/dynkin-type-A3-C2}

  }

  \begin{figure}
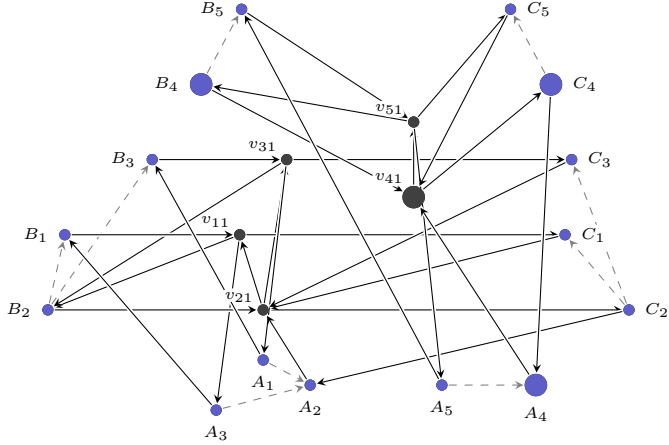

    \centering \includestandalone[mode=image|tex]{fig/QA3-C2}

    \caption{$Q$ for $A_3 \times C_2$.
      \label{fig:QA3-C2}
    }
  \end{figure}

  \inlmutseq{\murot}{ $v_{2 1}$, $v_{1 1}$, $v_{3 1}$, $v_{2 1}$, $v_{1
      1}$, $v_{3 1}$, $v_{2 1}$, $v_{1 1}$, $v_{3 1}$, $v_{2 1}$, $v_{1
      1}$, $v_{3 1}$, \qquad $v_{4 1}$, $v_{5 1}$ }

  \inlmutseq{\muflip}{ $w_{1 3}$, $w_{3 3}$, $w_{2 3}$, $w_{1 1}$, $w_{3
      1}$, $w_{2 1}$, $w_{2 3}$, $w_{1 3}$, $w_{3 3}$, $w_{2 2}$, $w_{1
      2}$, $w_{3 2}$, $w_{2 1}$, $w_{1 1}$, $w_{3 1}$, $w_{2 3}$, $w_{1
      3}$, $w_{3 3}$, $w_{2 2}$, $w_{1 2}$, $w_{3 2}$, $w_{2 3}$, $w_{1
      3}$, $w_{3 3}$, $w_{1 1}$, $w_{3 1}$, $w_{1 2}$, $w_{3 2}$, $w_{1
      3}$, $w_{3 3}$, $w_{1 1}$, $w_{3 1}$, $w_{1 2}$, $w_{3 2}$, $w_{1
      1}$, $w_{3 1}$, $w_{1 3}$, $w_{3 3}$, $w_{1 2}$, $w_{3 2}$, $w_{1
      1}$, $w_{3 1}$, $w_{2 1}$, $w_{2 2}$, $w_{2 3}$, $w_{2 1}$, $w_{2
      2}$, $w_{2 1}$, $w_{2 3}$, $w_{2 2}$, $w_{2 1}$, \qquad $w_{5 3}$,
    $w_{4 3}$, $w_{5 1}$, $w_{4 1}$, $w_{4 3}$, $w_{5 3}$, $w_{4 2}$,
    $w_{5 2}$, $w_{4 1}$, $w_{5 1}$, $w_{4 3}$, $w_{5 3}$, $w_{4 2}$,
    $w_{5 2}$, $w_{4 3}$, $w_{5 3}$, $w_{5 1}$, $w_{5 2}$, $w_{5 3}$,
    $w_{5 1}$, $w_{5 2}$, $w_{5 1}$, $w_{5 3}$, $w_{5 2}$, $w_{5 1}$,
    $w_{4 1}$, $w_{4 2}$, $w_{4 3}$, $w_{4 1}$, $w_{4 2}$, $w_{4 1}$,
    $w_{4 3}$, $w_{4 2}$, $w_{4 1}$ }

  To construct $\conftoquiv$, as described in
  \cref{lem:fgcs-exists-for-products}, the coordinates on $u$ are
  exactly divided between two elements of $N_{-}$ corresponding to $A_3$
  and $C_2$. Let $(w_k)_{1}$ be the word $w_k$ for $A_3$, and
  $(w_k)_{2}$ be the word $w_k$ for $C_2$ (with the indices of each
  component $s_{\alpha_i}$ increased by $3$). Then appropriate $i$ and
  $j$ in $\genminor{(w_k)_i}{j}(u)$ will allow extracting coordinates of
  $u$ in ways corresponding to the factors $G_i$.
  \begin{equation*}
    \begin{gathered}
      \begin{aligned}
        A_1 &= h_{11} &\quad A_2 &= h_{12} &\quad A_3 &= h_{13} &\qquad A_4 &= h_{14} &\quad A_5 &= h_{15} \\
        B_1 &= h_{21} &\quad B_2 &= h_{22} &\quad B_3 &= h_{23} &\qquad B_4 &= h_{24} &\quad B_5 &= h_{25} \\
        C_1 &= h_{33} &\quad C_2 &= h_{32} &\quad C_3 &= h_{31} &\qquad C_4 &= h_{34} &\quad C_5 &= h_{35} \\
      \end{aligned}
      \\
      \begin{aligned}
        v_{2 1} &= \genminor{(w_{4})_{1}}{2}(u) \cdot h_{22} & \quad
        v_{1 1} &= \genminor{(w_{5})_{1}}{1}(u) \cdot \frac{h_{12} h_{21} }{h_{13}} & \quad
        v_{3 1} &= \genminor{(w_{6})_{1}}{3}(u) \cdot \frac{h_{12} h_{23} }{h_{11}}
      \end{aligned}
      \\
      \begin{aligned}
        v_{4 1} &= \genminor{(w_{3})_{2}}{4}(u) \cdot h_{24} &\quad v_{5 1} &= \genminor{(w_{4})_{2}}{5}(u) \cdot \frac{h_{14}^2 h_{25} }{h_{15}}
      \end{aligned}
    \end{gathered}
  \end{equation*}

  \printbibliography 
\end{document}

%% file: fig/dynkin-D2.tex
  \begin{tikzpicture}
    \tikzstyle{vn} = [color=darkgray          , circle, outer sep=0.5pt, fill, inner sep=1.5pt]

    \node[vn] (1) at ( 0.00, 0.00) {};
    \node[vn] (2) at ( 0.25, 0.00) {};
  \end{tikzpicture}